\documentclass{article}
\usepackage[top=1in, bottom=1in, left=1in, right=1in]{geometry}
\usepackage{paper}
\usepackage[utf8]{inputenc} 
\usepackage[T1]{fontenc}    
\usepackage{todonotes}

\newcommand{\sNetName}{\textit{BNet}}
\newcommand{\NetName}{\textit{Butterfly-net}}
\newcommand{\sINetName}{\textit{IBNet}}
\newcommand{\INetName}{\textit{Inflated-Butterfly-net}}

\newcommand{\Kfun}[1]{e^{-2 \pi \imath #1}}
\newcommand{\negKfun}[1]{e^{2 \pi \imath #1}}
\newcommand{\LebConst}{\frac{2}{\pi}\ln r + 1}

\newcommand{\coef}[2]{\lambda^{#1}_{#2}}
\newcommand{\exeq}{\overset{\diamond}{=}}

\title{Butterfly-Net: Optimal Function Representation
Based on Convolutional Neural Networks}

\author{
    Yingzhou Li$^{\,\sharp}$,
    Xiuyuan Cheng$^{\,\sharp}$,
    Jianfeng Lu$^{\,\sharp\,\dagger}$
    \vspace{0.1in}\\
    $\sharp$ Department of Mathematics, Duke University\\
    $\dagger$ Department of Chemistry and Department of Physics,
    Duke University\\
}

\begin{document}

\maketitle

\begin{abstract}
    Deep networks, especially convolutional neural networks (CNNs),
    have been successfully applied in various areas of machine
    learning as well as to challenging problems in other scientific
    and engineering fields. This paper introduces \NetName{}, a
    low-complexity CNN with structured and sparse cross-channel
    connections, together with a {\it Butterfly} initialization
    strategy for a family of networks. Theoretical analysis of the
    approximation power of \NetName{} to the Fourier representation
    of input data shows that the error decays exponentially as the
    depth increases. Combining \NetName{} with a fully connected
    neural network, a large class of problems are proved to be well
    approximated with network complexity depending on the effective
    frequency bandwidth instead of the input dimension. Regular
    CNN is covered as a special case in our analysis. Numerical
    experiments validate the analytical results on the approximation
    of Fourier kernels and energy functionals of Poisson's
    equations. Moreover, all experiments support that training
    from {\it Butterfly} initialization outperforms training from
    random initialization. Also, adding the remaining cross-channel
    connections, although significantly increase the parameter number,
    does not much improve the post-training accuracy and is more
    sensitive to data distribution.
\end{abstract}

\section{Introduction}

Deep neural network is a central tool in machine learning and
data analysis nowadays~\cite{bengio2015deep}. In particular,
convolutional neural network (CNN) has been proved to be a
powerful tool in image recognition and representation. Deep
learning has also emerged to be successfully applied in solving
PDEs~\cite{berg2017unified, khoo2018solving, long2017pde} and physics
problems~\cite{behler2007generalized, weinan2017deep, Li2019b,
schneider2017stochastic, zhang2018deepcg}, showing the potential of
becoming a tool of great use for computational mathematics and physics
as well. Given the wide application of PDEs and wavelet based methods
in image and signal processing~\cite{cai2012image, chan2005image,
mallat2008wavelet}, an understanding of CNN's ability to approximate
differential and integral operators will lead to an explanation of
CNN's success in these fields, as well as possible improved network
architectures.

The remarkable performance of deep neural networks across various
fields relies on their ability to accurately represent functions of
high-dimensional input data. Approximation analysis has been a central
topic to the understanding of the neural networks. The classical
theory developed in 80's and early 90's~\cite{barron1993universal,
cybenko1989approximation, hornik1989multilayer} approximates a target
function by a linear combination of sigmoids, which is equivalent to a
fully connected neural network with one hidden layer. While universal
approximation theorems were established for such shallow networks,
the research interest in neural networks only revived in recent
years after observing the successful applications of deep neural
networks, particular the superior performance of CNNs in image and
signal processing.

Motivated by the empirical success, the approximation advantage
of deep neural networks over shallow ones has been theoretically
analyzed in several places. However, most results assume stacked
fully connected layers and do not apply to CNNs which have specific
geometrical constraints: (1) the convolutional scheme, namely
local-supported filters with weight sharing, and (2) the hierarchical
multi-scale architecture. The approximation power of deep networks
with hierarchical geometrically-constrained structure has been
studied recently~\cite{cohen2016expressive, mhaskar2016learning,
mhaskar2016deep}, yet the network architecture differ from the regular
CNN. The approximation theory of CNN has been studied in~\cite{Bao2019,
Zhou2018}. We review the related literature in more detail below.

\begin{figure}[t]
    \centering
    \includegraphics[width=0.9\linewidth]{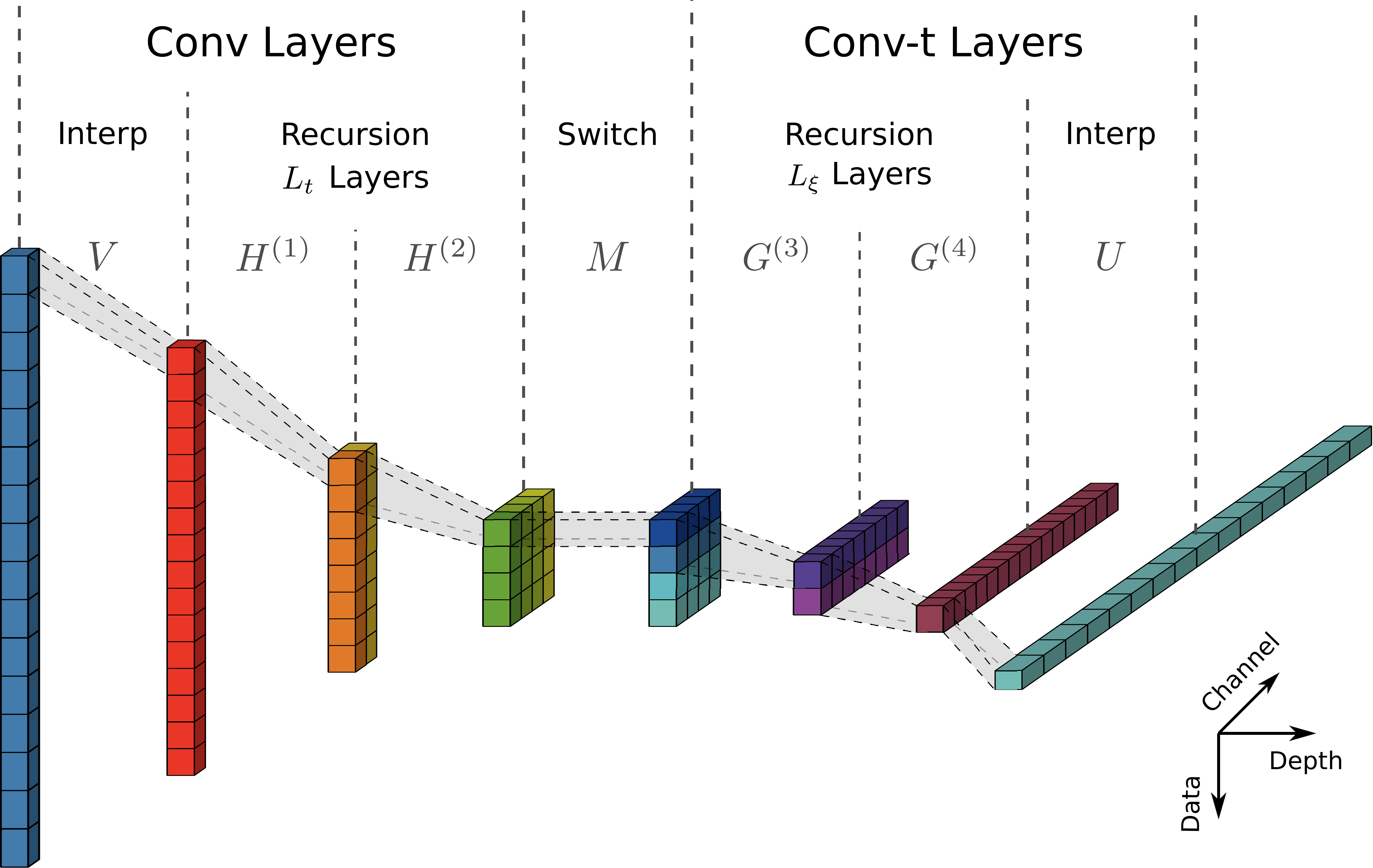}
    \caption{Illustration of a \NetName{}: The network consists of 3
    modules, (1) Conv layers, which has one Interpolation later and
    $L_t$ Recursion layers, (2) A Switch layer, and (3) Conv-t layers,
    which has $L_\xi$ Recursion layers and one Interpolation later.
    The conv and conv-t layers have convolutional and transpose
    convolutional structure respectively, and the Switch layer only
    involves local operations.  Detailed structures of these layers
    are given in Section~\ref{sec:bnet}.  This figure shows the case
    when $L_t = L_\xi = 2$, and the notations $V, H^{(\ell)}, M,
    G^{(\ell)}, U$ denote the matrix representations of the operations
    in the scheme, c.f. Section~\ref{sec:bnetMatrix}.  The name {\it
    ``butterfly''} can be understood from the shape of the network.}
    \label{fig:butterflyNet}
\end{figure}

This paper proposes a specific architecture under the CNN framework
based on the {\it Butterfly} scheme originally developed for the fast
computation of special function transforms~\cite{Michielssen1996,
ONeil2010, Ying2009} and Fourier integral operators~\cite{Candes2007,
Candes2009, Li2017, Li2015a, Li2015b, Li2018}. {\it Butterfly} scheme
provides a hierarchical structure with locally low-rank interpolation
of kernel functions and can be applied to solve many PDE related
problems. In terms of computational complexity, the scheme is near
optimal for Fourier kernels and Fourier integral operators. The
proposed \NetName{} explicitly adopts the hierarchical structure in
{\it Butterfly} scheme as the stacked convolutional layers. If the
parameters are hard-coded as that in the {\it Butterfly} scheme
({\it Butterfly} initialization), then \NetName{} collectively
computes the Fourier coefficients of the input signal with guaranteed
numerical accuracy. Unlike regular CNN which has dense cross-channel
connections, the channels in the \NetName{} have clear correspondences
to the frequency bands, namely the position in the spectral
representation of the signal, and meanwhile, the cross-channel weights
are sparsely connected. In this paper, we also study \NetName{} with
dense cross-channel connections, which is named \INetName{}. Regular
CNN is a special \INetName{}~\cite{Xu2019}. Comparing \NetName{}
and \INetName{}, \NetName{} is much lighter: the model complexity
(in terms of parameter number) is $O(K \log N)$ and computational
complexity is $O(N\log N)$, where $N$ is the length of the discrete
input signal and $K$ is the frequency bandwidth.

The approximation error of \NetName{} in representing Fourier kernels
is proved to exponentially decay as the network depth increases, which
is numerically validated in Section~\ref{sec:experiment}. Due to the
efficient approximation of Fourier kernels, \NetName{} thus possesses
all approximation properties of the Fourier representation of input
signals, which is particularly useful for solving PDEs and (local)
Fourier-based algorithms in image and signal processing. Theoretically,
the approximation guarantee for \NetName{} to represent Fourier
kernels leads to an approximation result of a family of sparsified
CNNs and regular CNNs. The informal statement of our main result on
neural network function approximation is as follows:

\begin{theorem}[Informal version]
    Consider using neural network to approximate a function
    $f(\vec{x})$, $\vec{x} \in \calX \subset \bbR^n$, where $f$ can
    be approximated by a band-limited function of $\vec{x}$ with
    bandwidth $K$ (this may be due to the frequency decay of $f$,
    or $\calX$, or both). Then there exists a family of \NetName{}s
    whose numbers of parameters are all bounded by
    \begin{equation}
            n_b \lesssim K \log N,
    \end{equation}
    such that using anyone of them to extract a deep feature of length
    $K$ from the input $x$ can reduce the effective dimension of the
    input data from $N$ to $K$ for neural network function approximation.
\end{theorem}

The dimension reduction is reflected by the needed network complexity
in approximating $f(\vec{x})$ by a fully-connected network, namely
a reduction from a factor of $\varepsilon^{-N/s}$ (for a fully
connected net to approximate $f(\vec{x})$ directly) to that of
$\varepsilon^{-K/s}$ (for a fully connected net to approximate $g$,
\NetName{} to approximate $\calB$, and $f(\vec{x}) \approx g(\calB
\vec{x})$), where $s$ is the regularity level and $\varepsilon$
is the uniform approximation error of $f$. The precise statement is
given in Theorem~\ref{thm:approx_bnet}.

Moreover, in the above statement, the family of \NetName{}s can
be replaced by the corresponding family of \INetName{} with $n_b
\lesssim K^2 \log N$ parameters and the same dimension reduction
argument still holds.  In particular, one member in the family of
\INetName{}s is a regular CNN. Hence, our approximation analysis
covers regular CNNs as a special case.

\subsection{Contributions}

Our contributions can be summarized as follows.

\begin{enumerate}[1)]

    \item We propose a family of novel neural network architectures,
    named \NetName{}s, which are composed of convolutional and
    transpose convolutional layers with sparse cross-channel
    connections, plus a locally connected  switch layer in between.
    An associated {\it Butterfly} initialization strategy is proposed
    for \NetName{}s to approximate Fourier kernels.  The \NetName{}
    architecture can be inflated via replacing the sparse cross-channel
    connections by dense connections, and this contains regular CNN
    as a special case.

    \item The approximation error of \NetName{}s representing the
    Fourier kernels is theoretically proved to be exponentially
    decay as the network depth increases. Concatenating \NetName{}
    (or its inflated version) with a fully-connected layer, we
    provide an approximation analysis for a wide class of functions
    with frequency decay property and the approximation complexity
    depends on the effective dimension $K$ instead of the input data
    dimension $N$. The regular CNN is covered as a special case.

    \item Numerically, we apply \NetName{} and its inflated version
    to a wide range of datasets. The successful trainings on all
    datasets support our approximation analysis. Further, we find
    that training from {\it Butterfly} initialization in all cases
    outperforms training from random initialization, especially
    when the output function only depends on a few frequencies among
    the wide frequency band of the input data.  \NetName{} achieves
    similar post-training accuracy as its inflated version with far
    less number of parameters.  \NetName{} also admits better transfer
    learning capability when the distribution of the testing data is
    shifted away from the training data.
    
\end{enumerate}

\subsection{Related Works}

Before we explain more details in the rest of the paper, we review
some related works.

{\bf Fast algorithm inspired neural network structures.}
Fast algorithm has inspired several neural network structures
recently. Based on $\mathcal{H}$-matrix and $\mathcal{H}^2$-matrix
structure, Fan and his coauthors proposed two multiscale neural
networks~\cite{Fan2018, Fan2019a}, which are more suitable in training
smooth linear or nonlinear operators due to the multiscale nature of
$\mathcal{H}$-matrices. In addition to that, nonstandard wavelet form
inspired the design of BCR-Net~\cite{Fan2019}, which is applied to
address the inverse of elliptic operator and nonlinear homogenization
problem and recently been embedded in a neural network for solving
electrical impedance tomography~\cite{Fan2019b} and pseudo-differential
operator~\cite{Feliu-Faba2019}. Multigrid method also inspired
MgNet~\cite{He2019}. In addition to the above approximation of
relatively smooth operators, {\it Butterfly} scheme inspired the
design of SwitchNet~\cite{Khoo2018}, which is a non-convolutional
three layer neural network and addresses scattering problems.

{\bf Classical approximation results of neural networks.} Universal
approximation theorems for fully-connected neural networks with one
hidden layer were established in~\cite{cybenko1989approximation,
hornik1989multilayer} showing that such networks can approximate
a target function with arbitrary accuracy if the hidden layer is
allowed to be wide enough. In theory, the family of target functions
can include all measurable functions~\cite{hornik1989multilayer}, when
exponentially many hidden neurons are used. \citet{gallant1988there}
proposed ``Fourier network'', proving universal approximation to
squared-integrable functions by firstly constructing a Fourier
series approximation of the target function in a hard-coded
way. These theorems are firstly proved for one-dimensional input,
and when generalizing to the multivariate case the complexity grows
exponentially.

Using the Fourier representation of the target function supported on
a sphere in $\bbR^d$, \citet{barron1993universal} showed that the
mean squared error of the approximation, integrated with arbitrary
data distribution on the sphere, decays as $O(n^{-1})$ when $n$
hidden nodes are used in the single hidden layer. The results for
shallow networks are limited, and the approximation power of depth
in neural networks has been advocated in several recent works,
see below.  Besides, while the connection to Fourier analysis was
leveraged, at least in~\cite{barron1993universal, gallant1988there},
it is different from the hierarchical function representation scheme
as what we consider here.

{\bf Approximation power of deep neural networks.} The expressive
power of deep neural networks has drawn many research interests
in recent years. The approximation power of multi-layer restricted
Boltzmann machines (RBM) was studied in~\cite{le2008representational},
which showed that RBMs are universal approximators of discrete
distributions and more hidden layers improves the approximation
power. Relating to the classical approximation results in harmonic
analysis, \citet{bolcskei2017optimal} derived lower bounds for the
uniform approximation of square-integrable functions, and proved the
asymptotic optimality of the sparsely connected deep neural networks
as a universal approximator. However, the network complexity also
grows exponentially when the input dimension increases.

The approximation advantage of deep architecture over shallow ones
has been studied in several works. \citet{delalleau2011shallow}
identified a deep sum-product network which can only be approximated
by an exponentially larger number of shallow ones. The exponential
growth of linear regions as the number of layers increases was
studied in \cite{montufar2014number, telgarsky2016benefits}.
\citet{eldan2016power} constructed a concrete target function which
distinguishes three and two-layer networks. \citet{liang2016deep}
showed that shallow networks require exponentially more neurons than
deep networks to obtain a given approximation error for a large
class of functions. The advantage of deep ReLU networks over the
standard single-layer ones was analyzed in \cite{yarotsky2017error}
in the context of approximation in Sobolev spaces. \citet{Lu2020a}
shows the advantage of deep ReLU networks in approximating smooth
(band-limited) functions. The above works address deep networks with
fully-connected layers, instead of having geometrically-constrained
constructions like CNNs.

Deep neural networks with such geometric constraints are relatively
less analyzed.  The approximation power of a hierarchical
binary tree network was studied in \cite{mhaskar2016learning,
mhaskar2016deep} which supports the potential advantage of deep
CNNs. \citet{cohen2016expressive} used convolutional arithmetic
circuits to show the equivalence between a deep network and a
hierarchical Tucker decomposition of tensors, and proved the
advantage of depth in function approximation. The networks being
studied differ from the regular CNNs widely used in the typical real
world applications. Recently, \citet{Zhou2018} proposed the universal
approximation theory of deep CNN with an estimation on the number
of free parameters. \citet{Bao2019} shows the approximation power
of CNNs over deep neural networks without constraints on a class of
functions. Comparing to~\cite{Zhou2018}, our analysis covers a wider
range of networks and the approximation complexity is much lower on
a restricted function class. Comparing to~\cite{Bao2019}, different
function classes are discussed and we both show the advantage of CNN.

\subsection{Organization}

The rest of this paper is organized as follows. Building
on top of the traditional {\it Butterfly} literature,
Section~\ref{sec:preliminaries} first shows the low-rank property of
Fourier kernel and illustrates the {\it Butterfly} algorithm tailored
for Fourier kernel. Section~\ref{sec:bnet} proposes interpolative
convolutional layer as building blocks for \NetName{} followed by
the \NetName{} architecture with {\it Butterfly} initialization
and its matrix representation. Section~\ref{sec:analysis} analyzes
the approximation power of \NetName{} on Fourier kernels and its
extension to general functionals. Numerical results of \NetName{} and
\INetName{} are presented in Section~\ref{sec:experiment} including
the performance comparison of network architectures, initializations,
and datasets. Finally, in Section~\ref{sec:conclusion}, we conclude
the paper together with discussion on future directions.

\section{Preliminaries}
\label{sec:preliminaries}

This section paves the path to \NetName{}. We first derives a
low-rank interpolation of the Fourier kernel, which is crucial to
the efficiency of {\it Butterfly} scheme and \NetName{}. Then in
Section~\ref{sec:forwardalgo}, we reviews the {\it Butterfly} algorithm
for Fourier kernel. Readers, who are familiar with butterfly scheme,
should be safe to skip it.

\subsection{Low-rank Approximation of Fourier Kernel}
\label{sec:lowrank}

Fourier kernel throughout this paper is defined as,
\begin{equation} \label{eq:fourierkernel}
    \calK(\xi,t) = \Kfun{\xi \cdot t}, \qquad \xi \in [K_0, K_0+K),
    \quad t \in [0,1),
\end{equation}
where $[K_0, K_0+K)$ denotes the frequency window of interests, $K_0$
denotes the starting frequency, and $K$ denotes the frequency window
width. It is well-known that the discrete Fourier transform (DFT)
matrix, i.e., uniform discretization of \eqref{eq:fourierkernel}
with proper scaling, has orthonormal rows and columns. Hence, the
DFT matrix is a unitary matrix of full rank and the Fourier kernel
is also full rank. Theorem~\ref{thm:low-rank} show when the Fourier
kernel is restricted to certain pairs of subdomains of $[0,1)$ and
$[K_0, K_0+K)$, it has low-rank property.

We first give a brief introduction of the Chebyshev interpolation
with $r$ points. The Chebyshev grid of order $r$ on $[-\frac{1}{2},
\frac{1}{2}]$ is defined as,
\begin{equation}
    \left\{ z_i = \frac{1}{2} \cos \left( \frac{ (i-1) \pi}{ r } \right)
    \right\}_{i=1}^r.
\end{equation}
The Chebyshev interpolation of a function $f(x)$ on $[-\frac{1}{2},
\frac{1}{2}]$ is defined as,
\begin{equation}
    \Pi_r f(x) = \sum_{k=1}^r f(z_k) \calL_k(x),
\end{equation}
where $\calL_k(x)$ is the Lagrange polynomial as,
\begin{equation} \label{eq:LagrangePoly}
    \calL_{k}(x) = \prod_{p \neq k}
    \frac{x-z_p}{z_{k}-z_p}.
\end{equation}

Several earlier works~\cite{Candes2007, Candes2009, Li2015b} proved the
Chebyshev interpolation representation for Fourier integral operators,
which are generalized Fourier kernel. Theorem~\ref{thm:low-rank}
is a special case of these earlier work but with more precise and
explicit estimation on the prefactor.

\begin{theorem} \label{thm:low-rank}
    Let $r$ be the number of Chebyshev points, $B \subset [0,1)$
    and $A \subset [K_0, K_0+K)$ denote a domain pair such that
    $w(A) w(B) \leq \frac{r}{\pi e}$, where $w(\cdot)$ is the domain
    length function.  Then there exists low-rank representations of
    the Fourier kernel restricted to the domain pair,
    \begin{equation} \label{eq:low-rank1}
        \sup_{\xi \in A, t \in B} \abs{\,\Kfun{\xi \cdot t} - \sum_{k
        = 1}^{r}\Kfun{\xi \cdot t_k} \Kfun{\xi_0 \cdot \left( t -
        t_k \right)} \calL_k(t)\,} \leq \left( 2 + \frac{2}{\pi}
        \ln r \right) \left( \frac{ \pi e w(A) w(B) }{2r} \right)^r,
    \end{equation}
    and 
    \begin{equation} \label{eq:low-rank2}
        \sup_{\xi \in A, t\in B} \abs{\,\Kfun{\xi \cdot t} - \sum_{k =
        1}^{r} \Kfun{\left( \xi - \xi_k \right) \cdot t_0} \calL_k(\xi)
        \Kfun{\xi_k \cdot t}\,} \leq \left( 2 + \frac{2}{\pi} \ln r
        \right) \left( \frac{ \pi e w(A) w(B) }{2r} \right)^r,
    \end{equation}
    where $\xi_0$ and $t_0$ are the centers of $A$ and $B$, and $\xi_k$
    and $t_k$ are Chebyshev points on $A$ and $B$.
\end{theorem}

The proof of Theorem~\ref{thm:low-rank} is deferred to 
Appendix~\ref{app:low-rank}. This property is also known as
complementary low-rank property~\cite{Li2015a, Li2015b}.

\subsection{Butterfly Algorithm for Fourier Kernel}
\label{sec:forwardalgo}

This section briefly describes the {\it Butterfly} algorithm tailored
for Fourier kernel based on Theorem~\ref{thm:low-rank}. Given
a function $x(t)$ discretized on a uniform grid, $\{t_q =
\frac{q-1}{N}\}_{q=1}^N$, the goal is to compute the discrete Fourier
transform $\{\hat{x}(\xi_p), \xi_p = K_0, \dots, K_0+K-1 \}$ defined by
\begin{equation}
    \hat{x}(\xi_p) = \sum_{t_q} \Kfun{\xi_p t_q} x(t_q), \quad \xi_p =
    K_0, \dots, K_0+K-1.
\end{equation}
Without loss of generality, we assume $N \geq K$ throughout this paper.

{\bf Hierarchical domain partition.} In order to benefit from
Theorem~\ref{thm:low-rank}, we first define the $L$ layer hierarchical
partition of $[0,1)$ for $L \leq \log N$.~\footnote{Without further
explanation, $\log$ denotes logarithm base $2$ and $\ln$ denotes
natural logarithm base $e$.} Extension to $L > \log N$ is possible but
not common in butterfly scheme. Let $B^{0}_0 = [0,1)$ be the domain
on layer $0$. On layer $1$, the domain $B^0_0$ is evenly partitioned
into $B^1_0 =[0,\frac{1}{2})$ and $B^1_1=[\frac{1}{2},1)$. We conduct
the partition recursively, i.e., $B^{\ell-1}_j$ is evenly partitioned
into $B^{\ell}_{2j}$ and $B^{\ell}_{2j+1}$. In the end, the partition
on layer $\ell$ is $\{B^{\ell}_j, j = 0, \dots, 2^{\ell}-1\}$ and
each $B^{\ell}_j = [\frac{j}{2^\ell}, \frac{j+1}{2^\ell})$ is of
length $2^{-\ell}$.

Before partitioning the frequency domain, we introduce two more
notations, $L_t$ and $L_\xi$, which split the $L$ layers into two
group, i.e.,
\begin{equation}
    L = L_t + L_\xi.
\end{equation}
And $L_\xi$ satisfies the constraint $0 \leq L_\xi \leq \log K$.

The partitioning of the frequency domain is described as follows. For
layers $\ell \leq L_{\min} = \min(L_t, \log K - L_\xi)$, we partition
the frequency domain in the same way as for the time domain. Hence the
partition is $\{A^\ell_i, i=0, \dots, 2^\ell-1\}$ and each $A^\ell_i =
[K_0 + \frac{i}{2^\ell}K, K_0 + \frac{i+1}{2^\ell} K)$ is of length
$K\cdot 2^{-\ell}$. For layers $L_{\min} < \ell \leq L_t$, the
partition remains the same as that on layer $L_{\min}$. The partition
is $\{A^\ell_i, i=0, \dots, 2^{L_{\min}}-1\}$ and each $A^\ell_i =
[K_0 + \frac{i}{2^{L_{\min}}}K, K_0 + \frac{i+1}{2^{L_{\min}}}
K)$ is of length $K\cdot 2^{-L_{\min}}$. For the rest layers $L_t <
\ell \leq L$, the hierarchical bi-partition is applied again starting
from domains on layer $L_{\min}$. The partition is $\{A^\ell_i, i=0,
\dots, 2^{\ell - L_t + L_{\min}}-1\}$ and each $A^\ell_i = [K_0 +
\frac{i}{2^{\ell - L_t + L_{\min}}}K, K_0 + \frac{i+1}{2^{\ell -
L_t + L_{\min}}} K)$ is of length $K\cdot 2^{- \ell + L_t - L_{\min}}$.

\begin{table}[htp]
    \centering
    \begin{tabular}{lccccccc}
        \toprule
        & \multicolumn{3}{c}{Frequency} & \multicolumn{3}{c}{Time} \\
        \cmidrule(lr){2-4}
        \cmidrule(lr){5-7}
        $\ell$ range & domain & $\sharp \{i\}$ & $w(A)$
        & domain & $\sharp \{j\}$ & $w(B)$ & $w(A) w(B)$ \\
        \toprule
        $0 \leq \ell \leq L_{\min}$ &
        $A^\ell_i$ & $2^\ell$ & $K\cdot 2^{-\ell}$ &
        $B^{L-\ell}_j$ & $2^{L-\ell}$ & $2^{\ell - L}$ &
        $K \cdot 2^{-L}$ \\
        $L_{\min} < \ell \leq L_t$ &
        $A^\ell_i$ & $2^{L_{\min}}$ & $K\cdot 2^{-L_{\min}}$ &
        $B^{L-\ell}_j$ & $2^{L-\ell}$ & $2^{\ell - L}$ &
        $K \cdot 2^{\ell - L_{\min} - L}$ \\
        $L_t < \ell \leq L$ &
        $A^\ell_i$ & $2^{\ell - L_t + L_{\min}}$ & $K\cdot 2^{-\ell +
        L_t - L_{\min}}$ &
        $B^{L-\ell}_j$ & $2^{L-\ell}$ & $2^{\ell - L}$ &
        $K \cdot 2^{- L_\xi - L_{\min}}$ \\
        \bottomrule
    \end{tabular}
    \caption{Complementary domain pairs used in \NetName{} on
    all layers. $\sharp \{i\}$ and $\sharp \{j\}$ are lengths
    of index $i$ and $j$, $w(\cdot)$ is the domain length
    function, and $L_{\min} = \min(L_t, \log K - L_\xi)$. }
    \label{tab:complementary-domain-pairs}
\end{table}

Table~\ref{tab:complementary-domain-pairs} lists all domain pairs used
in {\it Butterfly} algorithm and \NetName{}, which is important to
the later complexity analysis. Theorem~\ref{thm:low-rank} is applied
to all these domain pairs in our later analysis.  Notice that when $L
\leq \log K$, we have $L_{\min} = L_t$ and the second range of $\ell$
in Table~\ref{tab:complementary-domain-pairs} is empty.

{\bf Butterfly algorithm.} For properly chosen $r$,
the Fourier kernel restricted to each domain pair in
Table~\ref{tab:complementary-domain-pairs} admits the low-rank
representation and hence the submatrix $\{\Kfun{\xi_p t_q}\}_{\xi_p
\in A^\ell_i, t_q \in B^{L-\ell}_j}$ is approximately of rank $r$.
An explicit formula for the low-rank approximation is given by
a discrete version of Theorem~\ref{thm:low-rank}. Then the {\it
Butterfly} algorithm for Fourier kernel can be described layer by
layer as follows.

\begin{enumerate}

    \item{\emph{Interpolation $(\ell = 0)$.}} For $A = A^0_0$
    and each subdomain $B = B^L_j$, conduct a coefficient transference
    from uniform grid in $B$ to Chebyshev points in $B$ and denote
    the transferred coefficients as $\{ \coef{AB}{k}\}_{1\le k \le r}$,
    i.e.,
    \begin{equation} \label{eq:sum1}
        \coef{AB}{k} = \sum_{t \in B} 
        \Kfun{\xi_0 \cdot \left( t - t_k \right)}
        \calL_k(t)
        x(t), \quad 1\le k
        \le r,
    \end{equation}
    where $\xi_0$ is the center of $A$ and $t_k$ denotes the Chebyshev
    point. According to Theorem~\ref{thm:low-rank}, we note that,
    \begin{equation}
        \hat{x}(\xi_p) =
        \sum_{B = \{B^L_j\}}
        \sum_{t_q \in B} \Kfun{\xi_p \cdot t_q} x(t_q)
        \approx
        \sum_{B = \{B^L_j\}}
        \sum_{t_k \in B} \Kfun{\xi_p \cdot t_k}
        \coef{AB}{k}
        \quad \xi_p \in A,
    \end{equation}
    where $t_q$ and $\xi_p$ denote uniform grid points and the
    approximation accuracy is controlled by $r$ and $L$ based on
    \eqref{eq:low-rank1}.

    \item{\emph{Recursion $(\ell = 1, \dots, L_t)$.}} For each
    domain pair $(A,B) = (A^\ell_i,B^{L-\ell}_j)$, construct the
    transferred coefficients $\{\coef{AB}{k}\}_{1 \le k \le r}$. Let
    $P$ denote the parent of $A$ and $C$ denote a child of $B$ at
    layer $\ell-1$. Throughout, we shall use the notation $C \succ B$
    when $C$ is a child of $B$. At layer $\ell-1$, the coefficients
    $\{\coef{PC}{s}\}_{1\le s \le r}$ satisfy,
    \begin{equation}
        \hat{x}(\xi_p) 
        \approx
        \sum_{C = \{B^{L-\ell+1}_j\}}
        \sum_{t_s \in C} \Kfun{\xi_p \cdot t_s}
        \coef{PC}{s}
        =
        \sum_{B = \{B^{L-\ell}_j\}}
        \sum_{\substack{C \succ B \\ t_s \in C}} \Kfun{\xi_p \cdot t_s}
        \coef{PC}{s}
        \quad \xi_p \in P.
    \end{equation}
    Since $A \subset P$, the above approximation holds for $\xi_p \in
    A$ as well. Now conduct a coefficient transference from Chebyshev
    points in $C \succ B$ to Chebyshev points in $B$ and denote the
    transferred coefficients as $\{ \coef{AB}{k}\}_{1\le k \le r}$, i.e.,
    \begin{equation} \label{eq:sum2}
        \coef{AB}{k} = \sum_{\substack{C \succ B\\ t_s \in C}} 
        \Kfun{\xi_0 \cdot \left( t_s - t_k \right)}
        \calL_k(t_s)
        \coef{PC}{s}, \quad 1\le k
        \le r,
    \end{equation}
    where $\xi_0$ denotes the center of $A$, $t_s$ and $t_k$
    denote the Chebyshev in $C$ and $B$ respectively. According to
    Theorem~\ref{thm:low-rank}, the transferred coefficients admit
    the approximation,
    \begin{equation}
        \hat{x}(\xi_p)
        \approx
        \sum_{B = \{B^{L-\ell}_j\}}
        \sum_{t_k \in B} \Kfun{\xi_p \cdot t_k}
        \coef{AB}{k}
        \quad \xi_p \in A.
    \end{equation}

    \item{\emph{Switch $(\ell = L_t)$.}} For the layer visited $\ell
    \leq L_t$, the Chebyshev interpolation is applied to variable
    $t$, while for layer $\ell > L_t$ the interpolation is applied
    to variable $\xi$ . Hence, we switch the role of $t$ and $\xi$
    at this step. For all pairs $(A,B) = (A^{L_t}_i,B^{L_\xi}_j)$ in
    the last step, $\coef{AB}{s}$ denotes the coefficients obtained
    by Chebyshev interpolation. Let $\{\xi_k^A\}_k$ and $\{t_s^B\}_s$
    denote the Chebyshev points in $A$ and $B$ respectively. Then
    we abuse notation $\coef{AB}{k}$ and define Fourier transformed
    coefficients for $(A,B)$ as,
    \begin{equation} \label{eq:switchsum}
        \coef{AB}{k} := \sum_{s=1}^{r} \Kfun{\xi_k^A \cdot t_s^B}
        \coef{AB}{s} \approx \sum_{t_q \in B} \Kfun{\xi_k^A \cdot
        t_q} x(t_q),
    \end{equation}
    where $t_q$ denotes the original uniform distributed points in $B$
    and the approximation is due to the definition of $\coef{AB}{s}$
    and \eqref{eq:low-rank2}.

    \item{\emph{Recursion $(\ell = L_t + 1, \dots, L_t + L_\xi)$.}} For
    each pair $(A,B) = (A^\ell_i,B^{L-\ell}_j)$, and the corresponding
    parent domain $P$ and child domain $C$ of $A$ and $B$ respectively,
    the coefficients $\{\coef{PC}{s}\}_{1\le s \le r}$ satisfy,
    \begin{equation} \label{eq:condcoef}
        \coef{PC}{s} \approx \sum_{t_q \in C} \Kfun{\xi_s \cdot t_q}
        x(t_q),
    \end{equation}
    where $\xi_s$ denotes the Chebyshev points in $P$. Given the second
    approximation in Theorem~\ref{thm:low-rank}, we have, with notation
    $\xi_p$ and $\xi_k$ being uniform points and Chebyshev points in $A$
    respectively,
    \begin{equation} \label{eq:hatf-split}
        \begin{split}
            \hat{x}(\xi_p) 
            & = 
            \sum_{B = \{B^{L-\ell}_j\}}
            \sum_{\substack{C \succ B \\ t_q \in C}}
            \Kfun{\xi_p \cdot t_q} x(t_q) \\
            & \approx 
            \sum_{B = \{B^{L-\ell}_j\}}
            \sum_{\substack{C \succ B \\ t_q \in C}}
            \sum_{\xi_k \in A}
            \Kfun{\left( \xi_p - \xi_k \right) \cdot t_0^C} \calL_k(\xi_p)
            \Kfun{\xi_k \cdot t_q} x(t_q) \\
            & \approx 
            \sum_{B = \{B^{L-\ell}_j\}}
            \sum_{\substack{C \succ B \\ t_q \in C}}
            \sum_{\xi_k \in A}
            \Kfun{\left( \xi_p - \xi_k \right) \cdot t_0^C} \calL_k(\xi_p)
            \left( \sum_{\xi_s \in P}
            \Kfun{\left( \xi_k - \xi_s \right) \cdot t_0^C} \calL_s(\xi_k)
            \Kfun{\xi_s \cdot t_q} \right) x(t_q) \\
            & \approx 
            \sum_{B = \{B^{L-\ell}_j\}}
            \sum_{\xi_k \in A}
            \Kfun{\left( \xi_p - \xi_k \right) \cdot t_0^C} \calL_k(\xi_p)
            \left( \sum_{\xi_s \in P}
            \sum_{C \succ B}
            \Kfun{\left( \xi_k - \xi_s \right) \cdot t_0^C} \calL_s(\xi_k)
            \coef{PC}{s} \right), \\
        \end{split}
    \end{equation}
    where $t_0^C$ denotes the center of $C$, the first and second
    approximation are due to \eqref{eq:low-rank2} and the last
    approximation comes from the definition of $\coef{PC}{s}$. The
    summations in the bracket in the last line of \eqref{eq:hatf-split}
    defines a transference between coefficients. Hence, $\coef{AB}{k}$
    defined as 
    \begin{equation} \label{eq:sum3}
        \coef{AB}{k} = \sum_{\xi_s \in P}
        \sum_{C \succ B}
        \Kfun{\left( \xi_k - \xi_s \right) \cdot t_0^C} \calL_s(\xi_k)
        \coef{PC}{s},
    \end{equation}
    naturally satisfies \eqref{eq:condcoef}.

    \item{\emph{Interpolation $(\ell = L)$.}} Finally, $\ell =
    L$, for $B = B^0_0$ and each $A = A^L_i$, we approximate the
    $\hat{x}(\xi_p)$ for $\xi_p \in A$ as
    \begin{equation} \label{eq:sum4}
        \begin{split}
            \hat{x}(\xi_p) 
            & = 
            \sum_{t_q \in B}
            \Kfun{\xi_p \cdot t_q} x(t_q) \\
            & \approx
            \sum_{t_q \in B}
            \sum_{\xi_k \in A}
            \Kfun{\left( \xi_p - \xi_k \right) \cdot t_0} \calL_k(\xi_p)
            \Kfun{\xi_k \cdot t_q} x(t_q)
            \approx
            \sum_{\xi_k \in A}
            \Kfun{\left( \xi_p - \xi_k \right) \cdot t_0} \calL_k(\xi_p)
            \coef{AB}{k}.
        \end{split}
    \end{equation}

\end{enumerate}

We notice that summation kernels in \eqref{eq:sum1} and \eqref{eq:sum2}
relies only on the relative distance of $t$s and, hence, it can be viewed
as a convolution kernel. Similarly summation kernels in \eqref{eq:sum3}
and \eqref{eq:sum4} relies only on the relative distance of $\xi$s,
which can be viewed as convolution kernels as well. Such an observation
plays important role in the design of \NetName{}.

\section{\NetName{}}
\label{sec:bnet}

\NetName{} is a novel structured CNN which requires far less
number of parameters to accurately represent functions that can
be expressed in the frequency domain. The essential building
block of \NetName{} is the interpolation convolutional layer,
illustrated in Figure~\ref{fig:interpCNN} and described
in Section~\ref{sec:interpCNN}, which by itself is also
interesting as it gives another way of interpreting channel
mixing. Section~\ref{sec:butterflyNet} assembles interpolation
convolutional layers together and proposes the \NetName{}. A complexity
analysis is carefully derived in Section~\ref{sec:complexity}. Finally,
in Section~\ref{sec:bnetMatrix}, we provide the matrix representation
of the \NetName{}, which significantly simplifies the analysis in
Section~\ref{sec:analysis}.

\subsection{Interpolation Convolutional Layer}
\label{sec:interpCNN}

To introduce the interpolation convolutional operation, we first
introduce an equivalent formula of the usual convolutional layer
by ``channel unfolding'', and then illustrate the layer through a
hierarchical interpolation example.

\begin{figure}[htp]
\vskip -0.2 in  
    \centering
    \begin{subfigure}[b]{0.48\textwidth}
        \centering
        \includegraphics[width=0.6\linewidth]{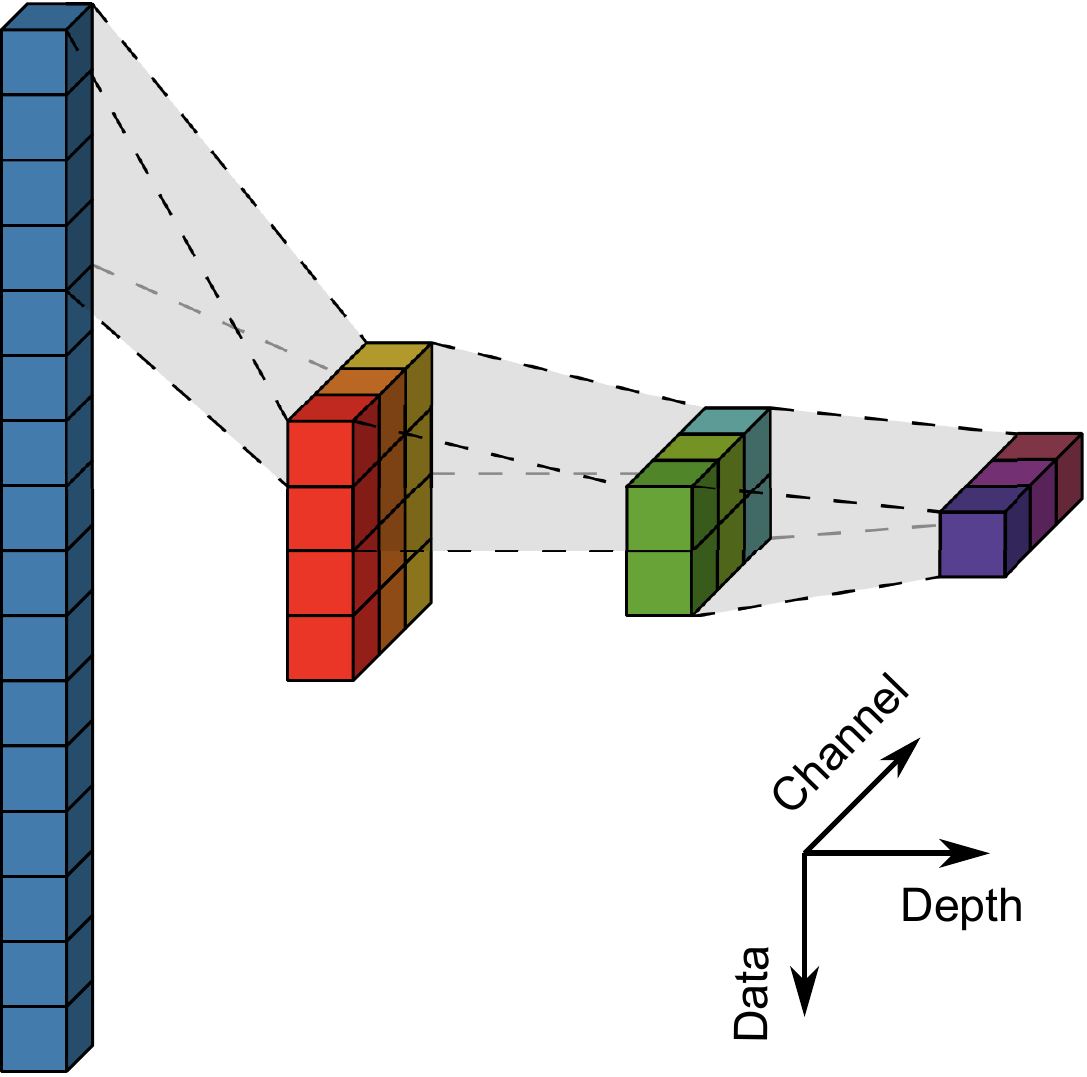}
        \caption{1D interpolation convolutional layers with folded
        channel representation}
    \end{subfigure}
    \quad
    \begin{subfigure}[b]{0.48\textwidth}
        \centering
        \includegraphics[width=0.6\linewidth]{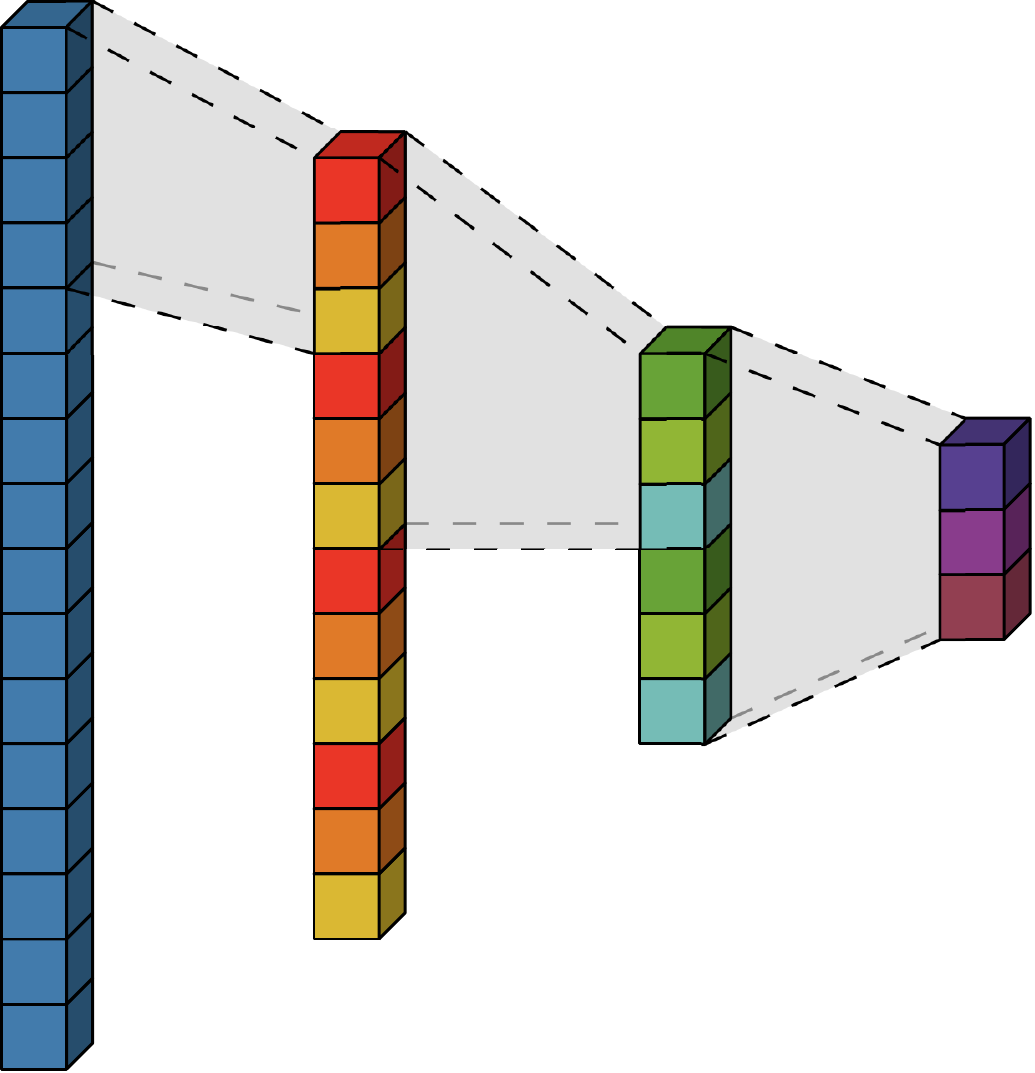}
        \caption{1D interpolation convolutional layers with unfolded
        channel representation}
    \end{subfigure}
    \caption{1D interpolation convolutional layers with input size
    16. The first layer is a 1D conv with filter size and stride size
    both being 4 and contains 3 output channels. All later layers are
    1D convs with filter size and stride size being 2 and the input
    and output channels sizes are 3.} \label{fig:interpCNN}
\end{figure}

Let $\{x(i,k_1) \mid i = 0, \dots, n-1; k_1 = 1, \dots, c_1\}$ be a
general input data with length $n$ and $c_1$ channels. Assume the
1D convolutional layer maps $c_1$ input channels to $c_2$ output
channels and the convolution filter is of size $w$. The parameters,
then, can be written as $W_{k_2,i,k_1}$ for $i = 0, \dots, w-1$,
$k_1 = 1, \dots, c_1$, and $k_2 = 1, \dots, c_2$. The output of the
convolutional layer, under these notations, is written as
\begin{equation} \label{eq:interpChannel}
    y(j,k_2) = \sum_{ \substack{ 0 \leq i \leq w-1 \\
    1 \leq k_1 \leq c_1 }} W_{k_2,i,k_1} x(i+s(j-1),k_1),
\end{equation}
where $s \geq 1$ is the stride size and $j = 0, \dots, (n-w)/s$
denotes the data index of output. In many cases, it is more convenient
to unfold the channel index into a vector as the input data, i.e.,
$x\left[i c_1+k_1\right] = x(i,k_1)$, $y\left[j c_2+k_2\right]
= y(j,k_2)$, and $W[k_2,i w + k_1] = W_{k_2,i,k_1}$. Hence
\eqref{eq:interpChannel} can be represented as the matrix vector product,
\begin{equation} \label{eq:matInterpCNN}
    y(j,:) = y[j c_2+(1:c_2)] = W x[s c_1 j+(1:wc_1)],
\end{equation}
where Matlab notation is adopted. Without considering the weight sharing
of bias term in the convolution layer, all channel direction can be
unfolded into the data dimension and the convolution is modified as a
block convolution. Such an unfolded convolutional layer will be called
the interpolation convolutional layer. Interpolation convolutional layer
is a way to understand the relation between channel dimension and data
dimension, while, in practice, it is still implemented through regular
convolutional layer.

The representation of interpolation convolutional layer is motivated by
the observation that function interpolation (coefficient transference
in {\it Butterfly} algorithm, e.g., \eqref{eq:sum1}, \eqref{eq:sum2},
\eqref{eq:sum3} and \eqref{eq:sum4}) can be naturally represented as a
multi-channel convolution. In this setting, unfolding channels is more
natural. Let $B_0, B_1, \dots, B_{J-1}$ be a equal spaced partition
of $[0,1)$, i.e., $B_j = [\nicefrac{j}{J}, \nicefrac{j+1}{J})$, and
$t_{k_1}^{B_j}$ denote the $k_1$-th discretization point in $B_j$
for $k_1 = 1, \dots, c_1$. We further assume that the locations of
$t_{k_1}^{B_j}$ relative to $B_j$ are the same for all $j$. The
input data is viewed as the function $x(t)$ evaluated at the points
$t_{k_1}^{B_j}$, i.e., $x(j,k_1) = x(t_{k_1}^{B_j})$. Let $s_{k_2}^{B_j}$
be the interpolation points on $B_j$ for $k_2 = 1, \dots, c_2$,
with the Lagrange basis polynomial given by
\begin{equation} \label{eq:LagrangePoly2}
    \calL_{k_2}(t) = \prod_{\substack{p = 1\\p \neq k_2}}^{c_2}
    \frac{t-s_p^{B_j}}{s_{k_2}^{B_j}-s_p^{B_j}}.
\end{equation}
The interpolated function of $x(t)$ at $s_{k_2}^{B_j}$ is then defined
as,
\begin{equation} \label{eq:sourcetrans}
    \begin{split}
        x(s_{k_2}^{B_j}) \approx y(j,k_2) = \sum_{k_1 = 1}^{c_1}
        \calL_{k_2}(t_{k_1}^{B_j})x(t_{k_1}^{B_j}) & =  \sum_{k_1
        = 1}^{c_1} \prod_{\substack{p = 1\\ p \neq k_2}}^{c_2}
        \frac{t_{k_1}^{B_j}-s_p^{B_j}}{s_{k_2}^{B_j}-s_p^{B_j}}
        x(j,k_1) \\
        & = \sum_{k_1 = 1}^{c_1}
        \prod_{\substack{p = 1\\ p \neq k_2}}^{c_2}
        \frac{t_{k_1}^{B_0}-s_p^{B_0}}{s_{k_2}^{B_0}-s_p^{B_0}}
        x(j,k_1),
    \end{split}
\end{equation}
where $j = 0, \dots, J-1$. The last equality in \eqref{eq:sourcetrans}
is due to the fact that each fraction in \eqref{eq:sourcetrans}
depends only on the relative distance and is thus independent of
$B_j$. Therefore, we could denote $W_{k_2,i,k_1} = \prod_{p \neq k_2}
\nicefrac{(t_{k_1}^{B_0}-s_p^{B_0})}{(s_{k_2}^{B_0}-s_p^{B_0})}$,
and thus the source transfer formula \eqref{eq:sourcetrans} can be
interpreted as convolution \eqref{eq:interpChannel} with the stride
size $s$ being the same as the filter size $w$, i.e., $s=w$. In
this representation, the two channel indices $k_1$ and $k_2$ denote
the original points and interpolation points within each domain
$B_j$. Therefore, unfolding the channel index of both $x$ and $y$
leads to the natural ordering of the index of points on $[0,1)$.

For CNN with multiple convolutional layers, the unfolding of the
channel index could be done recursively. Figure~\ref{fig:interpCNN}
(a) illustrates 1D interpolation convolutional layers, whereas
Figure~\ref{fig:interpCNN} (b) shows its unfolded representation. Gray
zones in both figures indicate the data dependency between
layers. Figure~\ref{fig:interpCNN} (b) can also be understood as a
recursive function interpolation. The domain is first divided into
four subdomains and the first layer interpolates the input function
within each subdomain to its three interpolation points. The layer
afterwards merges two adjacent subdomains into a biger subdomain and
interpolates the function defined on the previous 6 interpolation
points to the new 3 interpolation points on the merged subdomain.

If the assumption $w=s$ is removed, the convolutional layer can still
be understood as an interpolation with overlapping subdomains. Similar
idea is used in Simpson's rule and multi-step methods.

\subsection{\NetName{} Architecture}
\label{sec:butterflyNet}

This section formally introduce \NetName{} architecture. We follow the
exact structure of {\it Butterfly} algorithm here. Parallel reading
of Section~\ref{sec:forwardalgo} and this section is recommended. For
each layer, we introduce the neural network structure followed by
specifying the pre-defined {\it Butterfly} initialization and an
explanation related to the Fourier transform.

Let $x(t)$ be the input data viewed as a signal in time. Time-frequency
analysis usually decomposes the signal into different modes according
to frequency range, e.g., high-, medium-, low-frequency modes. Most
importantly, once the signal is decomposed into different modes,
they are analyzed separately and will not be mixed again. This
corresponds to the non-mixing channel idea in \NetName{} and the
non-mixing channel has an explicit correspondence with frequency modes.

We adopt the same notations as in Section~\ref{sec:forwardalgo}:
the input vector is of length $N$ and the output is a feature vector
of length $K$. Let $L$ denote the number of major layers in the
\NetName{}, and $r$ denote the size mixing channels.  Further $L_t$
and $L_\xi$ denote the number of layers before and after the switch
layer with $L = L_t + L_\xi$. We assume $L \leq \log N$ and $L_\xi
\leq \log K$. The input tensor is denoted as $f(t,1)$ for $t = 0,
\dots, N-1$. If {\it Butterfly} initialization is used, we construct
an $L = L_t + L_\xi$ layer hierarchical partition of both domains as
in Table~\ref{tab:complementary-domain-pairs}. Since the input vector,
the output vector, and the weights in {\it Butterfly} algorithm are
of complex value, the connection between complex-valued operations and
real-valued operations are complicated. Under ReLU activation function,
the connection is detailed in Appendix~\ref{app:complex-real}.
Throughout the \NetName{}, a general tensor notation
$\lambda^{(\ell)}(i,j,k)$ is used,
\begin{equation}
    \lambda^{(\ell)}(i,j,k) = \lambda^{A^\ell_i B^{L-\ell}_j}_k,
\end{equation}
corresponding to coefficients in {\it Butterfly} algorithm in
Section~\ref{sec:forwardalgo}. The index $k$ denotes the mixing
channel. The index $i$ and $j$ denote the non-mixing channel and
data dimension before switch layer and denote the data dimension and
non-mixing channel after switch layer. The range of index $\ell$, $i$,
and $j$ can be found in Table~\ref{tab:complementary-domain-pairs}.

Then the \NetName{} architecture is described as follows under
complex-valued operations.

\begin{enumerate}

    \item{\emph{Interpolation $(\ell = 0)$.}} Let $m =
    \nicefrac{N}{2^L}$ denote the filter size, which corresponds to
    the number of points in each $B^L_j$. A 1D convolution layer
    with filter size $m$, stride size $m$ and output channel $r$
    is applied to $x(:,1)$ together with added bias term and ReLU
    activation. The weight tensor is denoted as $W^{(0)}_{k,q,1}$,
    where $k = 1, \dots, r$ and $q = 1, \dots, m$. This layer maps
    the input tensor $x$ to an output tensor $\lambda^{(0)}$.

    Following the notation in \eqref{eq:sum1}, the weight tensor can
    be initialized as,
    \begin{equation}
        W^{(0)}_{k,q,1} \exeq \Kfun{\xi_0 \cdot \left( t_q - t_k \right)}
        \calL_k(t_q), \quad 1\le k \le r \text{ and } 1 \le q \le m,
    \end{equation}
    where $\exeq$ denotes extended assign operator as defined in
    Appendix~\ref{app:complex-real}.

    This step interpolates function from uniform grid points to Chebyshev
    points. When the frequency domain of the input signal is not
    symmetric around origin, this step also extracts extra phase term.

    \item{\emph{Recursion $(\ell = 1, \dots, L_t)$.}} The input
    and output tensors at layer $\ell$ are $\lambda^{(\ell-1)}$ and
    $\lambda^{(\ell)}$. For each of the non-mixing channel at previous
    layer, two (for $\ell \leq L_{\min}$) or one (for $L_{\min} <
    \ell \leq L_t$) 1D convolution layers with filter size $2$,
    stride $2$ and output channel $r$ are applied together with
    bias term and ReLU activation. The weight tensors are denoted as
    and $W^{(\ell),i}_{k,c,s}$, where $k, s = 1, \dots, r$ and $c$
    corresponds to the index of child domain $C$ of $B$.

    Following the notations in \eqref{eq:sum2}, the weight tensor
    can be initialized as,
    \begin{equation}
        W^{(\ell),i}_{k,c,s} \exeq \Kfun{\xi_0^i \cdot \left( t_s^c -
        t_k \right)} \calL_k(t_s^c)
    \end{equation}
    where $\xi^i_0$ denotes the center of $A^\ell_i$, $t_s^c$ denotes
    the Chebyshev points in $C_c = B^{L-\ell+1}_c$ and $t_k$ denotes
    the Chebyshev points in $B^{L-\ell}_0$.

    Each $\xi^i_0$ is the center of $A^\ell_i$ corresponding to
    different frequency domain. Different frequency component
    in the input signal is now organized in different non-mixing
    channels. They will be transformed independently later which
    is related to the orthogonality of basis functions in different
    non-overlapping frequency domains.

    \item{\emph{Switch $(\ell = L_t)$.}} This layer is
    a special layer of local operations. Denote the input
    tensor as $\lambda^{(L_t)}(i,j,s)$ and the dense weights as
    $W^{(L_t),i,j}_{k,s}$ for $k,s = 1, \dots, r$ and $i,j$ with range
    in Table~\ref{tab:complementary-domain-pairs}. For each $i,j$,
    $W^{(L_t)}$ is a $r$ by $r$ dense matrix.  The operation at this
    layer is as follows,
    \begin{equation}
        \lambda^{(L_t)}(i,j,k) = \sum_{s=1}^r
        W^{(L_t),i,j}_{k,s}
        \lambda^{(L_t)}(i,j,s)
    \end{equation}
    for each pair of $i,j$. A bias term and ReLU layer are applied to
    the output tensors.

    Following the notation in \eqref{eq:switchsum}, the dense
    weight tensors can be initialized as,
    \begin{equation}
        W^{(L_t),i,j}_{k,s}
        \exeq \Kfun{\xi_{k}^{A^{L_t}_{i}} \cdot
        t^{B^{L_\xi}_j}_{s}},
    \end{equation}
    where $\xi_{k}^{A^{L_t}_{i}}$ and $t^{B^{L_\xi}_j}_{s}$ are
    Chebyshev points in $A^{L_t}_i$ and $B^{L_\xi}_j$ respectively.

    For each $i,j$, the Fourier operator is applied at this layer.
    Afterwards, interpolation is applied again in frequency domains.

    \item{\emph{Recursion $(\ell = L_t+1, \dots, L_t + L_\xi)$.}} The
    input and output tensors at layer $\ell$ are $\lambda^{(\ell-1)}$
    and $\lambda^{(\ell)}$.  The weight tensors are denoted as
    and $W^{(\ell),j}_{k,c,s}$, where $k, s = 1, \dots, r$ and $c$
    corresponds to the index of child domain $C$ of $B$. For each
    of the non-mixing channel $j$, one 1D convolution layer is
    performed as,
    \begin{equation}
        \lambda^{(\ell)}(i,j,k) = \sum_{s=1}^r \sum_{c=0,1}
        W^{(\ell),j,a}_{k,c,s}
        \lambda^{(\ell-1)}(\floor{i/2},2j+c,s).
    \end{equation}
    where $a = i \bmod 2$. Such a convolution is also known as
    transposed convolution. This transpose property will become more
    clear later when the matrix representation is derived.

    Following the notations in \eqref{eq:sum3}, the weight tensor
    can be initialized as,
    \begin{equation}
        W^{(\ell),j,a}_{k,c,s} \exeq
        \Kfun{\left( \xi_k^a - \xi_s \right) \cdot t_0^{2j+c}}
        \calL_s(\xi_k^a)
    \end{equation}
    where $t^{2j+c}_0$ denotes the center of $C=B^{L-\ell+1}_{2j+c}$,
    $\xi_s$ denotes the Chebyshev points in $A^{\ell-1}_0$, $\xi_k^a$
    denotes the Chebyshev points in $A^{\ell}_a$ for $a=0,1$.

    This part is similar to step 3. Instead of organizing output in
    non-mixing frequency domains, different time component in the input
    signal is now organized in different non-mixing channel. This is
    due to the complementary property between time and frequency.

    \item{\emph{Interpolation $(\ell = L)$.}} Let $m =
    \frac{K}{2^{L_\xi + L_{\min}}}$ denote the output channel size,
    which corresponds to the number of points in each $A^L_i$. A 1D
    convolution layer with filter size $1$, stride size $1$, input
    channel $r$ and output channel $m$ is applied to $\lambda^{(L)}$
    together with added bias term and ReLU activation. The weight
    tensor is denoted as $W^{(L)}_{p,0,k}$, where $p = 1, \dots, m$
    and $k = 1, \dots, r$.

    Following the notations in \eqref{eq:sum4}, the weight tensor
    can be initialized as,
    \begin{equation}
        W^{(L)}_{p,0,k} \exeq \Kfun{\left( \xi_p - \xi_k \right)\cdot
        t_0} \calL_k(\xi_p), \quad 1\le p \le m \text{ and } 1 \le
        k \le r.
    \end{equation}

    This layer generates the output tensor denoted as $\hat{x}(i,p)$,
    for $i = 0, \dots, 2^L-1$ being the index of data and $p = 1,
    \dots, m$ being the index of channel. Reshaping $\hat{x}(s) =
    \hat{x}(i,p)$ for $s = im+p$ gives a single vector output, which
    is analogy of the output vector of the {\it Butterfly} algorithm.

    \item{\emph{Task-dependent layers.}} Any type of layers, e.g.,
    dense layer, convolution layer, transpose convolution layer, etc.,
    can be built on top of $\hat{x}$ and approximate the desired
    task. These layers are creatively designed by users, which are
    not regarded as parts of \NetName{} in the following.

\end{enumerate}

To further facilitate the understanding of the \NetName{},
Figure~\ref{fig:butterflyNet} demonstrates an example of the \NetName{}
with input vector being partitioned into 16 parts. We adopts the unfolded
representation of the mixing channel as in Figure~\ref{fig:interpCNN} (b),
and the channel direction only contains non-mixing channels.

In the above description, the non-mixing channels and mixing channels
are indexed different. If we combine this two indices into a single
channel index, and allow all channel connections to be dense,
we define another family of neural networks, namely \INetName{}.
If the {\it Butterfly} initialization in \NetName{} is applied to
\INetName{} and the rest channel connections are initialized as zero,
then \INetName{} is an identical operator as \NetName{} with {\it
Butterfly} initialization.

Also notice that $L_\xi$ is a tunable parameter for both \NetName{}
and \INetName{}. When $L_\xi = 0$, all transpose convolutional layers
disappear. In this case, the switch layer and interpolation layer
$\ell = L$ can be combined as a entry-wise product operator, which
can be implemented through a dense layer. \INetName{} with $L_\xi =
0$ is then a regular CNN and \NetName{} is a CNN with channel sparse
structure~\cite{Xu2019}. All of our following complexity analysis
and approximation analysis apply to $L_\xi = 0$ with/without merging
switch layer and interpolation layer $\ell = L$.

\subsection{Complexity Analysis}
\label{sec:complexity}

One major advantage of the proposed \NetName{} is the reduction of
model complexity and computational complexity. We now conduct a careful
count on the number of weights. Since we use complex embedding with
ReLU, i.e., ``$\exeq$'', each complex multiplicative weight is actually
implemented by a $4 \times 4$ matrix and each bias is implemented as a
vector of length $4$. Hence $4r$ is the actual number of channels. We
conduct the counting layer by layer based on the network description
and Table~\ref{tab:complementary-domain-pairs}. On the interpolation
($\ell = 0$), there is only one convolutional kernel, which has
$16r \cdot \frac{N}{2^L}$ filter weights and $4r$ bias weights. On
the recursive ($\ell = 1, \dots, L_t$), there are $\min(2^{\ell-1},
2^{L_{\min}})$ non-mixing input channels, $\min(2^\ell, 2^{L_{\min}})$
non-mixing output channels, and the filter size is $2$. The number
of both input and output mixing channels are $4r$. Hence there are
$\min(2^\ell, 2^{L_{\min}}) \cdot (4r)^2 \cdot 2$ filter weights and
$\min(2^\ell, 2^{L_{\min}}) \cdot 4r$ bias weights. On the switch layer
($\ell = L_t$), there are $\min(2^{L_t}, 2^{L_{\min}})$ non-mixing
channels and $2^{L_\xi}$ data. Since the connection is locally fully
connected layer, the overall number of multiplicative weights is
$\min(2^{L_t}, 2^{L_{\min}}) \cdot 2^{L_\xi} (4r)^2$ with additional
$\min(2^{L_t}, 2^{L_{\min}}) \cdot 2^{L_\xi} 4r$ bias weights. On
the recursive layer ($\ell = L_t + 1, \dots, L_t + L_\xi$), it is a
transpose of the previous recursive layer $2L_t-\ell$. The numbers of
filter weights and bias weights are then $2^{L - \ell} \cdot (4r)^2
\cdot 2$ and $2^{L-\ell} \cdot 4r$. The last interpolation ($\ell =
L$) is similar as the interpolation ($\ell = 0$). The numbers of
filter weights and bias weights are $4\max(1,\frac{K}{2^L})(4r)$ and
$4\max(1,\frac{K}{2^L})$.  In summary, the overall network complexity
for \NetName{} is
\begin{align}
     &   \frac{N}{2^L}16r + \sum_{\ell = 1}^{L_t} \min(2^\ell,
        \frac{K}{2^{L_\xi}}) (4r)^2 \cdot 2 + \min(2^{L}, K) (4r)^2 +
        \sum_{\ell = L_t+1}^{L} 2^{L-\ell} (4r)^2 \cdot 2 + \max(1,
        \frac{K}{2^L})16r \nonumber \\
     & ~~~ + 4r + \sum_{\ell = 1}^{L_t} \min(2^\ell, \frac{K}{2^{L_\xi}})
        4r + \min(2^{L}, K) 4r + \sum_{\ell = L_t+1}^{L}
        2^{L-\ell} 4r + \max(1, \frac{K}{2^L})4r \nonumber \\
    &  =  \bigO (\frac{Nr}{2^L} + LKr^2),
     \label{eq:complex-1}
\end{align}
where the first row is the number of multiplicative/filter weights
and the second row is the number of bias weights.

\begin{remark}
    We further derive a more precise upper bound for the number of
    weights in \NetName{} with an absolute constant under $L = \log
    N$. Denote the expression in \eqref{eq:complex-1} as $n_b$. We
    have,
    \begin{equation} \label{eq:bnet-complexity-precise}
        n_b \leq 40r + L_t \frac{K}{2^{L_\xi}} (32r^2+4r) + K
        (16r^2+4r) + 2^{L_\xi} (32r^2+4r) < 90 LKr^2,
    \end{equation}
    where the last inequality adopts the assumption on $r$ as in
    Theorem~\ref{thm:low-rank}.
\end{remark}

The number of parameters in the inflated \NetName{}, i.e., all
channels are mixed, can be derived in a similar way and the overall
complexity is,
\begin{align} 
        & \frac{N}{2^L}16r + \sum_{\ell = 1}^{L_t} \min(2^\ell,
        \frac{K}{2^{L_\xi}})^2 (4r)^2 \cdot 2 + \min(2^{L}, K) (4r)^2 +
        \sum_{\ell = L_t+1}^{L} 2^{2L-2\ell} (4r)^2 \cdot 2 + \max(1,
        \frac{K}{2^L}) 16r \nonumber \\
        &~~~ + 4r + \sum_{\ell = 1}^{L_t} \min(2^\ell, \frac{K}{2^{L_\xi}})
        4r + \min(2^{L}, K) 4r + \sum_{\ell = L_t+1}^{L}
        2^{L-\ell} 4r + \max(1, \frac{K}{2^L}) 4r \nonumber \\
        &= \bigO (\frac{Nr}{2^L} + L\frac{K^2}{2^{2L_\xi}} +
        2^{2L_\xi}).
        \label{eq:complex-2}
\end{align}

\begin{table}[t]
    \centering
    \begin{tabular}{lcccc}
        \toprule
        & \multicolumn{2}{c}{$K \sim N \sim 2^L$}
        & \multicolumn{2}{c}{$K = O(1), N \sim 2^L$} \\
        \cmidrule(lr){2-3} \cmidrule(lr){4-5}
        & \sNetName{} & Inflated \sNetName{} 
        & \sNetName{} & Inflated \sNetName{} \\
        \toprule
        Interpolation ($\ell = 0$)
        & $r$ & $r$ & $r$ & $r$ \\
        Recursion ($\ell \leq L_t$)
        & $2^\ell r^2$ & $2^{2\ell}r^2$ & $\frac{K}{2^{L_\xi}} r^2$ & $\frac{K^2}{2^{2L_\xi}} r^2$ \\
        Switch ($\ell = L_t$)
        & $2^{L} r^2$ & $2^{L} r^2$
        & $K r^2$ & $K r^2$ \\
        Recursion ($\ell \geq L_t + 1$)
        & $2^{L-\ell} r^2$ & $2^{2(L-\ell)}r^2$ & $2^{L-\ell} r^2$ & $2^{2(L-\ell)} r^2$ \\
        Interpolation ($\ell = L$)
        & $r$ & $r$ & $r$ & $r$ \\
        \midrule
        Overall  
        & $2^L r^2$ & $2^{2L_t} r^2$ & $L_t \frac{K}{2^{L_\xi}} r^2 + K
        r^2$ & $L_t \frac{K^2}{2^{2L_\xi}} r^2 + K r^2$ \\
        $L_\xi = 0$
        & $N r^2$ & $N^2 r^2$ & $K r^2 \log N$ & $K^2 r^2 \log N$ \\
        $L_\xi = \frac{1}{2}\log K$
        & $N r^2$ & $N r^2$ & $\sqrt{K} r^2 \log N + K
        r^2$ & $K r^2 \log N$ \\
        \bottomrule
    \end{tabular}
    \caption{Leading order network complexity for \NetName{} and
    \INetName{} under two scenarios, $K \sim N \sim 2^L$ and $K = O(1),
    N\sim 2^L$. In both cases, we assume $K \leq N$ and $L_\xi \leq
    \log K$, i.e. $L - \log K \le L_t \le L$. The asymptotic regime
    is $N\to \infty$, and $r$ denotes the number of interpolation
    points (size of mixing channels) which does not change with $N$.}
    \label{tab:complex}
\end{table}

In Table~\ref{tab:complex}, we summarize the layer-wise complexity
together with the overall complexity for two special scenarios,
$K \sim N \sim 2^L$ and $K = O(1), N\sim 2^L$, for both \NetName{}
and \INetName{}, where we use the notation $A \sim B$ to denote that
$\frac{A}{B} \to c$ which is an absolute strictly positive constant
in the limit of $N \to \infty$. Recall that the number of hiearchical
layer $L$ satisfies that $K \le 2^L$.

In addition to the network complexity, the overall computational
cost for a evaluation of the \NetName{} and the inflated \NetName{}
are $\bigO(N\log N)$ and $\bigO(N\log N + K^2)$.

\subsection{Matrix Representation of \NetName{} }
\label{sec:bnetMatrix}

This section aims to demonstrate the matrix representation
of \NetName{}, which is similar to {\it Butterfly}
factorization~\cite{Li2015a}. The matrix representation further
explain the sparse connectivity of channels and more importantly
facilitates the proof in the analysis of approximation power in
Section~\ref{sec:analysis}.

We first show the matrix representation of the interpolation
convolutional layer and the matrix representation of the \NetName{}
simply stacks the interpolation convolutional layer together
with a switch layer. Figure~\ref{fig:matInterpCNN} (a) represents
\eqref{eq:matInterpCNN} when both $g$ and $f$ are vectorized. If
we permute the row ordering of the matrix, we would result blocks
of convolution matrix with the number of blocks being the size of
the output channels. Figure~\ref{fig:matInterpCNN} (b) assumes
that $s=w$ and the matrix is further simplified to be a block
diagonal matrix. According to the figures, we note that when $s=w$,
the transpose of the matrix is a representation of a transposed
convolutional layer with $W$ replaced by $W^T$.

\begin{figure}
    \centering
    \begin{subfigure}[b]{0.48\textwidth}
        \centering
        \includegraphics[height=0.5\linewidth]{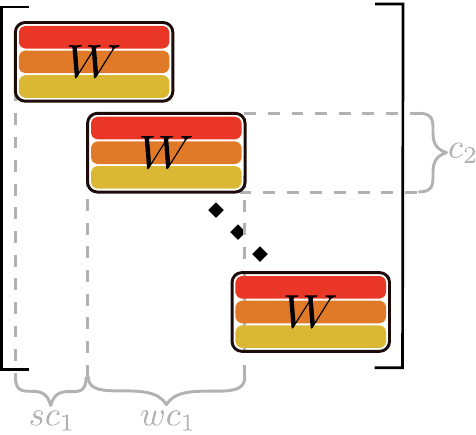}
        \caption{Matrix of 1D interp. conv. layer}
    \end{subfigure}
    ~~
    \begin{subfigure}[b]{0.48\textwidth}
        \centering
        \includegraphics[height=0.5\linewidth]{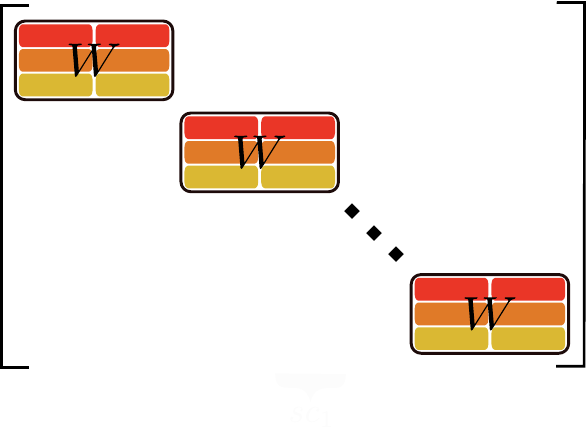}
        \caption{Matrix of 1D interp. conv. layer with $s=w$}
    \end{subfigure}
    \caption{Matrix representation of 1D interpolation convolutional
    layers} \label{fig:matInterpCNN}
\end{figure}

Since the convolutional layer with mixing channel can already be
represented as Figure~\ref{fig:matInterpCNN}, \NetName{} stack
interpolation convolutional layer together representing non-mixing
channel, which is equivalent to stack the matrix row-wise. We would
explain the matrix representation for each step of the \NetName{}.

\begin{enumerate}
    \item \emph{Interpolation $(\ell = 0)$.} The matrix representation
    is Figure~\ref{fig:matInterpCNN} (b) with $2^L$ diagonal blocks
    and each block is $W^{(0)}_{:,:,1}$ which is of size $r \times
    m$. The resulting matrix is denoted as $V$.

    \item \emph{Recursion $(\ell = 1, \dots, L_t)$.} The weight
    tensor $W^{(\ell),i}_{k,c,s}$ can be reshaped as a matrix with row
    indexed by $k$ and column indexed by $c$ and $s$, which is denoted
    as $W^{(\ell),i}$.  For $\ell \leq L_{\min}$ the convolutional
    operator of mixing and non-mixing channels can be viewed as the
    following matrices respectively,
    \begin{equation}
        H^{(\ell)}_{\floor{\nicefrac{i}{2}}}={\small
        \begin{pmatrix}
            W^{(\ell),2\floor{\nicefrac{i}{2}}} & & \\
            &  \ddots &\\
            &  & W^{(\ell),2\floor{\nicefrac{i}{2}}} \\
            \cmidrule(lr){1-3}
            W^{(\ell),2\floor{\nicefrac{i}{2}}+1} & & \\
            &  \ddots &\\
            &  & W^{(\ell),2\floor{\nicefrac{i}{2}}+1} \\
        \end{pmatrix}}, \quad \text{and }
        H^{(\ell)}= {\small
        \begin{pmatrix}
            H^{(\ell)}_{0} & &\\
            & \ddots &\\
            & & H^{(\ell)}_{2^{\ell-1}-1}\\
        \end{pmatrix}}.
    \end{equation}
    For $L_{\min} < \ell \leq L_t$ the convolutional operator can be
    viewed as the following matrices respectively,
    \begin{equation}
        H^{(\ell)}_{i}={\small
        \begin{pmatrix}
            W^{(\ell),i} & & \\
            &  \ddots &\\
            &  & W^{(\ell),i} \\
        \end{pmatrix}}, \quad \text{and }
        H^{(\ell)}= {\small
        \begin{pmatrix}
            H^{(\ell)}_{0} & &\\
            & \ddots &\\
            & & H^{(\ell)}_{2^{L_{\min}}-1}\\
        \end{pmatrix}}.
    \end{equation}

    \item \emph{Switch $(\ell = L_t)$.} The matrix representation of
    switch layer is
    \begin{equation}
        M =
        {\small
        \begin{pmatrix}
            M_{0,0} & \cdots  & M_{0,2^{L_t}-1}\\
            \vdots & \ddots & \vdots \\
            M_{2^{L_\xi}-1,0} & \cdots &
            M_{2^{L_\xi}-1,2^{L_t}-1}
        \end{pmatrix}}.
    \end{equation}
    Each $M_{i,j}$ is again an $2^{L_\xi} \times 2^{L_t}$ block matrix
    with block size $r \times r$, where all blocks are zero except
    that the $(j,i)$ block is $W^{(L_t),i,j}_{:,:}$.

    \item \emph{Recursion $(\ell = L_t+1, \dots, L_t + L_\xi)$.}
    The weight tensor $W^{(\ell),j,a}_{k,c,s}$ can be reshaped as
    a matrix with row indexed by $k,a$ and column indexed by $s$,
    which is denoted as $W^{(\ell),j}_c$. Then the convolutional
    operator of mixing and non-mixing channels can be viewed as the
    following matrices respectively,
    \begin{equation}
        G^{(\ell)}_{j}={\small\left(
        \begin{array}{ccc|ccc}
            W^{(\ell),j}_0 & & & 
            W^{(\ell),j}_1 & & \\
                           &  \ddots & &
            & \ddots &\\
            & & W^{(\ell),j}_0 & 
            & & W^{(\ell),j}_1 \\
        \end{array}
        \right)}, \quad \text{and }
        G^{(\ell)}= {\small
        \begin{pmatrix}
            G^{(\ell)}_{0} & &\\
            & \ddots &\\
            & & G^{(\ell)}_{2^{L-\ell}-1}\\
        \end{pmatrix}}.
    \end{equation}

    \item \emph{Interpolation $(\ell = L)$.} The matrix representation
    is Figure~\ref{fig:matInterpCNN} (b) with $2^L$ diagonal blocks
    and each block is $W^{(L)}_{:,0,:}$ which is of size $\min(1,
    \frac{K}{2^L}) \times r$. The resulting matrix is denoted as $U$.

\end{enumerate}

Based on the matrix representation of each layer of the \NetName{},
we could write down the overall matrix representation together with
bias terms and ReLU activation layer as,
\begin{equation} \label{eq:butterflyNet}
    y = \calB(x) = U \tilde{\sigma}\left[ G^{(L)} \tilde{\sigma}\left[
    \cdots G^{(L_t+1)} \tilde{\sigma}\left[ M \tilde{\sigma}\left[
    H^{(L_t)} \tilde{\sigma}\left[ \cdots H^{(1)} \tilde{\sigma}\left[
    V x \right] \right] \right] \right] \right] \right],
\end{equation}
where $U, G^{(\ell)}, M, H^{(\ell)}, V$ are defined as above and
$\tilde{\sigma}[\cdot]$ denotes the operation of properly adding the
bias and applying the ReLU activation.

\section{Analysis of Approximation Power}
\label{sec:analysis}

The main result of function approximation using \NetName{} is
given in Section~\ref{sec:approx-main}, where the main Theorem,
Theorem \ref{thm:approx_bnet}, relies on the Fourier kernel
approximation result Corollary~\ref{cor:bnet-para}.  The latter is
established in Section~\ref{sec:approx-fourier}, and the proof is
in Section~\ref{sec:proof}.

\subsection{Approximation of Network Function}
\label{sec:approx-main}

In this section, we first focus on approximating single-valued
function $f(\vec{x})$, where $\vec{x} \in \mathbb{R}^N$, $N$ is an
integer.  After showing the
main result Theorem \ref{thm:approx_bnet}, we discuss on the extension
to other prediction functions, particularly dense predictions in U-net.

Let $F_N$ be the $N$-by-$N$ unitary discrete Fourier matrix. Any vector
$\vec{x}$ and its discrete Fourier transform $\hat{x} = F_N \vec{x}$
obey $\norm{ \vec{x}}_2 = \norm{ \hat{x}}_2$.  By normalizing argument,
we assume that $\vec{x}$ lies in $\calX$, which is a subset of the
2-norm unit ball in $\bbR^N$. As a result, $\hat{x} \in \hat{\calX}$
is also contained in the 2-norm unit ball in $\bbC^N$. Suppose that
$\vec{x}$ is statistically distributed according to $dP(x)$ on $\calX$.
We denote the 2-norm ball of radius $r$ in $\bbC^m$ as $\hat{B}_r^m$,
and that in $\bbR^m$ as $B_r^m$.

We introduce the following assumption on $f$ and $\calX$ which,
as shown in Theorem \ref{thm:approx_bnet}, can be more efficiently
represented using \NetName{}.

\begin{assumption} \label{assump:f}
    Consider approximating $f$ under the $p$-norm of $L^p(\calX,
    dP)$, $ p \in [1,\infty]$.  There exists constants $C_2 > 0$,
    $s > 0$, an interval $I_K = [K_0, K_0+K)$ for $K \leq N$, and a
    function $g_K$ which maps from $\hat{B}_{1.1}^K $ to $\mathbb{R}$,
    such that the following three conditions hold.
    
    \begin{enumerate}[(i)]
        \item Let $\hat{x}|_{I_K}$ denote retrieving entries of the
        vector $\hat{x}$ with index in $I_k$,
        \begin{equation*}
            \| f(\vec{x}) - g_K(\hat{x}|_{I_K}) \|_{L^p(\calX, dP),
            \vec{x} \in \calX} < 0.1,
        \end{equation*}
        and note $\hat{x}|_{I_K} \in \hat{B}_1^K$ since $\hat{x}
        \in \hat{B}_1^N$.
    
        \item $g_K$ is $C_2$-Lipschitz on $\hat{B}_{1.1}^K$ with
        respect to the 2-norm, i.e.,
        \begin{equation*}
            |g_K( z ) - g_K(z')| \le  C_2 \|z - z'\|_2,
            \quad \forall z,~z' \in \hat{B}_{1.1}^K.
        \end{equation*}
        Note that if we view $g_K$ as a real-input function taking $2K$
        inputs, it also has Lipschitz constant $C_2$ (with respect to
        the 2-norm in real space). Here $C_2$ is uniformly for all $K$.
    
        \item For any $ 0 < \varepsilon < 1$, there is a multi-layer
        fully-connected network with number of parameters
        \begin{equation}\label{eq:nfc-bound}
            n_{fc} \le c(s, K) \varepsilon^{-K/s} (\log
            \frac{1}{\varepsilon}+1)
        \end{equation}
        which gives a network mapping $\phi_{K}: \hat{B}_{1.1}^K
        \to \mathbb{R}$ s.t. $\| \phi_K(z) -g_K(z)\|_{\infty, z \in
        \hat{B}_{1.1}^K} < \varepsilon $, where $c(s,K)$ is a constant
        that depends on $s$ and $K$.
    \end{enumerate}
\end{assumption}

\begin{remark}
    In Assumption~\ref{assump:f} (i), the constant $0.1$ is technical
    and can be any other constant less than one.  As will be shown in
    Theorem \ref{thm:approx_bnet}, the approximation result is most
    useful when the residual norm in Assumption~\ref{assump:f} (i) can
    be made small, say $\varepsilon \ll 1$, and then this $\varepsilon$
    will be used as the target $\varepsilon$ in both the \NetName{}
    approximation and the fully-connected network approximation.
    Note that if $f$ is band-limited to begin with, the residual
    can be made zero. For frequency decaying function, the residual
    can be made arbitrarily small by increasing $K$. Detailed in the
    examples below.
\end{remark}

{\bf Example 1.} Consider the energy functional of the 1D Laplace
operator with periodic boundary condition,
\begin{equation*}
    f(\vec{x}) = E(\vec{x}) = \sum_{1 \leq k \leq N/2}
    \frac{2}{k^2}|\hat{x}_k|^2, \quad \vec{x} \in \calX = B_{1}^N,
\end{equation*}
and $dP$ is some distribution of $\vec{x}$ on $\calX$.  For any $K >
0$, setting $I_K = [1, K]$, and
\begin{equation*}
    g_K( \hat{x}|_{I_K}) = \sum_{1 \leq \abs{k} \le K}
    \frac{2}{k^2}\abs{\hat{x}_k}^2.
\end{equation*}

\begin{enumerate}[(i)]
    \item We can verify that
    \begin{equation*}
        \| f(\vec{x}) - g_K(\hat{x}|_{I_K}) \|_{L^p(\calX, dP), \vec{x}
        \in \calX} < 2 \cdot K^{-3/2},
    \end{equation*}
    which decays as $K$ increases. This means that the residual can be
    made arbitrarily close to 0 if $K$ can be chosen sufficiently large.

    \item Due to that $g_K$ is a quadratic form, and $\|
    \hat{x}|_{I_K}\|_2 \le 1$, Assumption~\ref{assump:f} (ii) is
    satisfied with $C_2 = 2.2$.

    \item Viewing $g_K$ as a function taking $2K$ real input and
    defined on $B_{1.1}^{2K}$, it is again a quadratic form.  Then by
    setting $s=2$, or even higher positive number, the neural network
    approximation theory in \cite{Yarotsky2017} provides the uniform
    approximation needed in Assumption~\ref{assump:f} (iii) by a
    network whose model complexity is bounded by \eqref{eq:nfc-bound},
    where the constant $c(s,K)$ depends on the space dimension $2K$,
    the derivative order $s$, and the Sobolev norm of the function
    $g_K$ in ${\cal W}^s(B_{1.1}^{2K})$.
\end{enumerate}

{\bf Example 2.} Unlike the first example which makes use of the form
of $f$, this example mainly shifts assumptions on the data set.
Suppose the data domain $\calX$ consists of band-limited data points
with window length $K$,
\begin{equation*}
    \calX = \{ \vec{x} \in B_{1}^{N} \mid \hat{x}_{k} = 0 \text{
    for } k \not\in I_K\}.
\end{equation*}
Then for a regular function $f$ defined on $B_{1.1}^N$, e.g.,
$C_2$-Lipschitz and $s$-smooth, let $\calK^{-1}$ be the discrete
inverse Fourier transform from frequency window $I_K$ to $\calX
\subset \bbR^N$, and $g_K(z) = f( \calK^{-1} z)$ for any $z \in
\hat{B}_{1.1}^K$.  Hence the difference in Assumption~\ref{assump:f}
(i) is zero and the regularity assumptions on $g_K$ can be inherited
from that on $f$. Thus (ii) and (iii) in Assumption~\ref{assump:f}
are also satisfied.

In applications, if the input data vectors $\vec{x}$ are
discretizations of regular continuous signals, then they naturally
have spectrum decay property and can be approximated by band-limited
data points. Hence the Assumption~\ref{assump:f} also applies.

\begin{remark}
    Our theory leaves the abstract approximation bound
    in Assumption~\ref{assump:f} (iii) and focuses on the
    approximation analysis of \NetName{}, of which the key result is
    Corollary~\ref{cor:bnet-para}. Generally, as long as a ``nice''
    component $g_K$, which is band-limited on a Fourier window of
    length $K$ and has sufficient regularity, can be separated out
    from $f$ up to a small residual, the universal approximation
    theory of neural network gives approximation of $g_K$ depending
    on its regularity, and this fully-connected network is used on
    top of the \NetName{}.
\end{remark}

\begin{theorem} \label{thm:approx_bnet}
    Assume that the function $f$ and $\calX$ satisfy
    Assumption~\ref{assump:f}, and notations are the same as therein.
    Then for any $\varepsilon$ satisfying
    \begin{equation*}
        \| f(\vec{x}) - g_K(\hat{x}|_{I_K}) \|_{L^p(\calX, dP),
        \vec{x} \in \calX} \le \varepsilon < 0.1,
    \end{equation*}
    there exists a family of \NetName{}s with $L_\xi =0, \cdots,
    \log K$, whose numbers of parameters are all bounded by
    \begin{equation*}
        n_b \le C K \log N \left( \log \frac{1}{\varepsilon} +
        \log N \right)^2, \quad \text{$C$ being an absolute constant,}
    \end{equation*}
    and a fully connect network $\phi_K$  with number of parameters 
    \begin{equation*}
        n_{fc} \le c(s, K) \varepsilon^{-K/s} (\log
        \frac{1}{\varepsilon}+1),
    \end{equation*}
    such that using any member of the family of \NetName{}, denoting
    by $\calB$, gives
    \begin{equation*}
        \| f(\vec{x}) - \phi_K(\calB (\vec{x}) ) \|_{L^p(\calX, dP),
        \vec{x} \in \calX}  \le \varepsilon (2+ C_2) .
    \end{equation*}
\end{theorem}

\begin{remark}
    Theorem~\ref{thm:approx_bnet} holds for a family of \NetName{}s
    with different $L_\xi$s, which shows no difference in
    view of this approximation analysis. However, different
    $L_\xi$s lead to \NetName{} architectures with different
    number of parameters. \NetName{} with smaller $L_\xi$
    has more parameters (up to a logarithmic factor), hence
    achieves better post-training accuracy, which is verified in
    Section~\ref{sec:numerical-approx-power}.  \NetName{} with
    $L_\xi=0$ has the simplest architecture consisting of only
    convolutional layers and fully-connect layers, which is detailed
    in \cite{Xu2019}.
\end{remark}

\begin{remark}
    Since \NetName{} can be embedded in \INetName{},
    Theorem~\ref{thm:approx_bnet} also holds for a family of
    \INetName{} with $L_\xi = 0, \dots, \log K$, whose numbers of
    parameters are all bounded by
    \begin{equation} \label{eq:cnn-complexity}
        n_b \le C \left( \frac{K^2}{2^{2L_\xi}} \log N + 2^{2L_\xi}
        \right) \left( \log \frac{1}{\varepsilon} + \log N \right)^2.
    \end{equation}
    Notice the number of parameters first increases and then
    decrease as $L_\xi$ varying from $0$ to $\log K$. The lightest
    \INetName{} is achieved when $L_\xi = \frac{\log K}{2}$. As
    shown in \cite{Xu2019}, \INetName{} with $L_\xi=0$ is a regular
    CNN. Hence, Theorem~\ref{thm:approx_bnet} gives an approximation
    result for regular CNNs if \NetName{} is replaced by \INetName{}
    and the parameter number upper bound on $n_b$ is replaced by
    \eqref{eq:cnn-complexity}.
\end{remark}

\begin{proof}
    Under Assumption~\ref{assump:f}, for any $\vec{x} \in \calX$,
    \begin{equation}\label{eq:pf3term}
        |f(\vec{x}) -  \phi_K(\calB (\vec{x}) )| \le |f(\vec{x})
        - g_K(\hat{x}|_{I_K})| + |g_K(\hat{x}|_{I_K}) - g_K(
        \calB (\vec{x}) )| + |g_K(\calB (\vec{x})) - \phi_K(\calB
        (\vec{x}))|,
    \end{equation}
    and by Corollary \ref{cor:bnet-para}, there exists a family
    of \NetName{}s satisfying the requirement and for any $\calB$
    belonging to the family,
    \begin{equation}\label{eq:calB-uniform-bound}
        \| \calB (\vec{x}) -  \hat{x}|_{I_K}\|_2 < \varepsilon < 0.1,
        \quad \forall \vec{x} \in \calX,
    \end{equation}
    due to that $\calX \subset \hat{B}_1^N$.  Note that the $\calK$
    operator in Corollary~\ref{cor:bnet-para} is defined for frequency
    window, c.f. beginning of Section~\ref{sec:approx-fourier}, and
    thus $\hat{x}|_{I_K}$ equals $\calK \vec{x}$.  Combined with
    that $\hat{x}|_{I_K} \in \hat{B}_1^K$, both $\hat{x}|_{I_K}$
    and $ \calB (\vec{x})$ must lie in $\hat{B}_{1.1}^K$ on which
    $g_K$ and $\phi_K$ are defined.  We now bound the three terms in
    \eqref{eq:pf3term} respectively.
    
    The 2nd term: By that $\hat{x}|_{I_K}, \, \calB (\vec{x}) \in
    \hat{B}_{1.1}^K$, Assumption~\ref{assump:f} (ii) gives that
    \begin{equation*}
        |g_K(\hat{x}|_{I_K}) - g_K( \calB (\vec{x}) )| \le C_2 \|
        \hat{x}|_{I_K} - \calB (\vec{x}) \|_2 < C_2 \varepsilon,
    \end{equation*}
    where the second inequality uses \eqref{eq:calB-uniform-bound}.
    This pointwise upper bound leads to the $p$-norm bound of this
    term to be $C_2 \varepsilon$ for all $p$.
    
    The 3rd term: By the uniform approximation of $g_K$ on
    $\hat{B}_{1.1}^K$ guaranteed by Assumption~\ref{assump:f} (iii),
    the fully-connected network $\phi_K$ satisfying the requirement
    exists, and pointwisely,
    \begin{equation*}
        |g_K(\calB (\vec{x})) - \phi_K(\calB (\vec{x}))| < \varepsilon.
    \end{equation*}
    This proves that the $p$-norm bound of this term is smaller than
    $\varepsilon$ for all $p$.
    
    The 1st term: the $p$-norm is at most $\varepsilon$ by assumption.
    
    Putting together proves the $p$-norm upper bound in the claim.
\end{proof}

\begin{remark}
    When the $p$-norm residual in Assumption~\ref{assump:f} (i) can
    be made zero, which includes the case of band-limited $f$ and
    Example 2, the $\varepsilon$ in the Theorem~\ref{thm:approx_bnet}
    can be chosen arbitrarily close to zero and the final approximation
    bound can be made $\varepsilon(1+C_2)$.
\end{remark}

Note that the above results demonstrate the power of \NetName{} when
$K < N$. Assume the function $f$ and the frequency truncated $g_K$
has similar regularity level, which is typically true as shown in both
Example 1 and Example 2. Then a fully-connected network approximating
the original function $f: \calX \to \mathbb{R}$,  requires a total
number of parameter,
\begin{equation*}
    n_{fc} \le c(s, N) \varepsilon^{-N/s} (\log
    \frac{1}{\varepsilon}+1).
\end{equation*}
As a comparison, the overall number of parameter for \NetName{}
with a fully-connected network is,
\begin{equation*}
    n_b + n_{fc} \le C K \log N \left( \log \frac{1}{\varepsilon}
    + \log N \right)^2 + c(s, K) \varepsilon^{-K/s} (\log
    \frac{1}{\varepsilon}+1),
\end{equation*}
where the second term which involves $\varepsilon^{-K/s}$ dominates
the model complexity.  The improvement from $\varepsilon^{-N/s}$ to
$\varepsilon^{-K/s}$ is due to the feature extraction of truncated
Fourier coefficients which are suitable for the class of functions
as described in Assumption~\ref{assump:f}. In other words, while $N$
being the ambient dimensionality of the input data, $K$ upper-bounds
the intrinsic complexity (effective dimension) of the regression
problem $f$ on the data $\calX$.

{\bf Discussion on the dense prediction networks.} The method of
analyzing $x$-to-1 network functions in this section can extend
to the analysis of $x$-to-$x$ function mappings, known as dense
prediction in deep learning literature. Typical deep convolutional
networks for $x$-to-$x$ prediction include the U-net architecture
\cite{Ronneberger2015}, which consists of a module of multiple
convolutional layers, a bottleneck module with dense connections,
and another module of multiple conv-t layers, namely transpose
convolutional layers. The corresponding network architecture using
\NetName{} would be $\calB$-(fc layers)-$\calB^T$ where $\calB$
denotes a module of \NetName{} layers, and $\calB^T$ is again a
\NetName{} module due to the symmetric role of time of frequency in the
model. Extending the approximation theory, such U-net which replaces
traditional convolutional layers by \NetName{} layers can provably
approximate operators in the form of ${\cal F}^{-1} \circ \sigma \circ
{\cal F}$, where $\cal F$ is Fourier transform operator, and $\sigma$
is some non-linear operator (not necessarily entry-wise). The covered
family of operators involve low/high-pass filtering, de-convolution,
among others, in signal processing.

Another example of dense prediction mapping is the general Laplace
operator (with variable coefficient) in physics which can be
represented as $\sum_i \calA_i \mathcal{K} \calD_i \calF $, where
$i$ sums over a small number of terms corresponding to a low-rank
decomposition of the amplitude function, $\calA_i$s and $\calD_i$s
are diagonal matrices, and $\calK$ is a smooth Fourier integral
operator (FIO)~\cite{Candes2009, Demanet2011}. As a heads-up, a
parallel reading of the proof of Theorem~\ref{thm:matrix-approx}
and \cite{Candes2009, Li2017} reveals that similar theorem can be
provided for smooth FIOs, i.e., $\Kfun{\Phi(\xi,t)}$ with smooth
$\Phi(\xi,t)$ satisfying homogeneity condition of degree 1.  Such an
extension of Theorem~\ref{thm:matrix-approx} will give error control of
\NetName{} approximation of the operator $\mathcal{K}$, and $\calA_i$s,
$\calD_i$s can be represented by fully-connected networks, possibly
coodinate-separated and shallow. The approximation to smooth FIOs
will enable the usages of \NetName{} to represent a large class of
elliptic operators.

\subsection{Approximation of the Fourier Kernel}
\label{sec:approx-fourier}

In this section, we analyze the approximation power of the \NetName{}
on discrete Fourier kernel, whose matrix entry is defined as
$\calK_{ij} = \Kfun{\xi_i \cdot t_j}$ where $t_j$ and $\xi_i$ are
uniformly distributed on $[0,1)$ and $[k_0, k_0+K)$ ($K \leq N$)
respectively. The analysis result shows that though \NetName{} by
construction has very low complexity as the number of parameters
is on the order of the input/output data size, it exhibits full
approximation power in terms of function representations.

\begin{theorem} \label{thm:matrix-approx}
    Let $N$ denote the size of the input and $K$ denote the size
    of the output in the \NetName{}. $L$ and $r$ are two parameters
    such that $\pi e K \leq r 2^{\min(L, \log K)}$. $L$ is the depth
    of the \NetName{} and $r$ is the size of mixing channels. There
    exists a parametrized \NetName{}, $\calB(\cdot)$, approximating
    the discrete Fourier transform such that for any bounded input
    vector $\vec{x}$, the error of the output of the \NetName{}
    satisfies that for any $p \in [1, \infty]$
    \begin{equation}
        \norm{ \calK \vec{x} - \calB(\vec{x}) }_p \leq m^{\frac{1}{p}}
        r^{L_t (1-\frac{1}{p}) + L_\xi \frac{1}{p}+1} \left( 2\Lambda_r
        \right)^{L+3} \frac{ K(\pi e)^r }{r^{r-1}} \norm{\vec{x}}_p,
    \end{equation}
    where $m = \min(1, K/2^L)$ and $\Lambda_r = \LebConst{}$ is the
    Lebesgue constant.

    If $L \leq \log K$, then the error also satisfies,
    \begin{equation}
        \norm{ \calK \vec{x} - \calB(\vec{x}) }_p \leq C_{r,K} \left(
        \frac{\Lambda_r}{2^{r-2}} \right)^L r^{L_t (1-\frac{1}{p})
        + L_\xi \frac{1}{p}} \norm{\vec{x}}_p,
    \end{equation}
    where $C_{r,K} = (2\Lambda_r)^3 \frac{(\pi eK)^r}{(2r)^{r-1}}$
    is a constant depending on $r$ and $K$, and independent of $L$.
\end{theorem}

The proof of Theorem~\ref{thm:matrix-approx} is constructive. We first
fill the \NetName{} with a specific set of parameters ({\it Butterfly}
initialization) based on the complementary low-rank property of the
discrete Fourier kernel (see Theorem~\ref{thm:low-rank}).  Using the
matrix representation of \NetName{}, $1$-norm and $\infty$-norm of each
matrix can be bounded. Combined with the low-rank approximation error,
we derive the $1$-norm and $\infty$-norm upper bound for \NetName{}.
Applying Riese-Thorin interpolation theorem, we reach to the conclusion
of Theorem~\ref{thm:matrix-approx} for general norm index $p$.
Section~\ref{sec:proof} provides the detailed proof of the theorem.

Previously, in the context of fast algorithms, \citet{Kunis2012}
analyzed the approximations of a simplified {\it Butterfly} scheme
and \citet{Demanet2012a} analyzed general {\it Butterfly} scheme
under different error measures on the input and output. While as a
side product of our proof, we also obtain the error estimate of the
matrix approximation of the general {\it Butterfly} schemes in terms
of matrix norms.

For a problem with fixed input and output size, we can tune two
parameters $r$ and $L$ to reach desired accuracy. As $r$ increases,
which corresponds the increase of mixing channel size in each layer,
the approximation error decays mainly as $r^{-r}$. Interestingly,
when $L$ increases, which corresponds to increase the depth of the
\NetName{}, the error bound decays exponentially in $L$.

Combining Theorem~\ref{thm:matrix-approx} and the parameter number
estimation \eqref{eq:bnet-complexity-precise}, we derive the network
complexity analysis for \NetName{} under a given approximation accuracy
as follows.

\begin{corollary} \label{cor:bnet-para}
    Let $N$ denote the size of the input and $K$ denote the size
    of the output in the \NetName{}. For any $0 < \varepsilon < 1$,
    there exists a family of \NetName{}s with $L_\xi = 0, 1, \dots,
    \log K$, whose numbers of parameters are all bounded by
    \begin{equation}\label{eq:bound-nb}
        n_b \leq C K \log N \left( \log \frac{1}{\varepsilon} +
        \log N \right)^2,
    \end{equation}
    where $C < (2\pi e)^2 \cdot 90$ is an absolute constant, and for
    any $\calB$ denoting a \NetName{} in the family,
    \begin{equation}
        \norm{ \calK \vec{x} - \calB(\vec{x}) }_2 <
        \varepsilon \norm{\vec{x}}_2,
    \end{equation}
    for any input vector $\vec{x}$.
\end{corollary}

\begin{proof}
    Comparing the statements in Corollary~\ref{cor:bnet-para} and
    Theorem~\ref{thm:matrix-approx}, we aim to find a $r$ such that
    \begin{equation}
        T = \sqrt{\min(1, \frac{K}{2^L})} \sqrt{r}^{L+2} \left( 2 +
        \frac{4}{\pi} \ln r \right)^{L+3} \frac{ K(\pi e)^r }{r^{r-1}}
        < \varepsilon.
    \end{equation}
    In this proof, we assume that $L = \log N$.  Since we focus
    on the approximation error under $p=2$ norm, the prefactor $T$
    is then independent of $L_\xi$ and $r$ becomes the only tunable
    parameter in \NetName{}.  The following proof holds for all $0
    \leq L_\xi \leq \log K$.
    
    When $r \geq \alpha (\log N + \log \frac{1}{\varepsilon})$, we have,
    \begin{equation} \label{eq:boundeps}
        T \leq \left( \frac{ 2 \sqrt{r} \left( 2 + \frac{4}{\pi}
        \ln r \right)}{\left(\frac{r}{\pi e}\right)^\alpha} \right)^L
         \frac{ r^2 \left( 2 + \frac{4}{\pi} \ln r \right)^{3}
        }{\left( \frac{r}{\pi e} \right)^{\alpha-1}} \frac{ 1
        }{\left(\frac{r}{\pi e}\right)^{\log \frac{1}{\varepsilon}}}\,.
    \end{equation}
    For both the first and second term in \eqref{eq:boundeps} the
    denominator grows faster in $r$ than the numerator. When $\alpha =
    2\pi e$, both the first and second terms are smaller than one for
    all $r \geq \alpha$~\footnote{Both inequalities are numerically
    validated as well.}. Then we have when $r \geq 2\pi e (\log N +
    \log \frac{1}{\varepsilon})$,
    \begin{equation} 
        T < \frac{ 1 }{\left(\frac{r}{\pi e}\right)^{\log
        \frac{1}{\varepsilon}}} < \frac{ 1 }{2^{\log
        \frac{1}{\varepsilon}}} = \varepsilon.
    \end{equation}

    Based on the network complexity counting in
    \eqref{eq:bnet-complexity-precise}, there exists an absolute
    constant $C_0 = 90$, and $r$ such that the number of parameters
    in \NetName{} is
    \begin{equation} \label{eq:bnet-complexity-bound}
        n_b \leq C_0 K \log N r^2.
    \end{equation}
    Substituting the conditions for $r$ into
    \eqref{eq:bnet-complexity-bound}, we conclude that there exists
    a \NetName{} with the number of parameters bounded by
    \begin{equation}
        n_b \leq C K \log N \left( \log \frac{1}{\varepsilon} +
        \log N\right)^2, \quad C \leq (2\pi e)^2 C_0,
    \end{equation}
    such that $T < \varepsilon$, which proves the theorem.
\end{proof}

\subsection{Proof of Theorem~\ref{thm:matrix-approx}}
\label{sec:proof}

This section first provides a few lemmas and their proof bounding each
sparse matrix in \eqref{eq:butterflyNet}. And then
Theorem~\ref{thm:matrix-approx} is proved in detail.

\begin{lemma} \label{lem:Lebesgue}
    Let $\{z_i\}_{i=1}^{r}$ be $r$ Chebyshev points and $\calL_k(x)$ be
    the Lagrange polynomial of order $r$. For any $r$, the Lebesgue
    constant $\Lambda_r$ is bounded as
    \begin{equation*}
        \Lambda_r = \max_{-1\leq x \leq 1} \sum_{i=1}^r \abs{\calL_i(x)}
        \leq \frac{2}{\pi} \ln r + 1.
    \end{equation*}
\end{lemma}

Lemma~\ref{lem:Lebesgue} is a standard result of Chebyshev interpolation
and the proof can be found in \cite{Rivlin1990}.

\begin{corollary} \label{cor:Lebesgue}
    Let $\{z_i\}_{i=1}^{r}$ be $r$ Chebyshev points and $\calL_k(x)$ be
    the Lagrange polynomial of order $r$. For any $r$ and $i \leq r$,
    \begin{equation*}
        \max_{-1\leq x \leq 1} \abs{\calL_i(x)}
        \leq \frac{2}{\pi} \ln r + 1.
    \end{equation*}
\end{corollary}

Corollary~\ref{cor:Lebesgue} is an immediate result of
Lemma~\ref{lem:Lebesgue}.

\begin{lemma} \label{lem:U}
    Let $U$ be the block diagonal matrix defined at interpolation
    layer $(\ell = L)$, then
    \begin{equation*}
        \norm{U}_1 \leq m \left( \LebConst{} \right) \text{ and }
        \norm{U}_\infty \leq \LebConst{}
    \end{equation*}
    where $r$ is the number of Chebyshev points and $m = K/2^L$.
\end{lemma}

\begin{proof}
    $U$ is a block diagonal matrix with block $U_i$. $U_i$ are the
    same $m \times r$ with entry $\Kfun{ (\xi_p-\xi_k)\cdot t_0}
    \calL_k(\xi_p)$ for $\xi_p$ and $\xi_k$ being uniform and Chebyshev
    points in $A^{L}_i$.  By the definition of matrix 1-norm, we have
    \begin{equation}
        \norm{U}_1 = \norm{U_1}_1 \leq \max_{\xi_k \in A^L_1}
        \sum_{\xi_p \in A^L_1} \abs{ \Kfun{ (\xi_p-\xi_k) \cdot t_0}
        \calL_k(\xi_p) } \leq m \left( \LebConst{} \right).
    \end{equation}
    By the definition of matrix $\infty$-norm, we have
    \begin{equation}
        \norm{U}_\infty = \norm{U_1}_\infty \leq \max_{\xi_p \in A^L_1}
        \sum_{\xi_k \in A^L_1} \abs{ \Kfun{ (\xi_p-\xi_k) \cdot t_0}
        \calL_k(\xi_p) } \leq \LebConst{}.
    \end{equation}
\end{proof}

\begin{lemma} \label{lem:M}
    Let $M$ be the block matrix defined at switch layer, then
    \begin{equation*}
        \norm{M}_1 \leq r \text{ and } \norm{M}_\infty \leq r,
    \end{equation*}
    where $r$ is the number of Chebyshev points.
\end{lemma}

\begin{proof}
    Based on the structure of $M$ and the definition of
    matrix 1-norm, we have
    \begin{equation}
        \norm{M}_1 \leq \max_{j} \max_i \norm{W_{j,i}}_1 = \max_{j,i}
        \max_{t^{B^{L_\xi}_j}_k} \sum_{\xi_{k'}^{A^{L_t}_i}} \abs{
        \Kfun{ \xi_{k'}^{A^{L_t}_i} \cdot t^{B^{L_\xi}_j}_k} } = r.
    \end{equation}
    Based on the structure of $M$ and the definition of
    matrix $\infty$-norm, we have
    \begin{equation}
        \norm{M}_\infty \leq \max_j \max_i \norm{W_{j,i}}_\infty =
        \max_{j,i} \max_{t^{A^{L_t}_i}_k} \sum_{t_{k}^{B^{L_\xi}_j}}
        \abs{ \Kfun{ \xi_{k'}^{A^{L_t}_i} \cdot t^{B^{L_\xi}_j}_k}
        } = r.
    \end{equation}
\end{proof}

\begin{lemma} \label{lem:H1}
    Let $H^{(\ell)}$ be the block diagonal matrix defined at recursion
    layer $\ell = 1, \dots, L_t$, then
    \begin{equation*}
        \norm{H^{(\ell)}}_1 \leq 2\left( \LebConst{} \right),
    \end{equation*}
    where $r$ is the number of Chebyshev points.
\end{lemma}

\begin{proof}
    The building block of $H^{(\ell)}$ is $W^{(\ell)}_i$, whose 1-norm
    is bounded as,
    \begin{equation}
        \norm{W^{(\ell)}_i}_1 \leq \max_{t \in B^{L-\ell}_1}
        \sum_{t_k^{B^{L-\ell}_1}} \abs{ \Kfun{\xi_0^{A^\ell_i}
        \cdot (t-t^{B^{L-\ell}_1}_k)}\calL_k(t) } \leq \LebConst{}.
    \end{equation}

    For two ranges of $\ell$, i.e., $\ell \leq L_{\min}$ and $L_{\min} <
    \ell \leq L_t$, we have,
    \begin{equation}
        \norm{H^{\ell}}_1 = \max_i \norm{H^{(\ell)}_i}_1
        \leq 2 \norm{W^{(\ell)}_i}_1
        \leq 2 \left( \LebConst{} \right).
    \end{equation}
\end{proof}

\begin{lemma} \label{lem:G1}
    Let $G^{(\ell)}$ be the block diagonal matrix defined at recursion
    layer $\ell = L_t+1, \dots, L_t + L_\xi$, then
    \begin{equation*}
        \norm{G^{(\ell)}}_1 \leq 2r \left( \LebConst{} \right),
    \end{equation*}
    where $r$ is the number of Chebyshev points.
\end{lemma}

\begin{proof}
    Based on the structure of $G^{(\ell)}$ and the definition of
    matrix 1-norm, we have
    \begin{equation}
        \begin{split}
            \norm{G^{\ell}}_1 &  = \max_j \norm{G^{(\ell)}_j}_1
            \leq \max_{j} \norm{W^{(\ell)}_{j}}_1 \\
            & \leq \max_j \max_{k}  \left(
            \sum_{\xi_{k'}^{A^{\ell-1}_1}} \abs{
            \calL_k(\xi_{k'}^{A^{\ell-1}_1}) }
            + \sum_{\xi_{k'}^{A^{\ell-1}_2}} \abs{
            \calL_k(\xi_{k'}^{A^{\ell-1}_2}) } \right) \leq 2r \left(
            \LebConst{} \right). \\
        \end{split}
    \end{equation}
\end{proof}

\begin{lemma} \label{lem:Hinf}
    Let $H^{(\ell)}$ be the block diagonal matrix defined at recursion
    layer $\ell = 1, \dots, L_t$, then
    \begin{equation*}
        \norm{H^{(\ell)}}_\infty \leq 2r \left( \LebConst{} \right),
    \end{equation*}
    where $r$ is the number of Chebyshev points.
\end{lemma}

\begin{lemma} \label{lem:Ginf}
    Let $G^{(\ell)}$ be the block diagonal matrix defined at recursion
    layer $\ell = L_t + 1, \dots, L_t + L_\xi$, then
    \begin{equation*}
        \norm{G^{(\ell)}}_\infty \leq 2 \left( \LebConst{} \right),
    \end{equation*}
    where $r$ is the number of Chebyshev points.
\end{lemma}

The proofs of Lemma~\ref{lem:Hinf} and Lemma~\ref{lem:Ginf} follow that
of Lemma~\ref{lem:G1} and Lemma~\ref{lem:H1} respectively.

\begin{proof} [Proof of Theorem~\ref{thm:matrix-approx}]
    Since all bias weights are initialized as zero, and extensive assign
    is applied together with ReLU activation, \NetName{} with {\it
    Butterfly} initialization is equivalent to multiplying the matrices
    in matrix representation together, i.e.,
    \begin{equation}
        \calB(x) = U G^{(L)} \cdot G^{(L_t+1)} M H^{(L_t)} \cdot H^{(1)}
        Vx.
    \end{equation}
    We then write the exact matrix product with error matrices
    $E^{(\ell)}$,
    \begin{equation}
        \calK = E^{(L+1)} + U \left[ E^{(L)} + G^{(L)} \left[ \cdots +
        G^{(L_t+1)} \left[ E^{(S)} + M \left[ E^{(L_t)} + H^{(L_t)}
        \left[ \cdots + H^{(1)} \left[ E^{(0)} + V \right] \right]
        \right] \right] \right] \right]
    \end{equation}
    Then, following properties of matrix $p$-norm, we have,
    \begin{equation} \label{eq:error-long}
        \begin{split}
            \varepsilon_p & = \norm{ \calK - U G^{(L)} \cdots G^{(L_t+1)}
            M H^{(L_t)} \cdots H^{(1)} V }_p \\
            & \leq  \norm{E^{(L+1)}}_p + \norm{U E^{(L)}}_p + \cdots +
            \norm{U G^{(L)} \cdots G^{(L_t+1)} E^{(S)}}_p + \cdots \\
            & \qquad + \norm{U G^{(L)} \cdots G^{(L_t+1)} M H^{(L_t)} \cdots
            H^{(1)} E^{(0)}}_p \\
            & \leq \norm{E^{(L+1)}}_p + \norm{U}_p \norm{E^{(L)}}_p
            + \cdots + \norm{U}_p \left( \prod_{\ell =
            L_t+1}^L\norm{G^{(\ell)}}_p \right) \norm{E^{(S)}}_p +
            \cdots \\
            & \qquad + \norm{U}_p \left( \prod_{\ell=L_1+1}^L
            \norm{G^{(\ell)}}_p \right) \norm{M}_p \left( \prod_{\ell =
            1}^{L_t} \norm{H^{(\ell)}}_p \right) \norm{E^{(0)}}_p 
        \end{split}
    \end{equation}
    for $1 \leq p \leq \infty$.

    When matrix 1-norm is used, we adopt Lemma~\ref{lem:U},
    Lemma~\ref{lem:M}, Lemma~\ref{lem:H1}, and Lemma~\ref{lem:G1}
    to bound $\varepsilon_1$ as,
    \begin{equation} \label{eq:error-result-1}
        \varepsilon_1 \leq \delta_1 \left( 1 + m \Lambda_r
        \sum_{\ell = 0}^{L_\xi} (2r\Lambda_r)^\ell + m r\Lambda_r
        (2r\Lambda_r)^{L_\xi} \sum_{\ell = 0}^{L_t} (2\Lambda_r)^\ell
        \right) \leq m r^{L_\xi+1} (2 \Lambda_r)^{L+2} \delta_1,
    \end{equation}
    where the second inequality adopt the fact that $r>1$ and
    $\Lambda_r = \LebConst() > 1$, $\delta_1$ is a uniform upper bound
    for $\norm{E^{(\ell)}}_1$.  Theorem~\ref{thm:low-rank} provides an
    upper bound for each entry of $E^{(\ell)}$. Recalling the vector
    length ($\sharp \{i\} \cdot \sharp\{j\} \cdot r$) and product of
    domain lengths from Table~\ref{tab:complementary-domain-pairs},
    the uniform upper bound for the 1-norm follows,
    \begin{equation}
        \delta_1 = \max_\ell \norm{E^{(\ell)}_1}_1 \leq \pi e K \left(
        1 + \Lambda_r \right) \left( \frac{ \pi e K }{ r 2^{L_\xi +
        L_{\min} + 1}} \right)^{r-1},
    \end{equation}
    where the assumption $\pi e K \leq r 2^{\min(\log K, L)}$ is
    applied.

    When matrix $\infty$-norm is used, we adopt Lemma~\ref{lem:U},
    Lemma~\ref{lem:M}, Lemma~\ref{lem:Hinf}, and Lemma~\ref{lem:Ginf}
    to bound $\varepsilon_\infty$ as,
    \begin{equation} \label{eq:error-result-inf}
        \varepsilon_\infty \leq \delta_\infty \left( 1 + \Lambda_r
        \sum_{\ell = 0}^{L_\xi} (2\Lambda_r)^\ell + r\Lambda_r
        (2\Lambda_r)^{L_\xi} \sum_{\ell = 0}^{L_t} (2r\Lambda_r)^\ell
        \right) \leq r^{L_t+1} (2 \Lambda_r)^{L+2} \delta_\infty,
    \end{equation}
    where $\delta_\infty$ is a uniform upper bound for
    $\norm{E^{(\ell)}}_\infty$ and obeys the same upper bound as
    $\delta_1$.

    Applying Riesz-Thorin interpolation theorem together with
    \eqref{eq:error-result-1} and \eqref{eq:error-result-inf}, we obtain
    the error bound under the matrix $p$-norm,
    \begin{equation} \label{eq:matrix-approx-r}
        \varepsilon_p \leq m^{\frac{1}{p}} r^{L_t (1-\frac{1}{p})
        + L_\xi \frac{1}{p}+1} \left(2\Lambda_r\right)^{L+3} \frac{
        K(\pi e)^r }{r^{r-1}},
    \end{equation}
    for any $1 \leq p \leq \infty$.

    If we further assume $L \leq \log K$, then $2^{L_\xi + L_{\min}} =
    2^L$ and the matrix $p$-norm error bound can be rederived in
    terms of $L$, i.e.,
    \begin{equation} \label{eq:matrix-approx-L}
        \varepsilon_p \leq C_{r,K} \left( \frac{\Lambda_r}{2^{r-2}}
        \right)^L r^{L_t (1-\frac{1}{p})
        + L_\xi \frac{1}{p}},
    \end{equation}
    where $C_{r,K} = (2\Lambda_r)^3 \frac{(\pi eK)^r}{(2r)^{r-1}}$
    is a constant depending on $r$ and $K$, and independent of $L$.

    For any bounded vector $\vec{f}$, Theorem~\ref{thm:matrix-approx}
    is the direct result of \eqref{eq:matrix-approx-r} and
    \eqref{eq:matrix-approx-L}.
\end{proof}

\section{Numerical Results}
\label{sec:experiment}

We present four numerical experiments to demonstrate the approximation
power with or without training for \NetName{} and \INetName{}. The
first numerical experiment shows that the approximation error
without training of an {\it Butterfly} initialized \NetName{} decays
exponentially as the increases of the network depth $L$, which verifies
Theorem~\ref{thm:matrix-approx}. Then, through the second experiment,
we show that the training of \NetName{} and \INetName{} further refine
the approximation error. The approximation error after training
depends on both the properties of the function and the dataset. In
the third experiment, we apply both \NetName{}s and \INetName{}s to
testing datasets with different distribution comparing to the training
datasets and compare the transfer learning capabilities.  In the last
experiment, \NetName{} and \INetName{} with additional task layers,
one with single fully connect layer and another with square layer,
are tested and compared in the approximation of the energy functionals
of Poisson's equation. All algorithms are implemented using Tensorflow
2.1.0~\cite{tensorflow} and can be found on the authors' homepages.

\subsection{Approximation Power Without Training}

The first numerical experiment in this section aims to verify the
exponential decay of the approximation error of the \NetName{} as
the depth $L$ increases. We construct a \NetName{} to approximate
the discrete Fourier kernel with fixed number of Chebyshev points,
$r = 8$. The \NetName{} is filled with {\it Butterfly} initialization
weights. The input vector in this example is of size $N = 1024$ and
various output vector sizes are tested. The output vector represents
integer frequency of the input function in the frequency domain
$[0, K)$. The approximation error of the \NetName{} is measured
against the dense discrete Fourier kernel matrix and relative matrix
$p$-norm error is reported, $\epsilon_p = \nicefrac{\norm{\calK -
\calB}_p}{\norm{\calK}_p}$, where $\calB$ denotes \NetName{}.

\begin{table}[htp]
    \centering
    \begin{tabular}{cccccccccc}
        \toprule
        \multicolumn{5}{c}{$K = 64$}
        & \multicolumn{5}{c}{$K=256$} \\
        \cmidrule(r){1-5}
        \cmidrule(r){6-10}
        $L$ & $L_\xi$ & $\epsilon_1$ & $\epsilon_2$ &
        $\epsilon_\infty$ &
        $L$ & $L_\xi$ & $\epsilon_1$ & $\epsilon_2$ &
        $\epsilon_\infty$ \\
        \toprule
        \multirow{3}{*}{4}
        &  1 & 2.06\np{e-}1 & 2.46\np{e-}1 & 2.56\np{e-}1 
        &
        \multirow{3}{*}{6}
        &  1 & 2.52\np{e-}1 & 3.40\np{e-}1 & 2.82\np{e-}1 
        \\
        &  2 & 2.02\np{e-}1 & 2.60\np{e-}1 & 2.66\np{e-}1 
        &
        &  2 & 2.51\np{e-}1 & 3.45\np{e-}1 & 2.89\np{e-}1 
        \\
        &  3 & 1.90\np{e-}1 & 2.89\np{e-}1 & 2.72\np{e-}1 
        &
        &  3 & 2.46\np{e-}1 & 3.60\np{e-}1 & 2.95\np{e-}1 
        \\
        \midrule
        \multirow{3}{*}{5}
        &  1 & 1.79\np{e-}3 & 2.56\np{e-}3 & 2.31\np{e-}3 
        &
        \multirow{3}{*}{7}
        &  1 & 2.03\np{e-}3 & 3.40\np{e-}3 & 2.44\np{e-}3 
        \\
        &  2 & 1.69\np{e-}3 & 2.32\np{e-}3 & 1.84\np{e-}3 
        &
        &  2 & 1.97\np{e-}3 & 3.33\np{e-}3 & 2.01\np{e-}3 
        \\
        &  3 & 1.61\np{e-}3 & 2.16\np{e-}3 & 1.94\np{e-}3 
        &
        &  3 & 1.91\np{e-}3 & 3.15\np{e-}3 & 2.11\np{e-}3 
        \\
        \midrule
        \multirow{3}{*}{6}
        &  1 & 9.21\np{e-}6 & 1.30\np{e-}5 & 1.94\np{e-}5 
        &
        \multirow{3}{*}{8}
        &  1 & 1.15\np{e-}5 & 2.01\np{e-}5 & 2.00\np{e-}5 
        \\
        &  2 & 8.90\np{e-}6 & 1.33\np{e-}5 & 1.76\np{e-}5 
        &
        &  2 & 1.13\np{e-}5 & 2.04\np{e-}5 & 1.82\np{e-}5 
        \\
        &  3 & 8.65\np{e-}6 & 1.49\np{e-}5 & 1.70\np{e-}5 
        &
        &  3 & 1.10\np{e-}5 & 2.07\np{e-}5 & 1.77\np{e-}5 
        \\
        \bottomrule
    \end{tabular}
    \caption{Approximation accuracy of the Fourier kernel by the
    \NetName{}. The input vector size is $N=1024$, the number of
    Chebyshev points is $r=8$, and frequency domain is $[0,K)$.}
    \label{tab:DFT-result}
\end{table}

Table~\ref{tab:DFT-result} shows for both choices of $K$, the relative
approximation errors measured in 1-norm, 2-norm, and $\infty$-norm
decay exponentially as $L$ increases and stay constant for different
$L_\xi$. The decay factors for different $K$ remain similar, while
the prefactor is larger for large $K$. All of these observations
agree with the error bound in Theorem~\ref{thm:matrix-approx}.

\subsection{Approximation Power After Training}
\label{sec:numerical-approx-power}

The second numerical experiment in this
section aims to demonstrate the approximation power after training
of \NetName{}.

{\bf Dataset setup.} Both the training and testing datasets are
generated as follows. We first generate an array of $N$ random complex
number in the frequency domain with each number sampled uniformly
from $[-1,1)$. Then, we multiply the array by a Gaussian function
centered at $G_{\mathrm{center}}$ and width $G_{\mathrm{width}}$
(the standard deviation of the Gaussian function). The input data
$\vec{x}$ is then the real part of the inverse discrete Fourier
transform of the array.  Given a frequency window $[k_0, k_0+K)$,
the output data is the discrete Fourier transform of $\vec{x}$
restricted to the frequency window.

\begin{table}[htp]
    \centering
    \begin{tabular}{lccc}
        \toprule
        Short Name & $G_{\mathrm{center}}$ & $G_{\mathrm{width}}$ &
        Freq Window \\
        \toprule
        DFT-Lfreq & 0 & 500 & $[0, 128)$ \\
        DFT-Hfreq & 0 & 500 & $[256, 384)$ \\
        DFTSmooth-Lfreq & 0 & 10 & $[0, 128)$ \\
        DFTSmooth-Hfreq & 256 & 10 & $[256, 384)$ \\
        \bottomrule
    \end{tabular}
    \caption{Datasets setups and their short names.}
    \label{tab:datasets}
\end{table}

We perform numerical results on four groups of datasets as shown
in Table~\ref{tab:datasets} with their short names.  Notice that
DFT-Lfreq and DFT-Hfreq have the same input data distribution and
the corresponding widths are very large.  Hence the input data of
these two groups are close to white noise. However, DFTSmooth-Lfreq
and DFTSmooth-Hfreq have input data generated with Gaussian
centered at $0$ and $256$ with small width $10$. The input data
is then close to band limited signal around the given frequency
window, and in other words, data in dimension $N$ actually lie on
a low-dimensional subspace of dimension about $6G_{\mathrm{width}}
= 60$.  In Appendix~\ref{app:dataset-example}, we include one instance
for each datasets in Table~\ref{tab:datasets}.

{\bf Training and evaluation setup.} All \NetName{}s and \INetName{}s
with different initializations and frequency windows are trained under
the infinity data setting, i.e., training data is randomly generated on
the fly. The input data length is $N=1024$, the batch size is $256$,
the maximum number of iteration is $50,000$, and ADAM optimizer is
used with an exponentially decay learning rate. The initial learning
rate is $10^{-3}$ and $10^{-4}$ for random initialized neural networks
and {\it Butterfly} initialized ones respectively. The decay steps
and the decay rate are 100 and 0.985. The maximum number of iteration
is sufficient for the convergence of relative errors in all settings
(see Appendix~\ref{app:conv-behaviors} for examples of convergence
behaviors). The loss function is defined as,
\begin{equation}
    \ell (\{\vec{x}_i, \vec{y}_i\}) = \sum_{i} \norm{\calN(\vec{x}_i) -
    \vec{y}_i}_2^2,
\end{equation}
where $\vec{y}_i$ is the output data and $\calN$ denotes a neural
network. Relative errors are reported for comparison.  In the
following, the pre-training relative error is evaluated on the first
batch and the post-training relative error is evaluated on a testing
data of size 1000. Default values are used for other unspecified
hyper parameters.

\begin{table}[t]
    \centering
    \small
    \begin{tabular}{cccrcccccccc}
        \toprule
        & & & &
        \multicolumn{2}{c}{DFT-Lfreq} &
        \multicolumn{2}{c}{DFT-Hfreq} &
        \multicolumn{2}{c}{DFTSmooth-Lfreq} &
        \multicolumn{2}{c}{DFTSmooth-Hfreq} \\
        \cmidrule(lr){5-6}
        \cmidrule(lr){7-8}
        \cmidrule(lr){9-10}
        \cmidrule(lr){11-12}
        $L_\xi$ &
        \shortstack{Neural\\Network} &
        Initial &
        \shortstack{Num\\Paras} &
        \shortstack{Pre\\Train} &
        \shortstack{Post\\Train} &
        \shortstack{Pre\\Train} &
        \shortstack{Post\\Train} &
        \shortstack{Pre\\Train} &
        \shortstack{Post\\Train} &
        \shortstack{Pre\\Train} &
        \shortstack{Post\\Train} \\
        \toprule
        \multirow{4}{*}{1} &\multirow{2}{*}{\sNetName} & prefix &
          136304 
        & 1.9\np{e-}2 & 1.6\np{e-}4 
        & 1.9\np{e-}2 & 2.0\np{e-}4 
        & 1.9\np{e-}2 & 1.2\np{e-}5 
        & 2.0\np{e-}2 & 3.0\np{e-}5 
        \\
        & & random &
          136304 
        & 1.0\np{e-}0 & 1.7\np{e-}2 
        & 1.0\np{e-}0 & 2.0\np{e-}2 
        & 1.0\np{e-}0 & 8.8\np{e-}3 
        & 1.0\np{e-}0 & 1.3\np{e-}2 
        \\
        & \multirow{2}{*}{\sINetName} & prefix &
         3533936 
        & 1.9\np{e-}2 & 5.5\np{e-}5 
        & 1.9\np{e-}2 & 6.5\np{e-}5 
        & 1.9\np{e-}2 & 1.5\np{e-}4 
        & 1.9\np{e-}2 & 8.5\np{e-}5 
        \\
        & & random &
         3533936 
        & 1.0\np{e-}0 & 6.9\np{e-}1 
        & 1.0\np{e-}0 & 6.1\np{e-}1 
        & 1.0\np{e-}0 & 3.2\np{e-}1 
        & 1.0\np{e-}0 & 2.3\np{e-}1 
        \\
        \midrule
        \multirow{4}{*}{2} &\multirow{2}{*}{\sNetName} & prefix &
           87728 
        & 1.9\np{e-}2 & 8.4\np{e-}4 
        & 2.0\np{e-}2 & 9.7\np{e-}4 
        & 2.0\np{e-}2 & 1.4\np{e-}4 
        & 2.0\np{e-}2 & 4.9\np{e-}4 
        \\
        & & random &
           87728 
        & 1.0\np{e-}0 & 8.2\np{e-}2 
        & 1.0\np{e-}0 & 9.1\np{e-}2 
        & 1.0\np{e-}0 & 1.5\np{e-}2 
        & 1.0\np{e-}0 & 2.8\np{e-}2 
        \\
        & \multirow{2}{*}{\sINetName} & prefix &
          915120 
        & 1.9\np{e-}2 & 3.0\np{e-}4 
        & 1.9\np{e-}2 & 4.2\np{e-}4 
        & 2.0\np{e-}2 & 5.1\np{e-}5 
        & 2.0\np{e-}2 & 1.6\np{e-}4 
        \\
        & & random &
          915120 
        & 1.0\np{e-}0 & 5.6\np{e-}1 
        & 1.0\np{e-}0 & 5.8\np{e-}1 
        & 1.0\np{e-}0 & 1.6\np{e-}1 
        & 1.0\np{e-}0 & 1.6\np{e-}1 
        \\
        \midrule
        \multirow{4}{*}{3} &\multirow{2}{*}{\sNetName} & prefix &
           66608 
        & 2.2\np{e-}2 & 1.3\np{e-}3 
        & 2.2\np{e-}2 & 1.4\np{e-}3 
        & 2.2\np{e-}2 & 4.1\np{e-}4 
        & 2.2\np{e-}2 & 6.3\np{e-}4 
        \\
        & & random &
           66608 
        & 1.0\np{e-}0 & 9.4\np{e-}2 
        & 1.0\np{e-}0 & 9.9\np{e-}2 
        & 1.0\np{e-}0 & 1.6\np{e-}2 
        & 1.0\np{e-}0 & 3.4\np{e-}2 
        \\
        & \multirow{2}{*}{\sINetName} & prefix &
          275504 
        & 2.2\np{e-}2 & 1.1\np{e-}3 
        & 2.2\np{e-}2 & 1.2\np{e-}3 
        & 2.2\np{e-}2 & 1.7\np{e-}4 
        & 2.2\np{e-}2 & 2.4\np{e-}4 
        \\
        & & random &
          275504 
        & 1.0\np{e-}0 & 2.1\np{e-}1 
        & 1.0\np{e-}0 & 2.7\np{e-}1 
        & 1.0\np{e-}0 & 3.1\np{e-}2 
        & 1.0\np{e-}0 & 6.3\np{e-}2 
        \\
        \bottomrule
    \end{tabular}
    \caption{Numerical results of \NetName{}s (\sNetName{})
    and \INetName{}s (\sINetName{}) on DFT-Lfreq, DFT-Hfreq,
    DFTSmooth-Lfreq, and DFTSmooth-Hfreq datasets. Prefix
    initialization refers to the {\it Butterfly} initialization
    defined in Section~\ref{sec:bnet}. Pre-training and post-training
    relative errors are reported for each dataset. Both \NetName{}
    and \INetName{} use 16 mixing channels ($r=4$) and $L=8$ layers.}
    \label{tab:approx-dft}
\end{table}

\begin{figure}[t]
    \centering
    \begin{subfigure}[b]{0.48\textwidth}
        \centering
        \includegraphics[width=\linewidth]{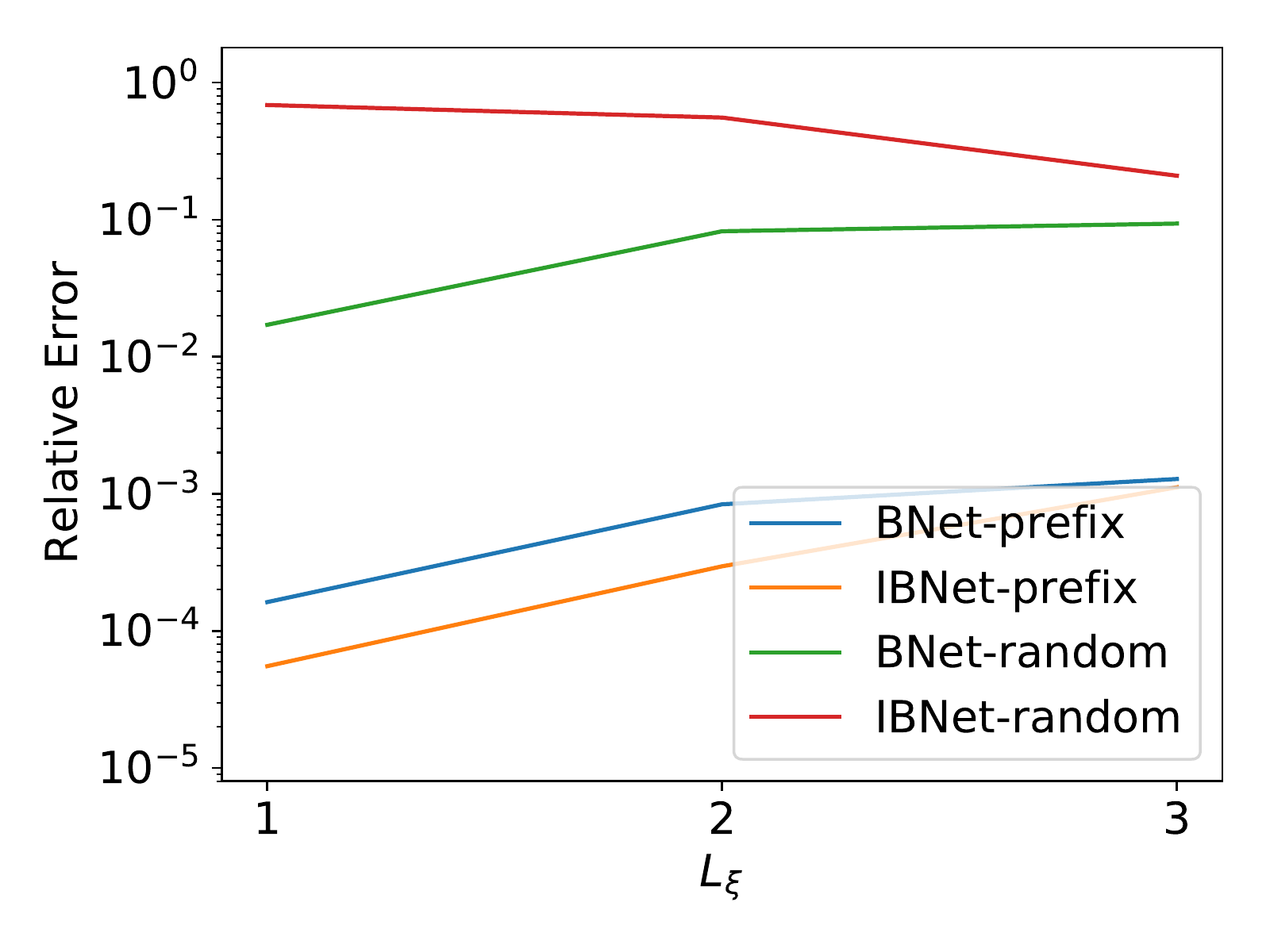}
        \caption{DFT-Lfreq.}
    \end{subfigure}
    \quad
    \begin{subfigure}[b]{0.48\textwidth}
        \centering
        \includegraphics[width=\linewidth]{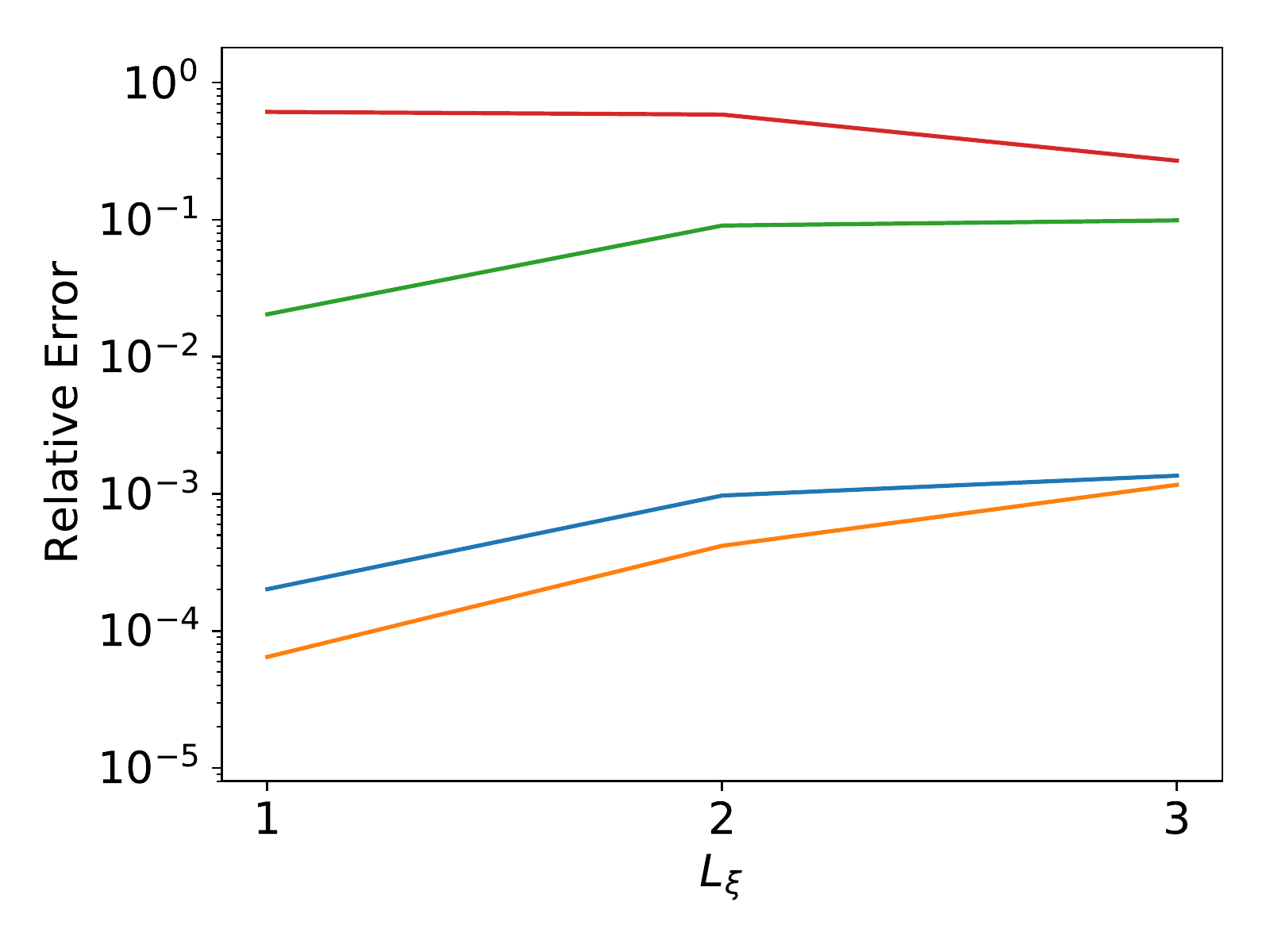}
        \caption{DFT-Hfreq.}
    \end{subfigure}
    \begin{subfigure}[b]{0.48\textwidth}
        \centering
        \includegraphics[width=\linewidth]{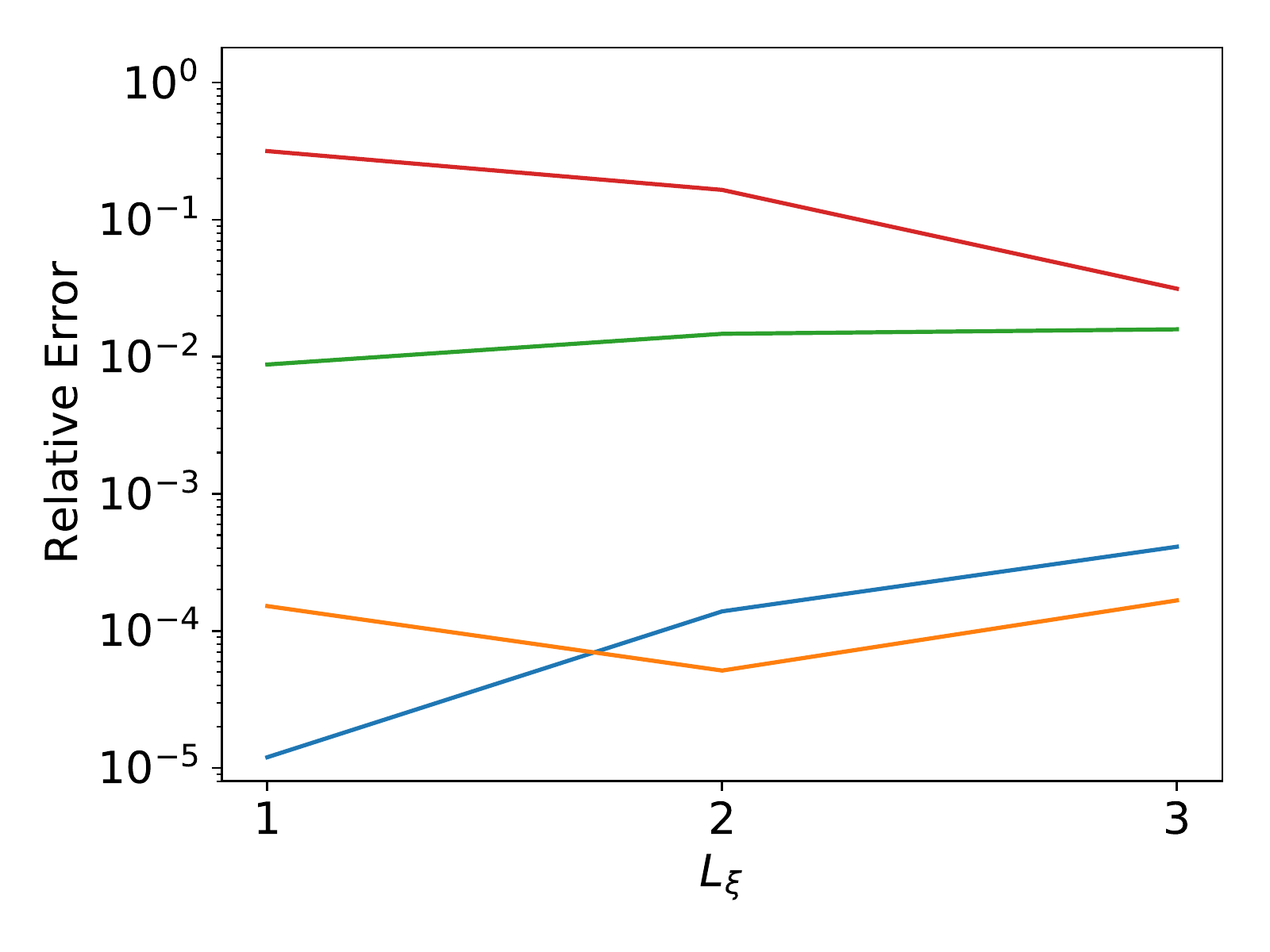}
        \caption{DFTSmooth-Lfreq.}
    \end{subfigure}
    \quad
    \begin{subfigure}[b]{0.48\textwidth}
        \centering
        \includegraphics[width=\linewidth]{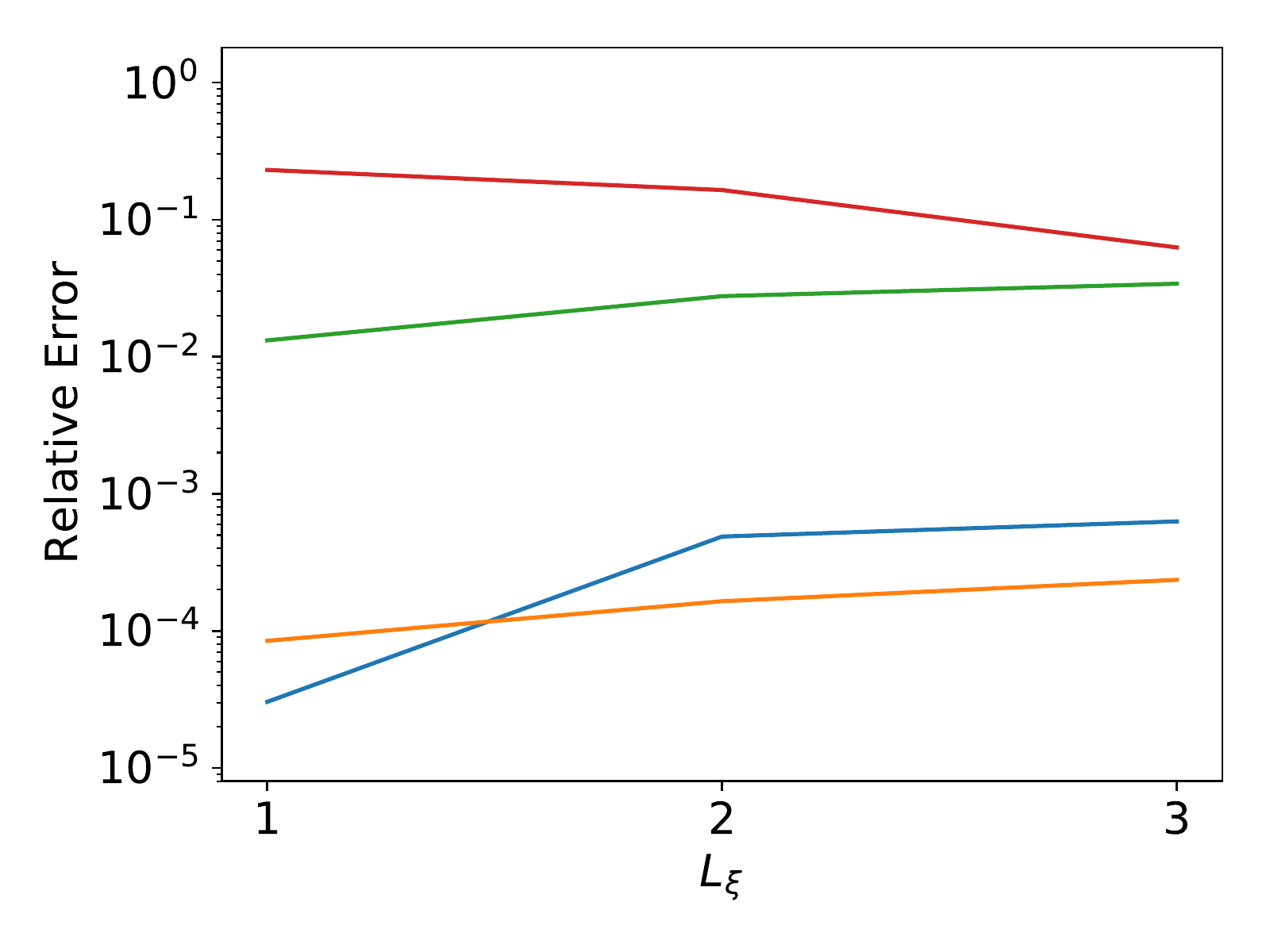}
        \caption{DFTSmooth-Hfreq.}
    \end{subfigure}
    \caption{Post-training relative error against $L_\xi$ for all
    datasets. $L_\xi$ indicates the position of the switch layer. All
    four figures share the same legend as (a).} \label{fig:testerr}
\end{figure}

{\bf Results.} Table~\ref{tab:approx-dft} reports the result on
all four datasets. Figure~\ref{fig:testerr} further illustrates the
post-training relative errors against $L_\xi$, where $L_\xi$ indicates
the position of the switch layer. We summarize the experimental
observations as follows.
\begin{enumerate}[(1)]

\item (Weight initialization, random v.s. prefix) In all cases,
    training from {\it Butterfly} initialization significantly
    outperforms training from random initialization by two to three
    digits. In almost all cases, the post-training relative errors of
    the randomly initialized networks are not even as accurate as the
    pre-training relative error of the {\it Butterfly} initialized
    counterparts.

\item (Channel sparsity, \NetName{} v.s. \INetName{}) Given a fixed
    $L_\xi$, \NetName{} has much less parameters than
    \INetName{}. While, their post-training relative errors stay at a
    similar level. For these datasets, channel sparsification reduces
    the number of parameters without loss much post-training accuracy.

\item (Influence of datasets)
    Given the same input data and different output frequency windows,
    as in DFT-Lfreq and DFT-Hfreq, the post-training relative errors
    of high frequency output are slightly less accurate than their
    low-frequency counterparts.  The impact of output frequency
    window on the training performance is limited. Given the same
    output frequency window and different input data, i.e., DFT-Lfreq
    v.s. DFTSmooth-Lfreq, and DFT-Hfreq v.s.  DFTSmooth-Hfreq, the
    post-training relative errors on datasets with smoother input data
    are half digit to one digit more accurate than that on datasets
    with less smooth input data.

\item (Position of switch layer, $L_\xi$)
    The number of parameters increases as $L_\xi$ decreases for
    both \NetName{} and \INetName{}.  In general, the post-training
    relative error increases as $L_\xi$ increases for all \NetName{}
    results. However, for \INetName{} with random initialization,
    the post-training relative error decreases as $L_\xi$ increases on
    all datasets.

\end{enumerate}

We further give a discussion on the above results. The comparison of
random initialization and {\it Butterfly} initialization indicates
that the trainings reach two different local minima.  The relative
error associated with {\it Butterfly} initialization is much smaller.
Numerically, we also found that if we set the learning rate for {\it
Butterfly} initialized training to be large then the training loss
first increases and then converges to a number worse than the results
given in the Table~\ref{tab:approx-dft}. Hence, we conjecture that
the energy landscape of this task is nasty and the {\it Butterfly}
initialization lies in a narrow but deep well. Training from {\it
Butterfly} initialization with small learning rate achieves the local
minima within this narrow but deep well.  Regarding the influence
of datasets, the intrinsic dimensionality of the input data plays
a crucial role in training.  If no guidance of feature selection is
provided on an intrinsically high-dimensional data, e.g., \INetName{}
with $L_\xi=1$ and random initialization applied to DFT-Lfreq, the
training fails to learn useful information.  Adding more guidance,
either refining network structure (increasing $L_\xi$, adding channel
sparsification) or providing {\it Butterfly} initialization, helps
training and achieves lower post-training relative error.  However,
if the intrinsic dimension of the input data is lowered, the training
learns more information and achieves lower relative error. Providing
extra guidance further helps training. For the position of switch
layer, we can further decrease $L_\xi$ to zero and end up with the
neural network proposed in \cite{Xu2019}. We emphasize that when
$L_\xi=0$, the \INetName{} is indeed a regular CNN. More numerical
results and connections between \INetName{} and regular CNN refer
to \cite{Xu2019}.

\subsection{Comparison to CNN and Transfer Learning}

This numerical experiment is to compare the \NetName{} and \INetName{}
on their transfer learning capability.  We adopt the neural networks
trained on DFTSmooth-Lfreq in previous section here. A sequence of
testing datasets are generated in the same way as in previous section
with $G_{\mathrm{center}} = 0, 2, \dots, 44$ and $G_{\mathrm{width}}
= 10$. Each testing dataset contains 1000 samples.

\begin{figure}[htp]
    \centering
    \begin{subfigure}[b]{0.48\textwidth}
        \centering
        \includegraphics[width=\linewidth]{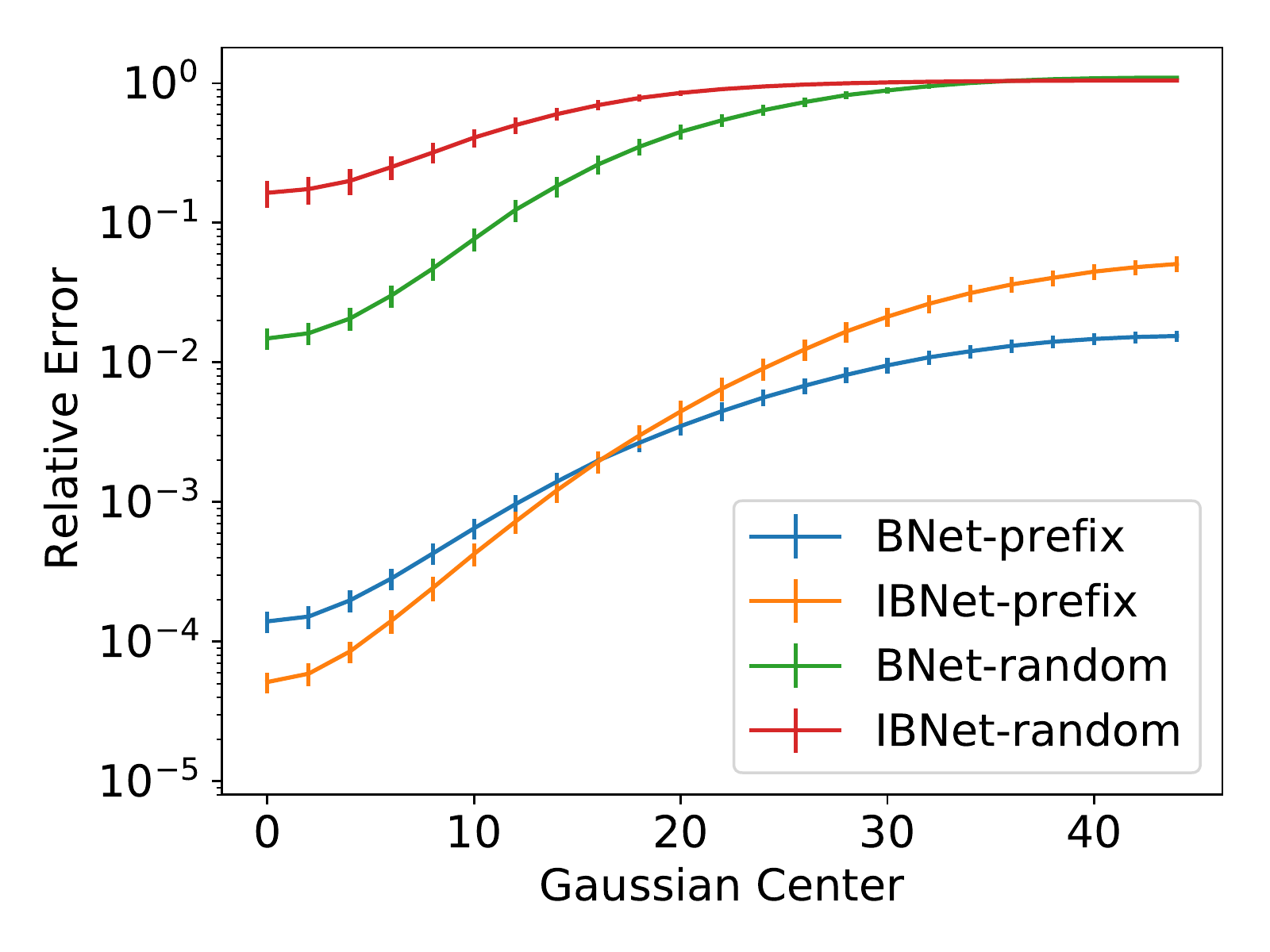} \caption{$L_\xi = 2$.}
    \end{subfigure}
    \quad
    \begin{subfigure}[b]{0.48\textwidth}
        \centering
        \includegraphics[width=\linewidth]{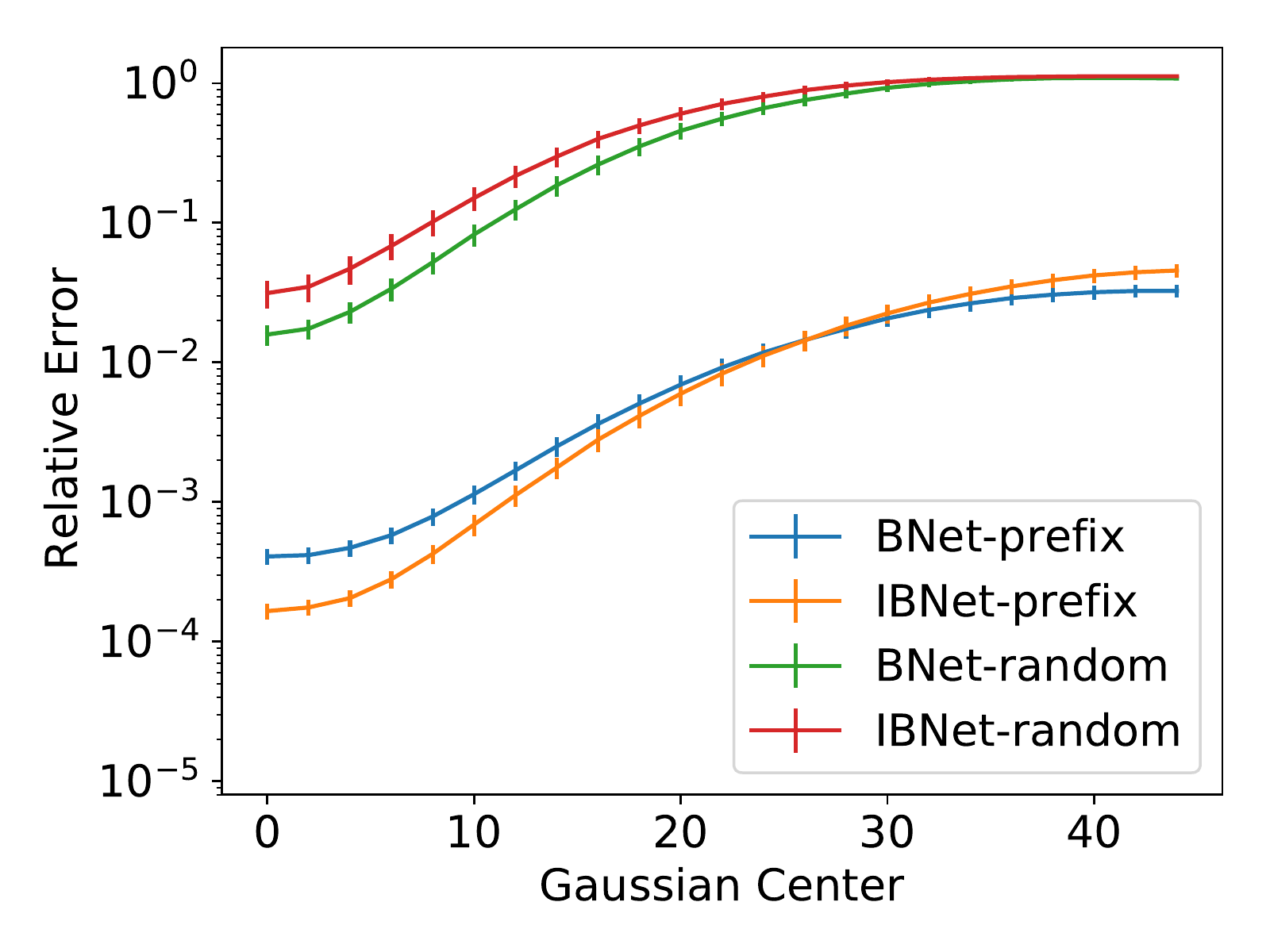}
        \caption{$L_\xi = 3$.}
    \end{subfigure}
    \caption{Mean and standard deviation for testing relative errors
    of \NetName{} and \INetName{} on testing datasets with various
    $G_{\mathrm{center}} = 0, 2, \dots, 44$ and $G_{\mathrm{width}}
    = 10$. All networks are trained on DFTSmooth-Lfreq dataset
    ($G_{\mathrm{center}} = 0$ and $G_{\mathrm{width}} = 10$).}
    \label{fig:trans}
\end{figure}

Figure~\ref{fig:trans} shows the means and standard deviations of the
transferred testing relative errors of neural networks with $L_\xi=2$
and $L_\xi=3$. For every transferred testing, the standard deviation
is orders of magnitude smaller than the corresponding mean. Hence the
variation of the testing relative error for each setting is relatively
small. Regarding the transfer learning capabilities, all randomly
initialized networks quick loss accuracy as the testing dataset shifted
away from the training dataset.  While, {\it Butterfly} initialized
neural networks still preserve reasonable level of accuracy.  Further,
comparing \NetName{} and \INetName{} both with {\it Butterfly}
initialization, \INetName{} achieves lower testing relative error when
the testing dataset has significant overlap with the training dataset.
However, when the overlap reduces, the testing relative error of
\INetName{} increase faster than that of \NetName{}.  Hence we conclude
that \NetName{} has better generalizability than that of \INetName{}.
Further, comparing the case of $L_\xi=2$ and $L_\xi=3$, we notice
that the crossover of two testing relative error curves comes later
in $L_\xi=3$. This is because that the \INetName{} with $L_\xi=3$
has much less number of parameters than that of \INetName{} with
$L_\xi=2$, hence is less adapted to training dataset after training.

\subsection{General Function Approximation}
\label{sec:energyLaplace}

The last numerical example aims to construct an approximation of
the energy functional of 1D Poisson's equation and another analog
functional for high frequency input data, both of which correspond
directly to the approximation power in Section~\ref{sec:approx-main}
in representing general functions.

{\bf Functional setup.} For a Poisson's equation $\Delta u(t) = x(t)$
with periodic boundary condition, the energy functional of Poisson's
equation is defined as the negative inner product of $u$ and $x$,
which can also be approximated by a quadratic form of the leading
low-frequency Fourier components,
\begin{equation} \label{eq:energyfunc1}
    \calE_1(x) = - \left<x,u\right> \approx \sum_{k \in (-K,K)}
    \frac{1}{\abs{k}^2} \abs{\widehat{x}_k}^2 = \sum_{k \in [1,K)}
    \frac{2}{\abs{k}^2} \abs{\widehat{x}_k}^2,
\end{equation}
where $\widehat{x}_k$ is the Fourier component of $x$ at frequency $k$
and the last equality comes from the assumption of real input $x$.  If
the input function $x$ is a band limited function within the frequency
window, $(-K, K)$, then equality is achieved in \eqref{eq:energyfunc1}.
The other analogy functional for input data $x$ is defined as,
\begin{equation} \label{eq:energyfunc2}
    \calE_2(x) = \sum_{k \in (K_0,K_0+K)} \frac{2}{\abs{k-K_0}^2}
    \abs{\widehat{x}_k}^2.
\end{equation}
Throughout this section, $K_0$ is 256 and $K$ is 128.
The input data is generated in the same way as that in
Section~\ref{sec:numerical-approx-power}. The centers of Gaussian
are $G_{\mathrm{center}} = 0$ and $G_{\mathrm{center}} = 256$ for
$\calE_1$ and $\calE_2$ respectively. The widths of the Gaussian for
both functionals are $G_{\mathrm{width}} = 30$.

{\bf Neural network setup.} Task layers are attached to \NetName{} and
\INetName{}. We have two different task layers, namely square-sum-layer
and dense-dense-layer.  The square-sum-layer first squares all
output of the previous layers in \NetName{} or \INetName{}, then
multiplies each squared value by a weight, and finally sums them
together. This is equivalent to square the output and then connect
a single dense layer with one output unit. The square-sum-layer is
able to exactly represent both functionals, \eqref{eq:energyfunc1} and
\eqref{eq:energyfunc2} if the weights of the dense layer is properly
initialized. The dense-dense-layer attaches a dense layer with 256
output units with both bias and ReLU activation function enabled.
Then another dense layer with one output unit is attached afterwards.
Both functionals can only be approximated by the dense-dense-layer.

{\bf Training and evaluation setup.}
All \NetName{}s and \INetName{}s with different initializations are
trained under the infinity data setting. The input data length is
$N=1024$, the batch size is 256, the maximum number of iteration
is $50,000$, and ADAM optimizer is used with an exponentially
decay learning rate. The initial learning rate is $10^{-4}$ for {\it
Butterfly} initialized networks with square-sum-layer, and is $10^{-3}$
for all other settings.  The decay steps and the decay rate are 100
and 0.985 respectively. Again, the maximum number of iteration is
sufficient for the convergence of relative errors in all settings.
The loss function is defined as,
\begin{equation}
    \ell (\{\vec{x}_i\}) = \sum_{i} \abs{\calN(\vec{x}_i) -
    \calE(\vec{x}_i)}^2,
\end{equation}
for $\calE$ being either $\calE_1$ or $\calE_2$. Relative errors are
reported for comparison. The testing data is of size 1000. Default
values are used for other unspecified hyper parameters.

\begin{table}[t]
    \centering
    \small
    \begin{tabular}{lccccccccc}
        \toprule
        \multirow{2}{*}{Task Layer} & \multirow{2}{*}{$L_\xi$} &
        \multirow{2}{*}{\shortstack{Neural\\Network}}
        & \multirow{2}{*}{Initial}
        & \multicolumn{2}{c}{Num Paras}
        & \multicolumn{2}{c}{Functional $\calE_1$}
        & \multicolumn{2}{c}{Functional $\calE_2$} \\
        \cmidrule(lr){5-6}
        \cmidrule(lr){7-8}
        \cmidrule(lr){9-10}
        & & & & \sNetName{}/\sINetName{} & Task
        & \shortstack{Pre\\Train} & \shortstack{Post\\Train}
        & \shortstack{Pre\\Train} & \shortstack{Post\\Train} \\
        \toprule
        \multirow{12}{*}{Square-sum-layer}
        & \multirow{4}{*}{1} &\multirow{2}{*}{\sNetName} & prefix &
          136304 &      256 
        & 1.80\np{e-}2 & 2.35\np{e-}5 
        & 1.84\np{e-}2 & 3.01\np{e-}5 
        \\
        & & & random &
          136304 &      256 
        & 1.00\np{e-}0 & 7.62\np{e-}3 
        & 1.00\np{e-}0 & 1.85\np{e-}2 
        \\
        & & \multirow{2}{*}{\sINetName} & prefix &
         3533936 &      256 
        & 1.69\np{e-}2 & 1.78\np{e-}4 
        & 1.80\np{e-}2 & 1.92\np{e-}4 
        \\
        & & & random &
         3533936 &      256 
        & 1.00\np{e-}0 & 5.80\np{e-}3 
        & 1.00\np{e-}0 & 9.48\np{e-}3 
        \\
        \cmidrule(lr){2-10}
        & \multirow{4}{*}{2} &\multirow{2}{*}{\sNetName} & prefix &
           87728 &      256 
        & 1.79\np{e-}2 & 4.36\np{e-}5 
        & 1.72\np{e-}2 & 5.84\np{e-}5 
        \\
        & & & random &
           87728 &      256 
        & 1.00\np{e-}0 & 9.62\np{e-}3 
        & 1.00\np{e-}0 & 1.74\np{e-}2 
        \\
        & & \multirow{2}{*}{\sINetName} & prefix &
          915120 &      256 
        & 1.78\np{e-}2 & 6.33\np{e-}5 
        & 1.66\np{e-}2 & 7.30\np{e-}5 
        \\
        & & & random &
          915120 &      256 
        & 1.00\np{e-}0 & 8.71\np{e-}3 
        & 1.00\np{e-}0 & 1.02\np{e-}2 
        \\
        \cmidrule(lr){2-10}
        & \multirow{4}{*}{3} &\multirow{2}{*}{\sNetName} & prefix &
           66608 &      256 
        & 1.50\np{e-}2 & 9.51\np{e-}5 
        & 1.55\np{e-}2 & 2.27\np{e-}4 
        \\
        & & & random &
           66608 &      256 
        & 1.00\np{e-}0 & 1.46\np{e-}2 
        & 1.00\np{e-}0 & 4.04\np{e-}2 
        \\
        & & \multirow{2}{*}{\sINetName} & prefix &
          275504 &      256 
        & 1.59\np{e-}2 & 5.22\np{e-}5 
        & 1.69\np{e-}2 & 1.17\np{e-}4 
        \\
        & & & random &
          275504 &      256 
        & 1.00\np{e-}0 & 1.15\np{e-}2 
        & 1.00\np{e-}0 & 2.38\np{e-}2 
        \\
        \toprule
        \multirow{12}{*}{Dense-dense-layer}
        & \multirow{4}{*}{1} &\multirow{2}{*}{\sNetName} & prefix &
          136304 &    66048 
        & 1.00\np{e-}0 & 7.44\np{e-}3 
        & 9.99\np{e-}1 & 9.30\np{e-}3 
        \\
        & & & random &
          136304 &    66048 
        & 1.00\np{e-}0 & 1.27\np{e-}2 
        & 1.00\np{e-}0 & 2.26\np{e-}2 
        \\
        & & \multirow{2}{*}{\sINetName} & prefix &
         3533936 &    66048 
        & 9.99\np{e-}1 & 4.11\np{e-}3 
        & 1.00\np{e-}0 & 5.84\np{e-}3 
        \\
        & & & random &
         3533936 &    66048 
        & 1.00\np{e-}0 & 7.67\np{e-}3 
        & 1.00\np{e-}0 & 1.44\np{e-}2 
        \\
        \cmidrule(lr){2-10}
        & \multirow{4}{*}{2} &\multirow{2}{*}{\sNetName} & prefix &
           87728 &    66048 
        & 9.97\np{e-}1 & 8.37\np{e-}3 
        & 1.00\np{e-}0 & 1.18\np{e-}2 
        \\
        & & & random &
           87728 &    66048 
        & 1.00\np{e-}0 & 2.30\np{e-}2 
        & 1.00\np{e-}0 & 2.19\np{e-}2 
        \\
        & & \multirow{2}{*}{\sINetName} & prefix &
          915120 &    66048 
        & 1.00\np{e-}0 & 5.74\np{e-}3 
        & 9.99\np{e-}1 & 7.80\np{e-}3 
        \\
        & & & random &
          915120 &    66048 
        & 1.00\np{e-}0 & 1.11\np{e-}2 
        & 1.00\np{e-}0 & 2.12\np{e-}2 
        \\
        \cmidrule(lr){2-10}
        & \multirow{4}{*}{3} &\multirow{2}{*}{\sNetName} & prefix &
           66608 &    66048 
        & 9.99\np{e-}1 & 9.83\np{e-}3 
        & 1.00\np{e-}0 & 1.01\np{e-}2 
        \\
        & & & random &
           66608 &    66048 
        & 1.00\np{e-}0 & 1.66\np{e-}2 
        & 1.00\np{e-}0 & 2.75\np{e-}2 
        \\
        & & \multirow{2}{*}{\sINetName} & prefix &
          275504 &    66048 
        & 9.98\np{e-}1 & 6.84\np{e-}3 
        & 9.97\np{e-}1 & 9.00\np{e-}3 
        \\
        & & & random &
          275504 &    66048 
        & 1.00\np{e-}0 & 1.21\np{e-}2 
        & 1.00\np{e-}0 & 3.55\np{e-}2 
        \\
        \bottomrule
    \end{tabular}
    \caption{Numerical results of \NetName{}s and \INetName{}s with
    various task layers for functional \eqref{eq:energyfunc1} and
    \eqref{eq:energyfunc2}. Pre-training and post-training relative
    errors are reported for each functional.  All neural networks
    use 16 mixing channels and $L=8$ layers. } \label{tab:energyfunc1}
\end{table}

{\bf Result.}
Table~\ref{tab:energyfunc1} show the results for functional
\eqref{eq:energyfunc1} and \eqref{eq:energyfunc2}. The comparison
of the number of parameters for \NetName{}/\INetName{} is the same
as that in Section~\ref{sec:numerical-approx-power}, while the
number of parameters in dense-dense-layer is much larger than that
in square-sum-layer. More accurate approximation is achieved using
square-sum-layer comparing to dense-dense-layer, which is due to
fact that the functionals can be exactly represented by the former
task layer but not the latter. The post-training relative errors for
\NetName{} and \INetName{} given the same initialization, $L_\xi$, and
task layer, remain similar in all cases. Hence the significant larger
number of parameters in \INetName{} does not improve the post-training
accuracy much. Most importantly, as we compare the post-training
relative errors of different initializations under dense-dense-layer,
{\it Butterfly} initialized networks achieves better accuracy comparing
to its random initialized counterpart. The dense-dense-layers here
are all randomly initialized and only the weights in \NetName{}
and \INetName{} are initialized differently. Hence we conclude
that {\it Butterfly} initialization, even just on part of the whole
neural network, helps finding better approximations in representing
functionals \eqref{eq:energyfunc1} and \eqref{eq:energyfunc2}.

\section{Conclusion and Discussion}
\label{sec:conclusion}

A low-complexity convolutional neural network with structured {\it
Butterfly} initialization and sparse cross-channel connections is
proposed, motivated by the {\it Butterfly} scheme. The functional
representation by \NetName{} is optimal in the sense that the model
complexity is $O(K \log N)$ and the computational complexity is $O(N
\log N)$ for $N$ and $K$ being the input and output vector lengths. The
approximation accuracy to the Fourier kernel is proved to exponentially
decay as the depths of the \NetName{} increases.  We also conduct an
approximation analysis of \NetName{} in representing a large class
of problems in scientific computing and image and signal processing.
Comparing \NetName{} to fully connected networks, the leading term
in network complexity is reduced from $\varepsilon^{-N/s}$ down to
$\varepsilon^{-K/s}$, where $N$ is the input dimension, $K$ is the
effective dimension, and $s$ is regularization level of the problem.
Regular CNN can be viewed a special network under the analysis.

The trained \NetName{}s from {\it Butterfly} initialization and
random initialization are applied to represent discrete Fourier
transforms and energy functionals. For these examples, \NetName{}
achieves better accuracy than its no-trained version. We also compared
\NetName{} against \INetName{}. From the numerical results, we find
that \NetName{} is able to achieve similar accuracy as \INetName{},
while the number of parameters is orders of magnitudes smaller. In
the transfer learning settings, \NetName{} generalizes better than
\INetName{} when the distribution of the input data has domain shift.

The work can be extended in several directions. First, more
applications of the \NetName{} can be explored such as those in
image analysis and signal processing. Likely, \NetName{} is able to
replace some CNN structures in practice such that similar accuracy
can be achieved while the parameter number is much reduced. Second,
our current theoretical analysis does not address the case when the
input data contain noise. In particular, adding rectified layers in
\NetName{} can be interpreted as a thresholding denoising operation
applied to the intermediate representations; a statistical analysis
is desired.

\subsubsection*{Acknowledgments}

The work of YL and JL is supported in part by National Science
Foundation via grants DMS-1454939 and ACI-1450280. 
XC is partially supported by NSF (DMS-1818945, DMS-1820827), NIH (grant R01GM131642) and the Alfred P. Sloan Foundation.

\bibliographystyle{plainnat}
\bibliography{library}

\appendix

\section{Proof of Theorem~\ref{thm:low-rank}}
\label{app:low-rank}

Here we first include a well-known lemma of Chebyshev interpolation
for completeness and then prove Theorem~\ref{thm:low-rank}.

\begin{lemma} \label{lem:cheb}
    Let $f(y) \in C_{[a,b]}$ and $\calP_r$ be the space spanned by
    the monomials $y^{r}$. The projection operator $\Pi_r$ mapping $f$
    into its Lagrange interpolation on the $r$ Chebyshev grid obeys,
    \begin{equation}
        \norm{f - \Pi_r f}_\infty \leq \left( 2 + \frac{2}{\pi}
        \ln r \right) \inf_{g \in \calP_r} \norm{f-g}_\infty.
    \end{equation}
\end{lemma}

The proof of Lemma~\ref{lem:cheb} can be found in \cite{Rivlin1990}.

\begin{proof} [Proof of Theorem~\ref{thm:low-rank}]
    The Fourier kernel $\func{\calK}{\xi,t} = \Kfun{\xi \cdot t}$ can be
    decomposed as,
    \begin{equation} \label{eq:expan-K}
        \begin{split}
            \func{\calK}{\xi,t} & =  \Kfun{\left( \xi \cdot t - \xi_0
            \cdot t - \xi \cdot t_0 + \xi_0 \cdot t_0 \right)}
            \cdot \Kfun{\xi_0 \cdot t} \cdot \Kfun{\xi \cdot t_0}
            \cdot \negKfun{\xi_0 \cdot t_0} \\
            & =  \Kfun{R(\xi, t)}
            \cdot \Kfun{\xi_0 \cdot t} \cdot \Kfun{\xi \cdot t_0}
            \cdot \negKfun{\xi_0 \cdot t_0},
        \end{split}
    \end{equation}
    where $R(\xi,t) = \left(\xi - \xi_0\right) \cdot \left(t-
    t_0\right)$, $t_0$ and $\xi_0$ are centers of $B$ and $A$
    respectively.

    Next, we show the $r$-term truncation error for the first term
    in the second line of \eqref{eq:expan-K}. Based on the power
    expansion of $\Kfun{R(\xi,t)}$, i.e.,
    \begin{equation}
        \Kfun{R(\xi,t)} = \sum_{k=0}^\infty \frac{\left( - 2 \pi \imath
        R(\xi,t) \right)^k}{k!},
    \end{equation}
    the $r$-term truncation error can be bounded as,
    \begin{equation}
        \begin{split}
            \delta & = \abs{\Kfun{R(\xi,t)} - \sum_{k=0}^r \frac{\left(
            - 2 \pi \imath R(\xi,t) \right)^k )}{k!}} = \abs{
            \sum_{k=r+1}^\infty \frac{\left( - 2 \pi \imath R(\xi,t)
            \right)^k}{k!} } \\
            & \leq \sum_{k=r+1}^\infty \frac{(\pi w(A)w(B))^k} {2^k
            k!} \leq \sum_{k=r+1}^\infty \left( \frac{\pi e w(A)
            w(B)} {2k} \right)^k \leq \left( \frac{ \pi e w(A) w(B)
            }{2r} \right)^{r},
        \end{split}
    \end{equation}
    where the last inequality uses $w(A) w(B) \leq \frac{r}{\pi e}$.
    We also notice that, for any fixed $\xi$, $\sum_{k=0}^r \frac{ \left(
    2 \pi \imath R(\xi,\cdot) \right)^k }{k!} \in \calP_r$. Applying
    Lemma~\ref{lem:cheb}, we obtain,
    \begin{equation}
        \norm{\Kfun{R(\xi,t)} - \sum_{k = 1}^{r}\Kfun{R(\xi,t_k)}
        \calL_k(t) }_\infty \leq \left( 2 + \frac{2}{\pi} \ln r
        \right) \delta.
    \end{equation}
    By substituting the
    explicit expression of $R(\xi,t)$, we obtain one of the conclusion,
    \begin{equation}
        \norm{\Kfun{\xi \cdot t} - \sum_{k = 1}^{r}\Kfun{\xi \cdot
        t_k} \Kfun{\xi_0 \cdot \left( t - t_k \right)} \calL_k(t)}_\infty
        \leq \left( 2 + \frac{2}{\pi} \ln r \right) \left( \frac{
        \pi e w(A) w(B) }{2r} \right)^r,
    \end{equation}
    for any $\xi \in A$ and $t \in B$.
    
    Similarly, for any fixed $t$, we have $\sum_{k=0}^r \frac{ \left(
    2 \pi \imath R(\xi,t) \right)^k }{k!} \in \calP_r(\xi)$. Hence
    the second conclusion can be obtained through the same procedure.
\end{proof}

\section{Complex valued operation with ReLU activation function}
\label{app:complex-real}

A complex-valued linear operator can be represented through real-valued
operation with ReLU activation function.
Assume we want to represent the complex-valued linear operation,
\begin{equation}
    y = a x
\end{equation}
where $x,y \in \bbC$ can be generalized to vectors and $a \in \bbC$
can be generalized to a complex-valued matrix.  We first introduce
the real-valued representation for complex numbers/vectors. A complex
number $x = \Re x+\imath \Im x \in \bbC$ is represented as
\begin{equation} \label{eq:realx}
    \begin{pmatrix}
        (\Re x)_+ & (\Im x)_+ & (\Re x)_- & (\Im x)_-
    \end{pmatrix}^\top,
\end{equation}
where $(z)_+ = \max(z,0)$ and $(z)_- = -\min(z,0)$ for any $z \in
\bbR$. If $x$ is a complex vector of size $n$, then the real-valued
representation concatenates \eqref{eq:realx} for each complex value and
results a real vector of size $4n$.

The multiplication $y = ax$ is produced as the ReLU activation function
acting on a matrix vector multiplication, i.e.,
\begin{equation} \label{eq:exeq}
    \sigma \left(
    \begin{pmatrix}
        \Re a & -\Im a & -\Re a & \Im a \\
        \Im a & \Re a & -\Im a & -\Re a \\
        -\Re a & \Im a & \Re a & -\Im a \\
        -\Im a & -\Re a & \Im a & \Re a \\
    \end{pmatrix}
    \begin{pmatrix}
        (\Re x)_+ \\
        (\Im x)_+ \\
        (\Re x)_- \\
        (\Im x)_-
    \end{pmatrix}
    \right)
    =
    \begin{pmatrix}
        (\Re y)_+ \\
        (\Im y)_+ \\
        (\Re y)_- \\
        (\Im y)_-
    \end{pmatrix},
\end{equation}
where the complex-valued linear operator $a$ is extended to be a
real-valued $4 \times 4$ matrix. Extension of $a$ to a matrix can
also be done through concatenating in both row and column directions.

In order to simplify the description in this paper, we define
an extensive assign operator as $\exeq$ such that the 4 by 4 matrix $A$
in \eqref{eq:exeq} then obeys $A \exeq a$.

The above description can also be found in \cite{Xu2019}.

\section{Example signals of datasets}
\label{app:dataset-example}

We provide instance input and output data for each dataset in
Table~\ref{tab:datasets}, namely Figure~\ref{fig:data-DFT-Lfreq}
for DFT-Lfreq, Figure~\ref{fig:data-DFT-Hfreq} for DFT-Hfreq,
Figure~\ref{fig:data-DFTSmooth-Lfreq} for DFTSmooth-Lfreq, and
Figure~\ref{fig:data-DFTSmooth-Hfreq} for DFTSmooth-Hfreq.

\begin{figure}[htp]
    \centering
    \begin{subfigure}[b]{0.48\textwidth}
        \centering
        \includegraphics[width=\linewidth]{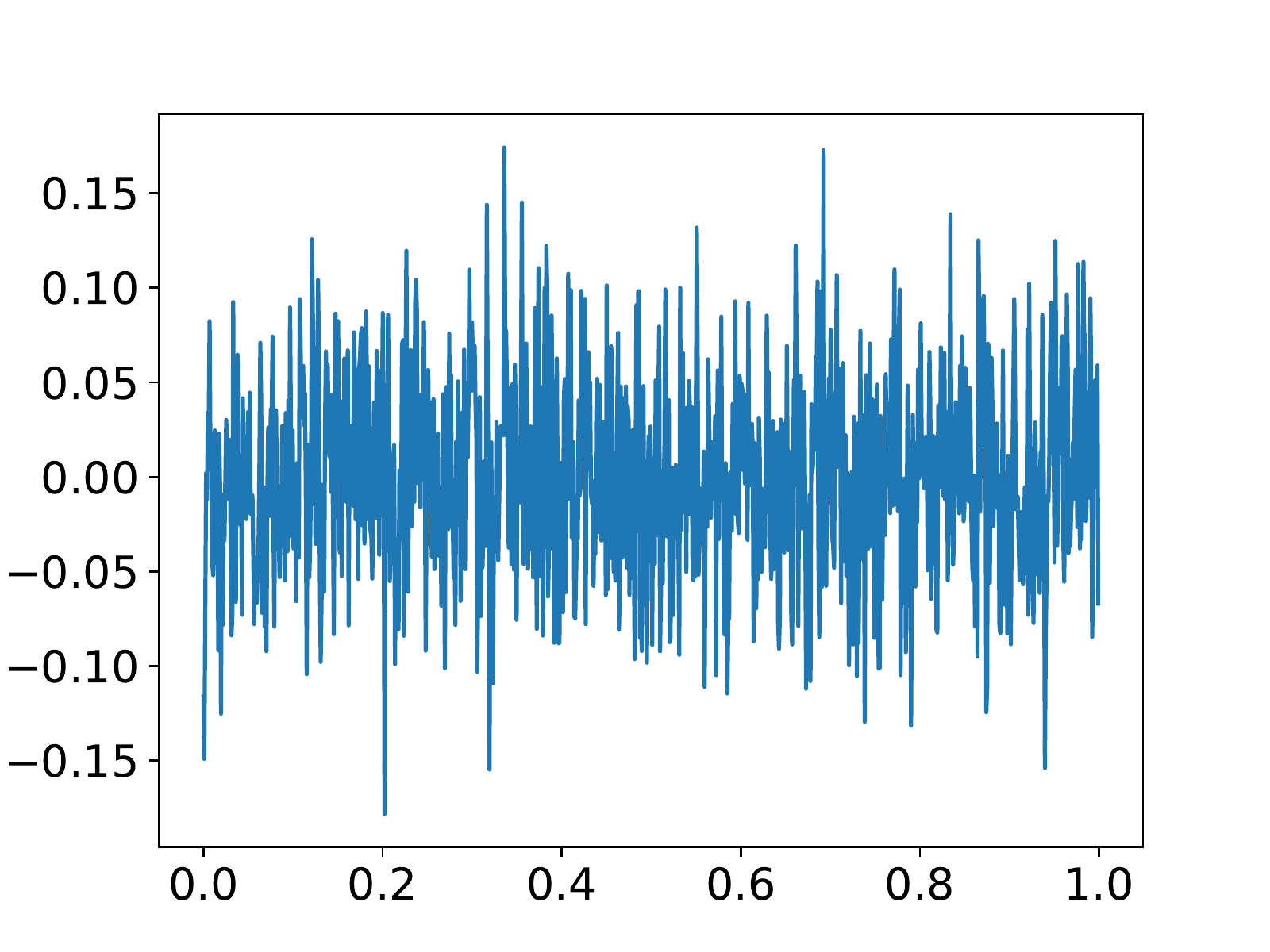}
        \caption{Input data $\vec{x}$.}
    \end{subfigure}
    \quad
    \begin{subfigure}[b]{0.48\textwidth}
        \centering
        \includegraphics[width=\linewidth]{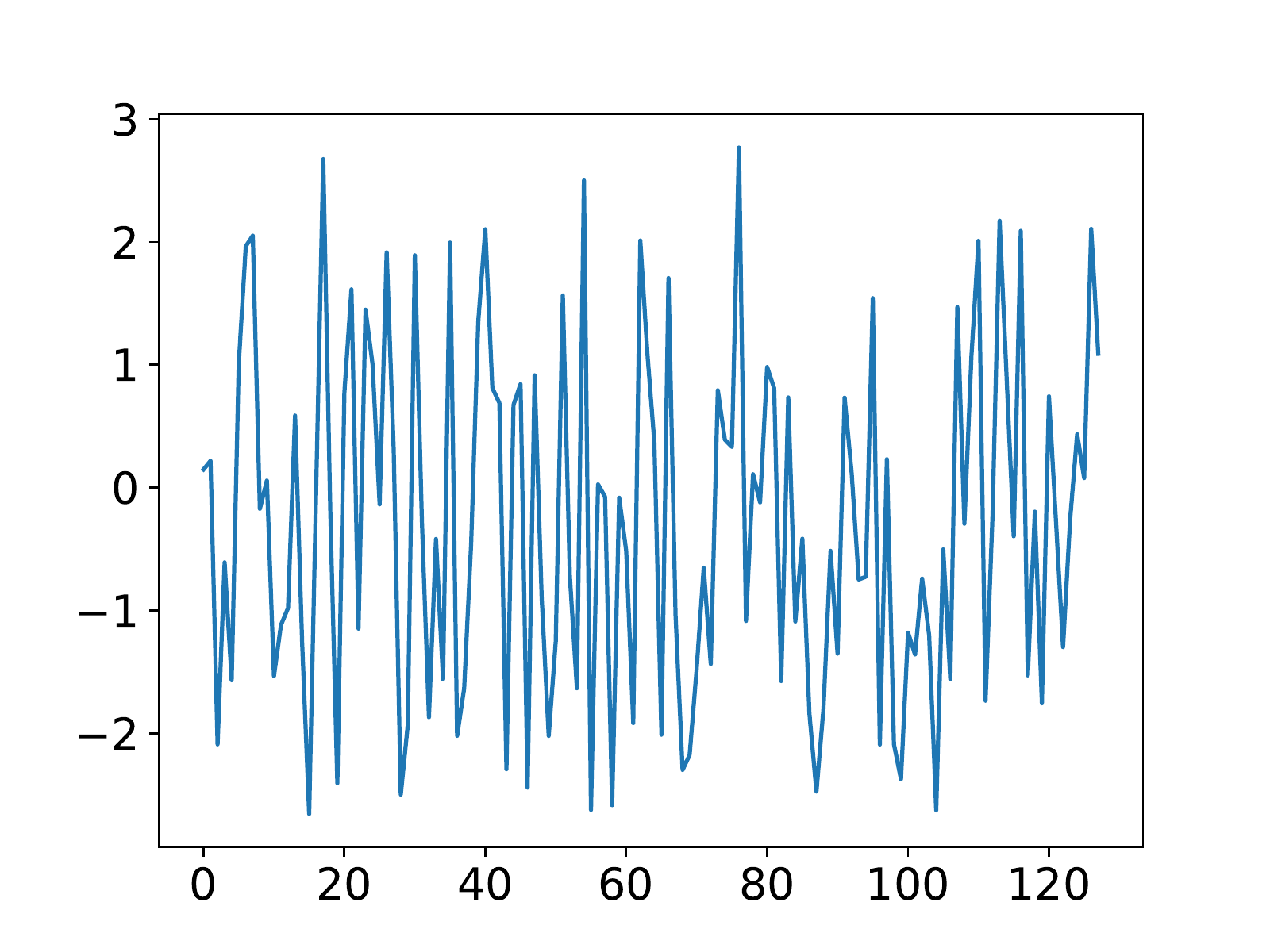}
        \caption{Real part of output data $\vec{y}$.}
    \end{subfigure}
    \caption{One instance in dataset DFT-Lfreq.}
    \label{fig:data-DFT-Lfreq}
\end{figure}

\begin{figure}[htp]
    \centering
    \begin{subfigure}[b]{0.48\textwidth}
        \centering
        \includegraphics[width=\linewidth]{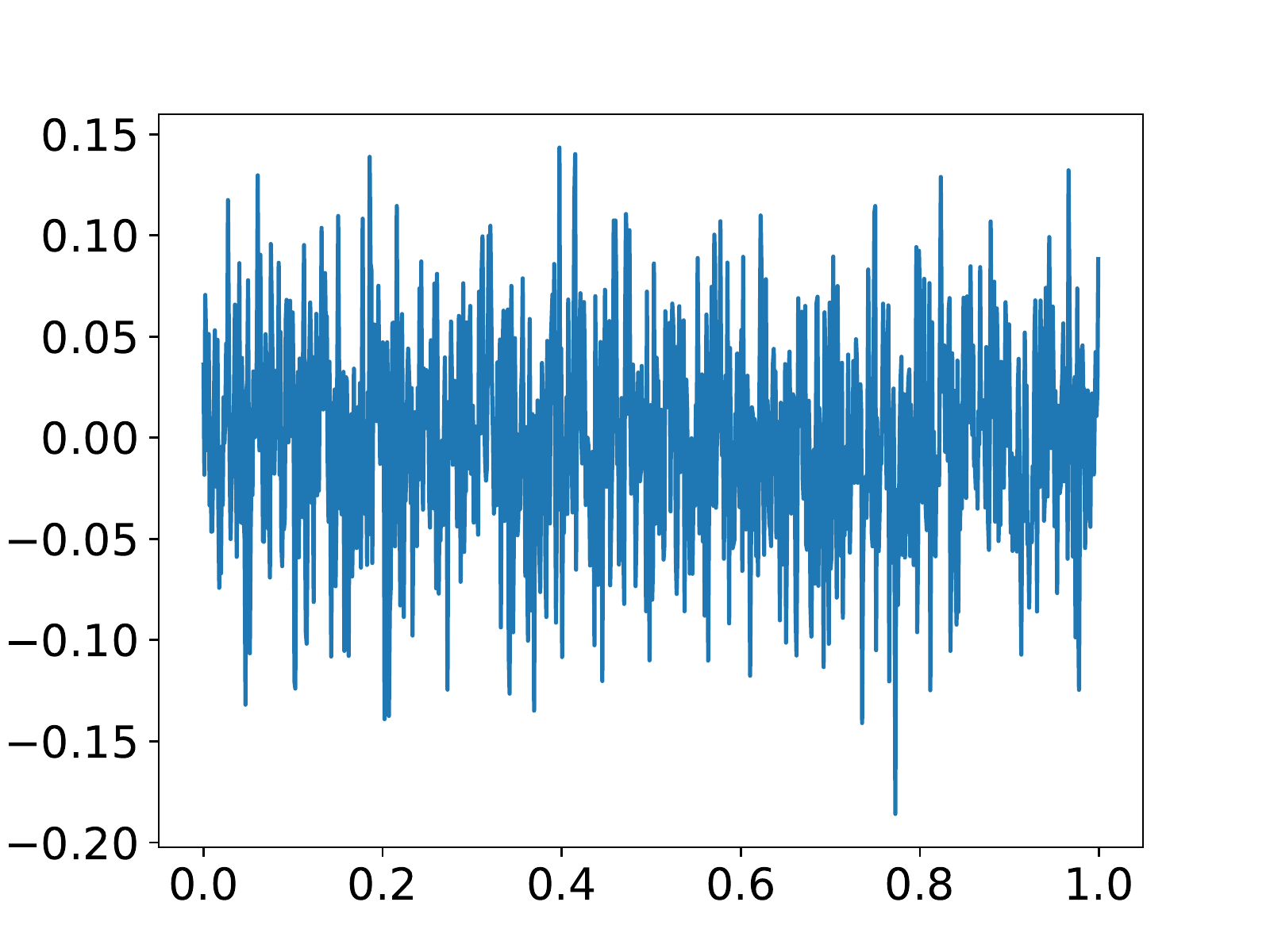}
        \caption{Input data $\vec{x}$.}
    \end{subfigure}
    \quad
    \begin{subfigure}[b]{0.48\textwidth}
        \centering
        \includegraphics[width=\linewidth]{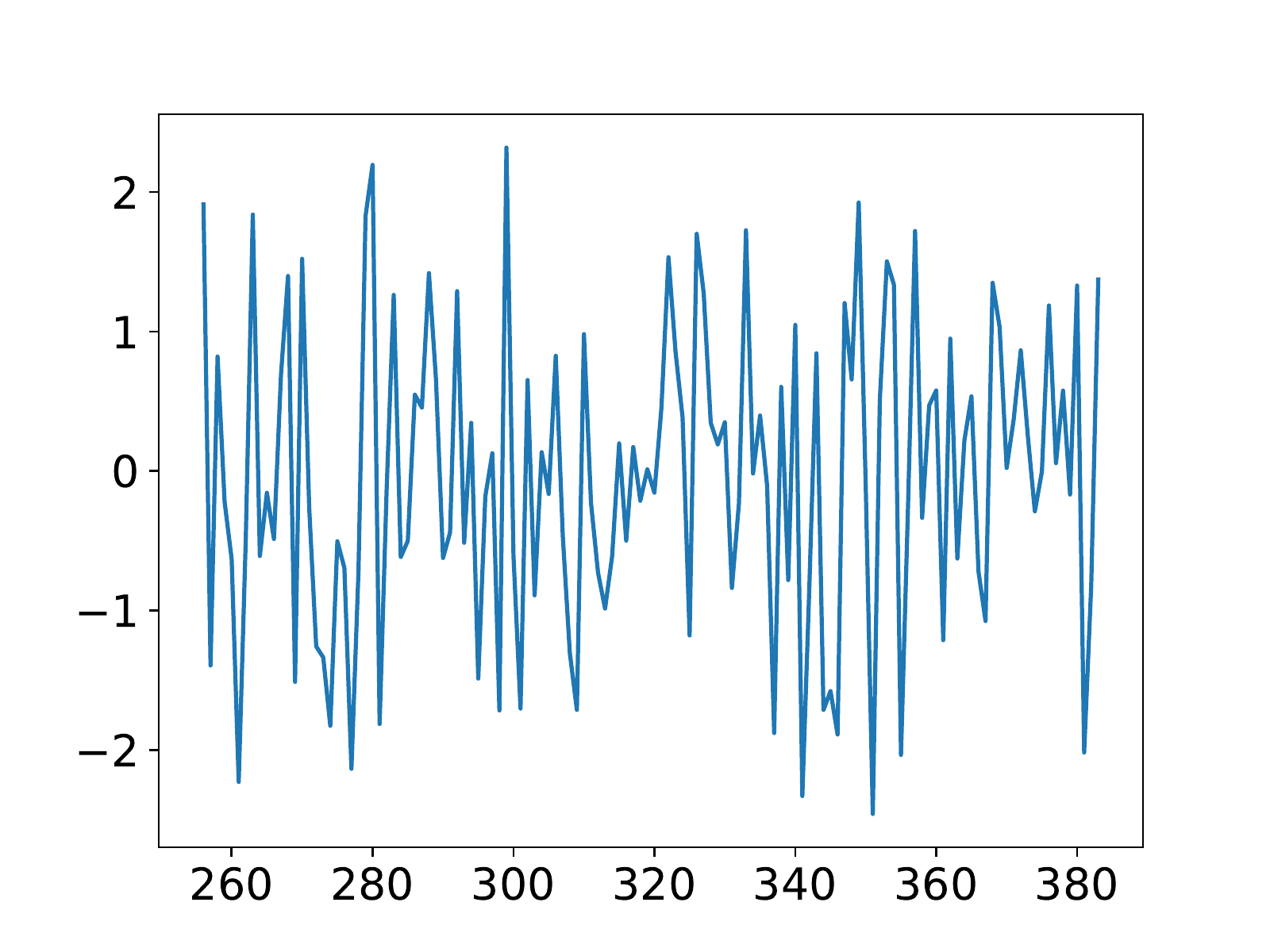}
        \caption{Real part of output data $\vec{y}$.}
    \end{subfigure}
    \caption{One instance in dataset DFT-Hfreq.}
    \label{fig:data-DFT-Hfreq}
\end{figure}

\begin{figure}[htp]
    \centering
    \begin{subfigure}[b]{0.48\textwidth}
        \centering
        \includegraphics[width=\linewidth]{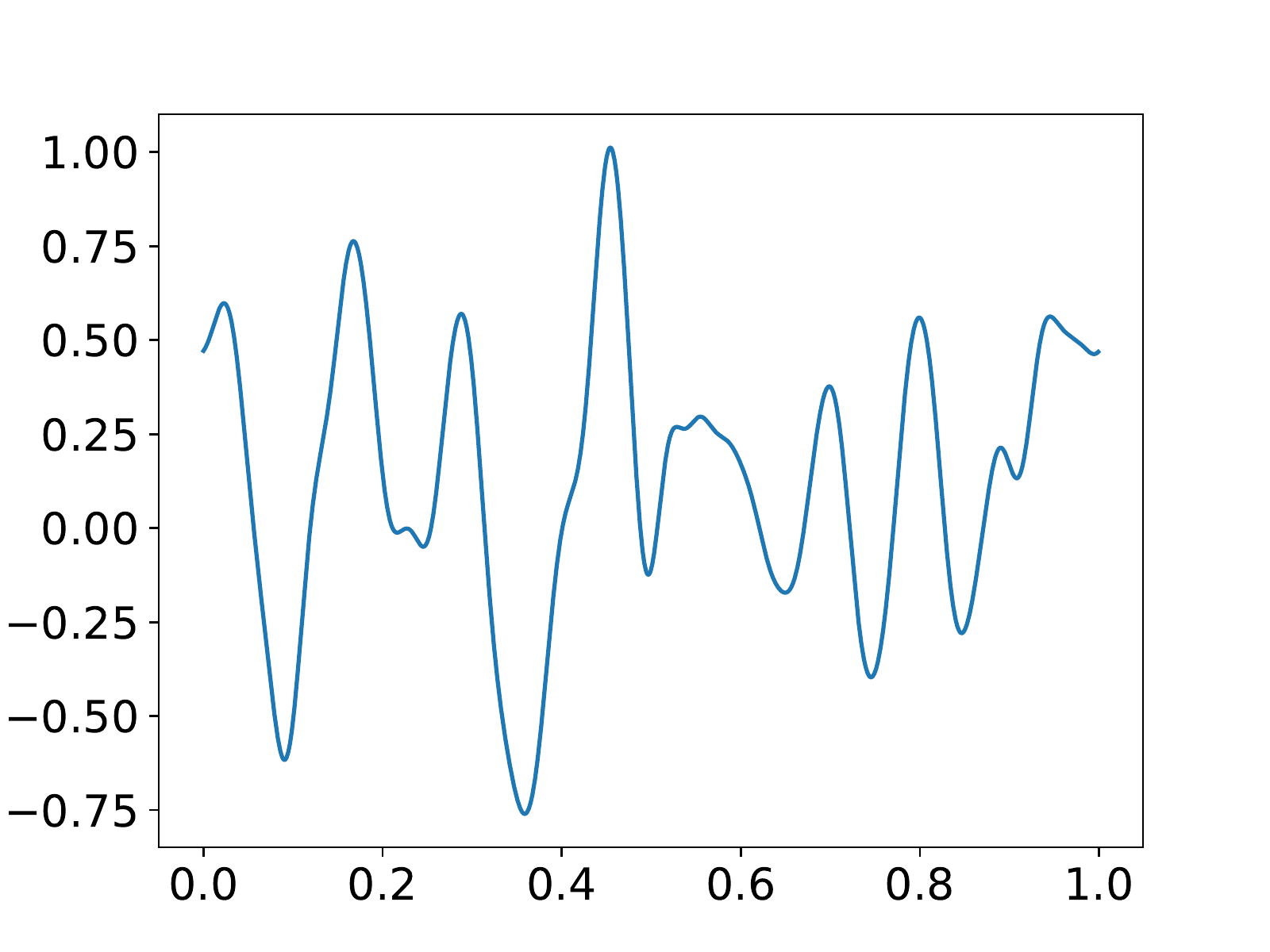}
        \caption{Input data $\vec{x}$.}
    \end{subfigure}
    \quad
    \begin{subfigure}[b]{0.48\textwidth}
        \centering
        \includegraphics[width=\linewidth]{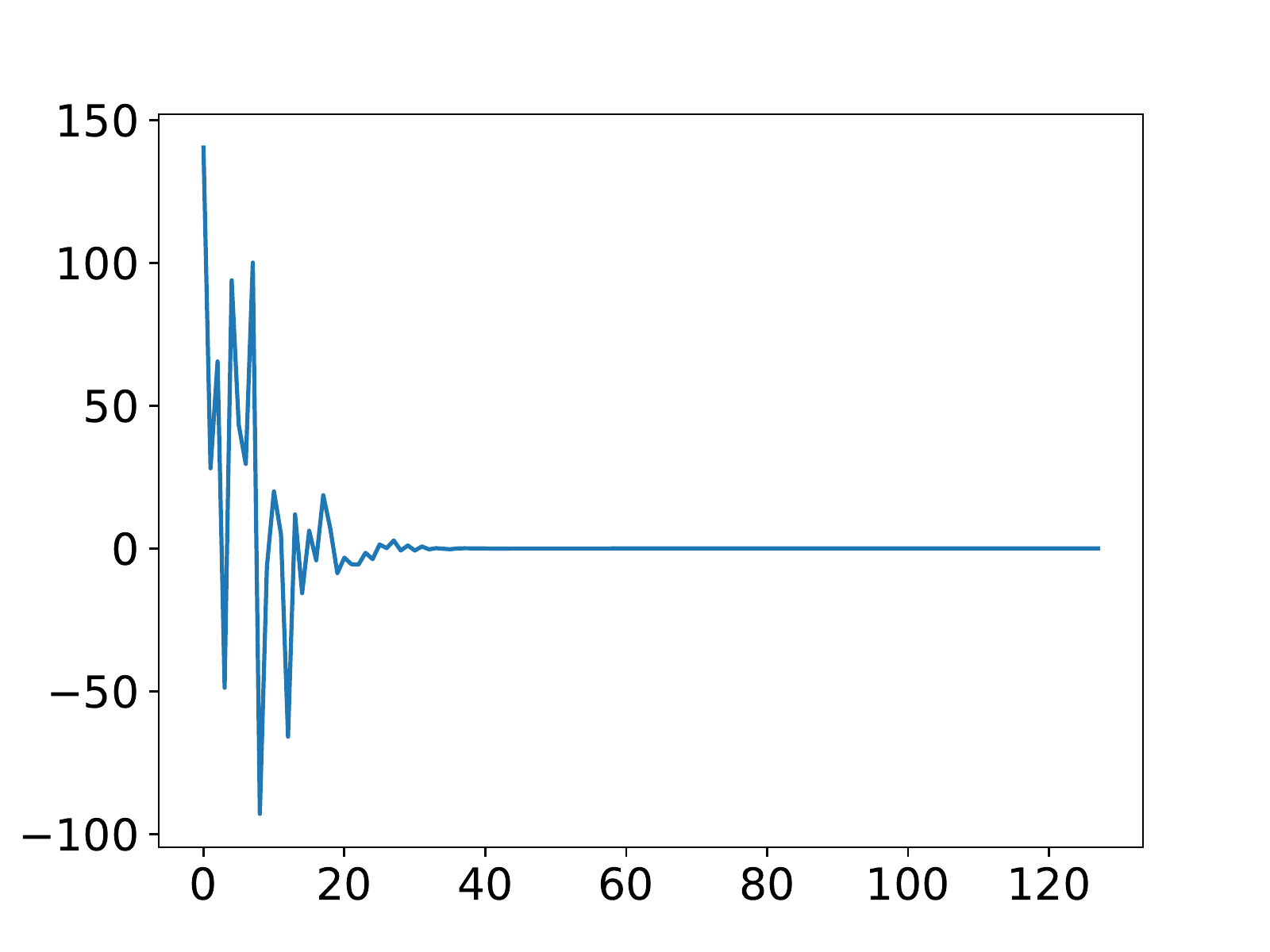}
        \caption{Real part of output data $\vec{y}$.}
    \end{subfigure}
    \caption{One instance in dataset DFTSmooth-Lfreq.}
    \label{fig:data-DFTSmooth-Lfreq}
\end{figure}

\begin{figure}[htp]
    \centering
    \begin{subfigure}[b]{0.48\textwidth}
        \centering
        \includegraphics[width=\linewidth]{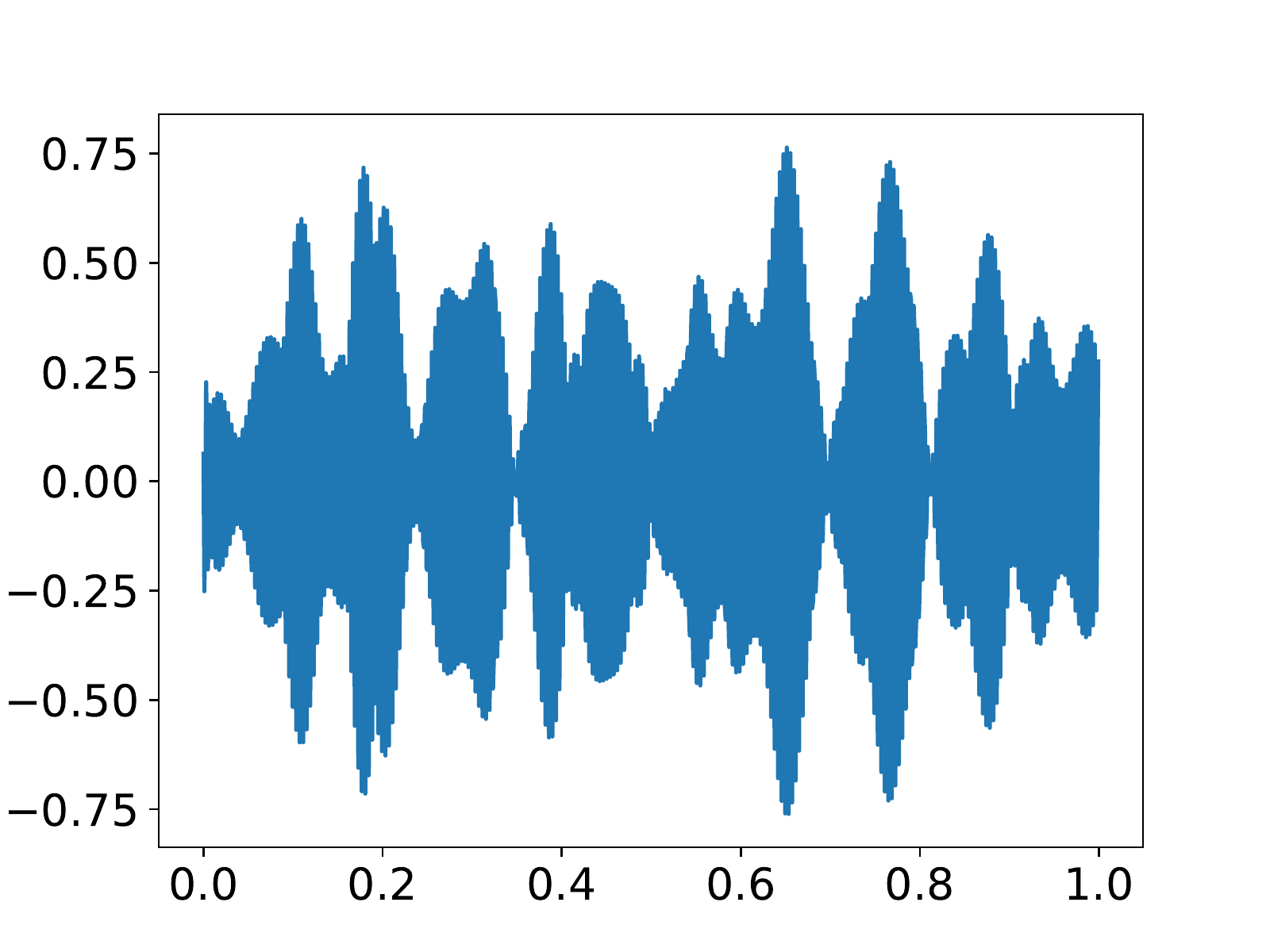}
        \caption{Input data $\vec{x}$.}
    \end{subfigure}
    \quad
    \begin{subfigure}[b]{0.48\textwidth}
        \centering
        \includegraphics[width=\linewidth]{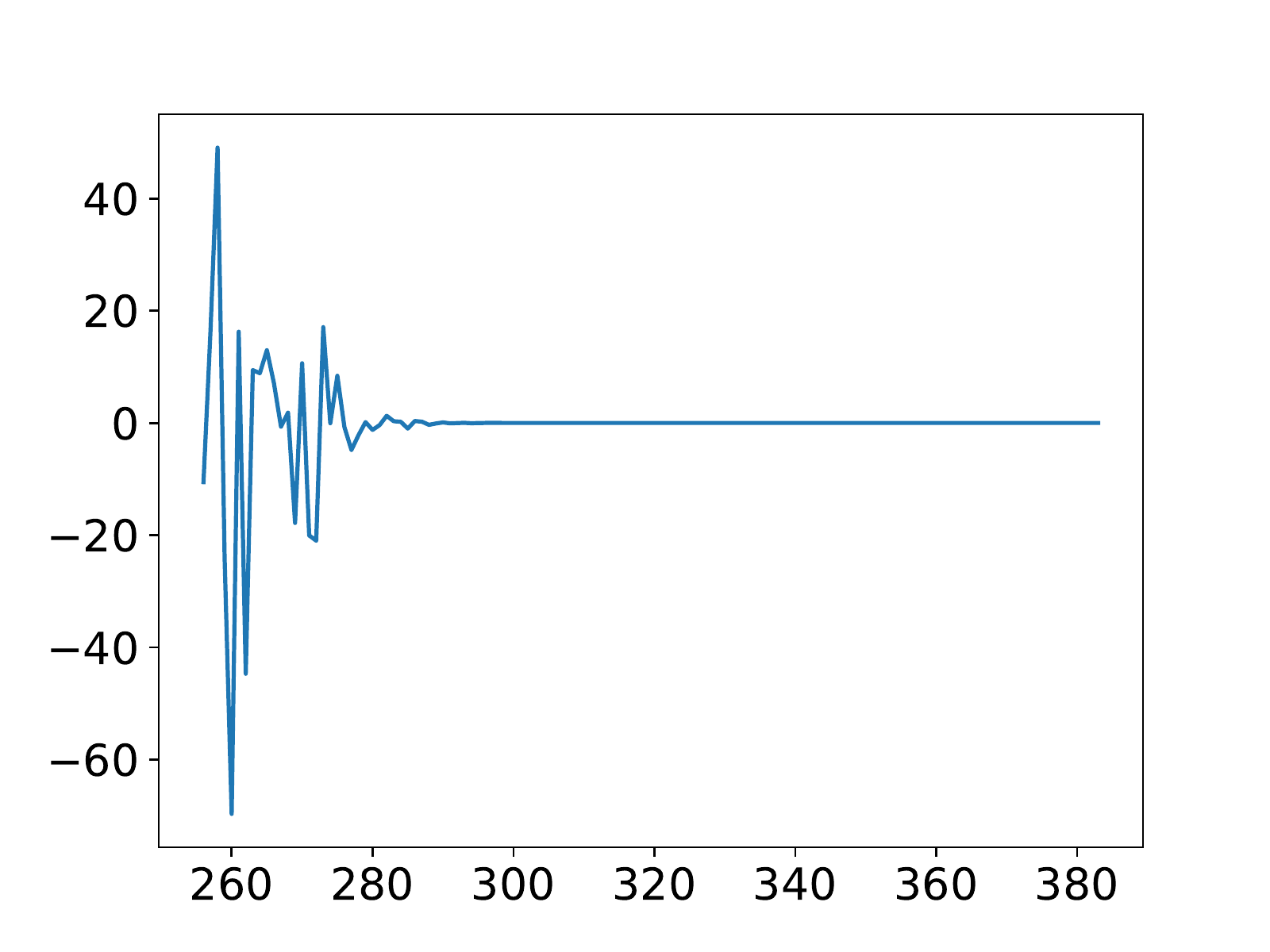}
        \caption{Real part of output data $\vec{y}$.}
    \end{subfigure}
    \caption{One instance in dataset DFTSmooth-Hfreq.}
    \label{fig:data-DFTSmooth-Hfreq}
\end{figure}

\section{Convergence behaviors}
\label{app:conv-behaviors}

We provide four convergence behaviors for the training
loss of \NetName{} and \INetName{} with random and {\it
Butterfly} initialization applied to DFTSmooth-Lfreq in
Figure~\ref{fig:conv-behaviors}. In all four networks, the number
of layers after switch layer is $L_\xi = 1$. Other settings have the
similar convergence behaviors.

\begin{figure}[htp]
    \centering
    \begin{subfigure}[b]{0.48\textwidth}
        \centering
        \includegraphics[width=\linewidth]{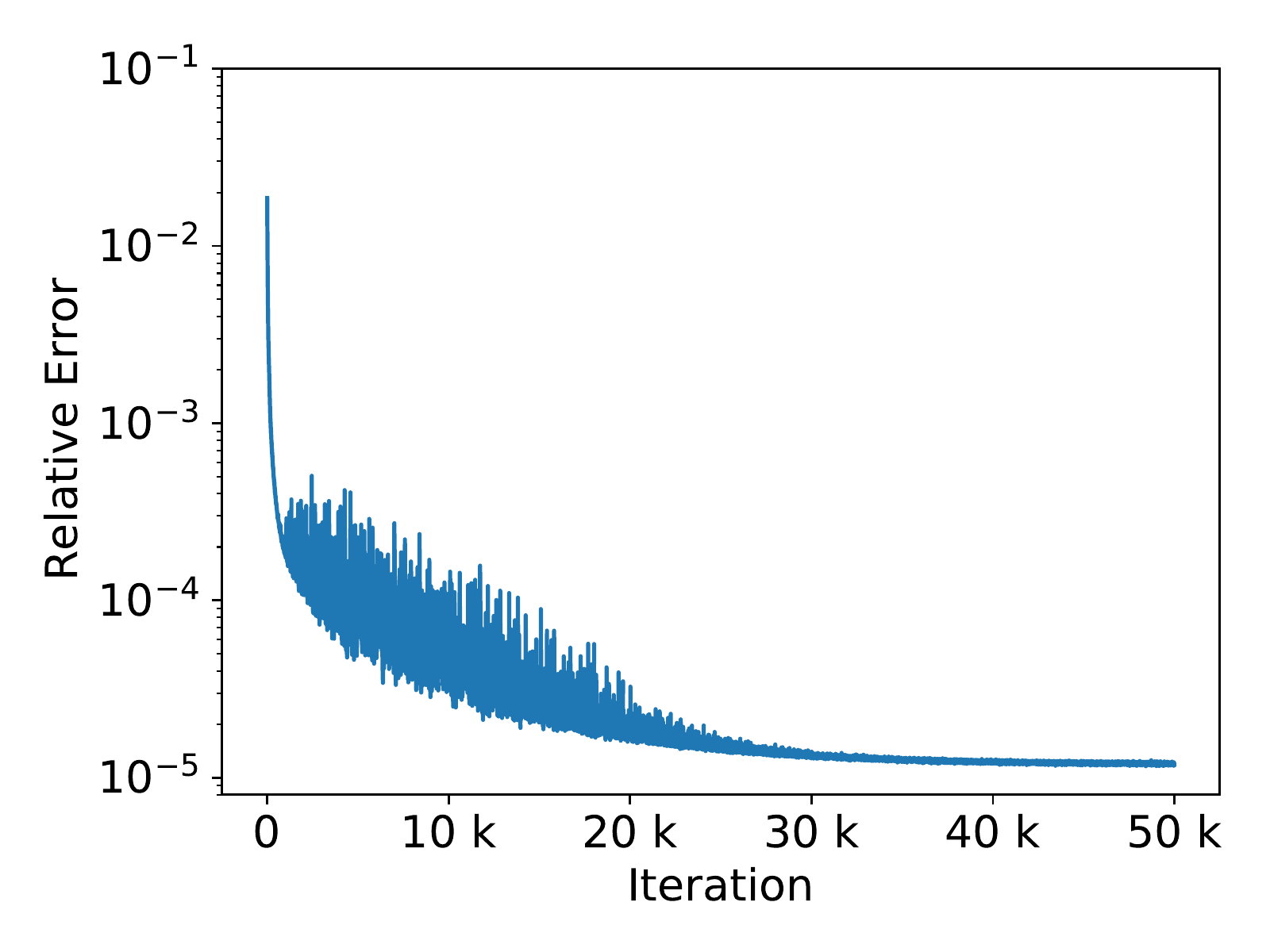}
        \caption{\NetName{} with {\it Butterfly} initialization.}
    \end{subfigure}
    \quad
    \begin{subfigure}[b]{0.48\textwidth}
        \centering
        \includegraphics[width=\linewidth]{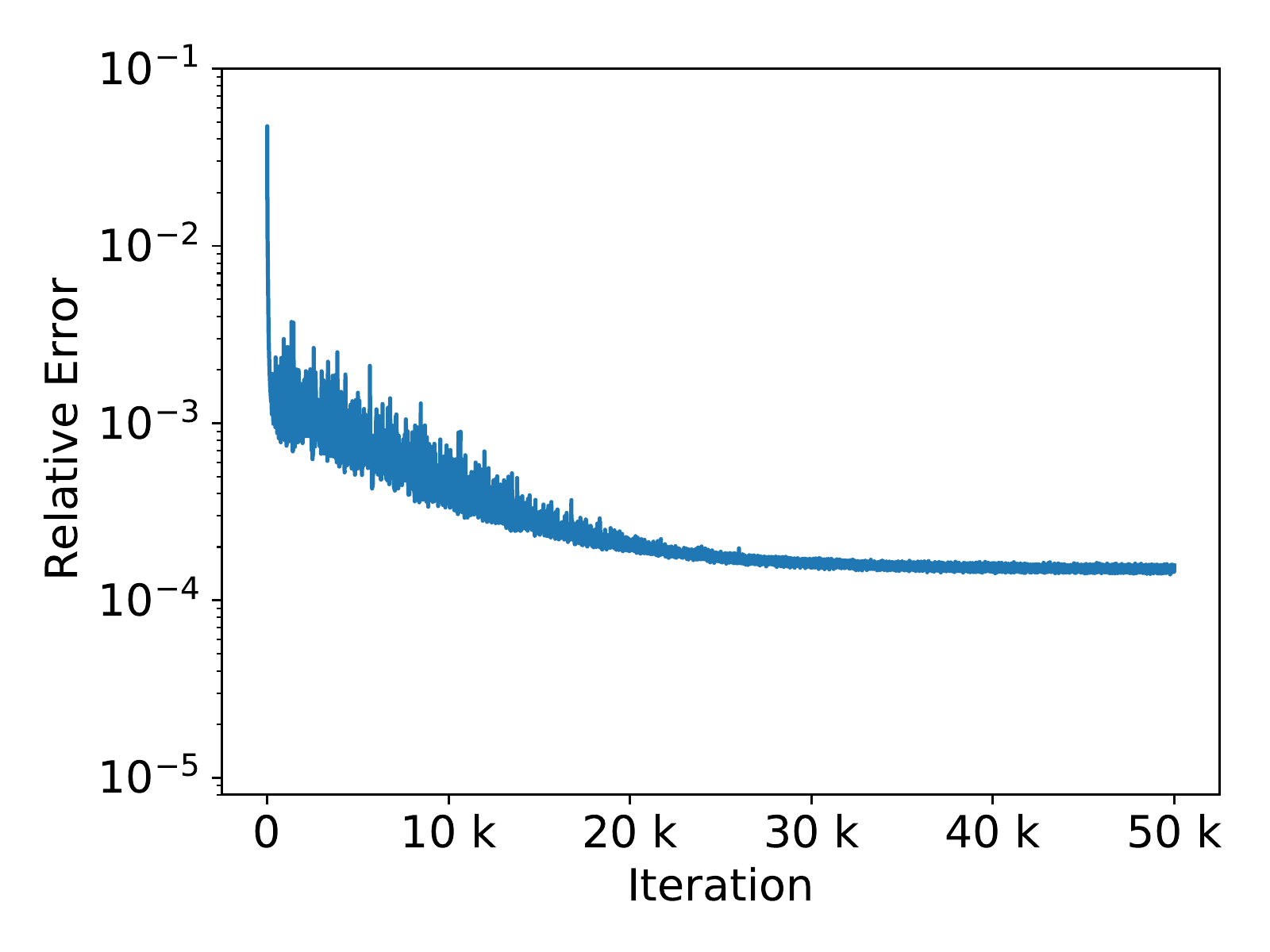}
        \caption{\INetName{} with {\it Butterfly} initialization.}
    \end{subfigure}
    \begin{subfigure}[b]{0.48\textwidth}
        \centering
        \includegraphics[width=\linewidth]{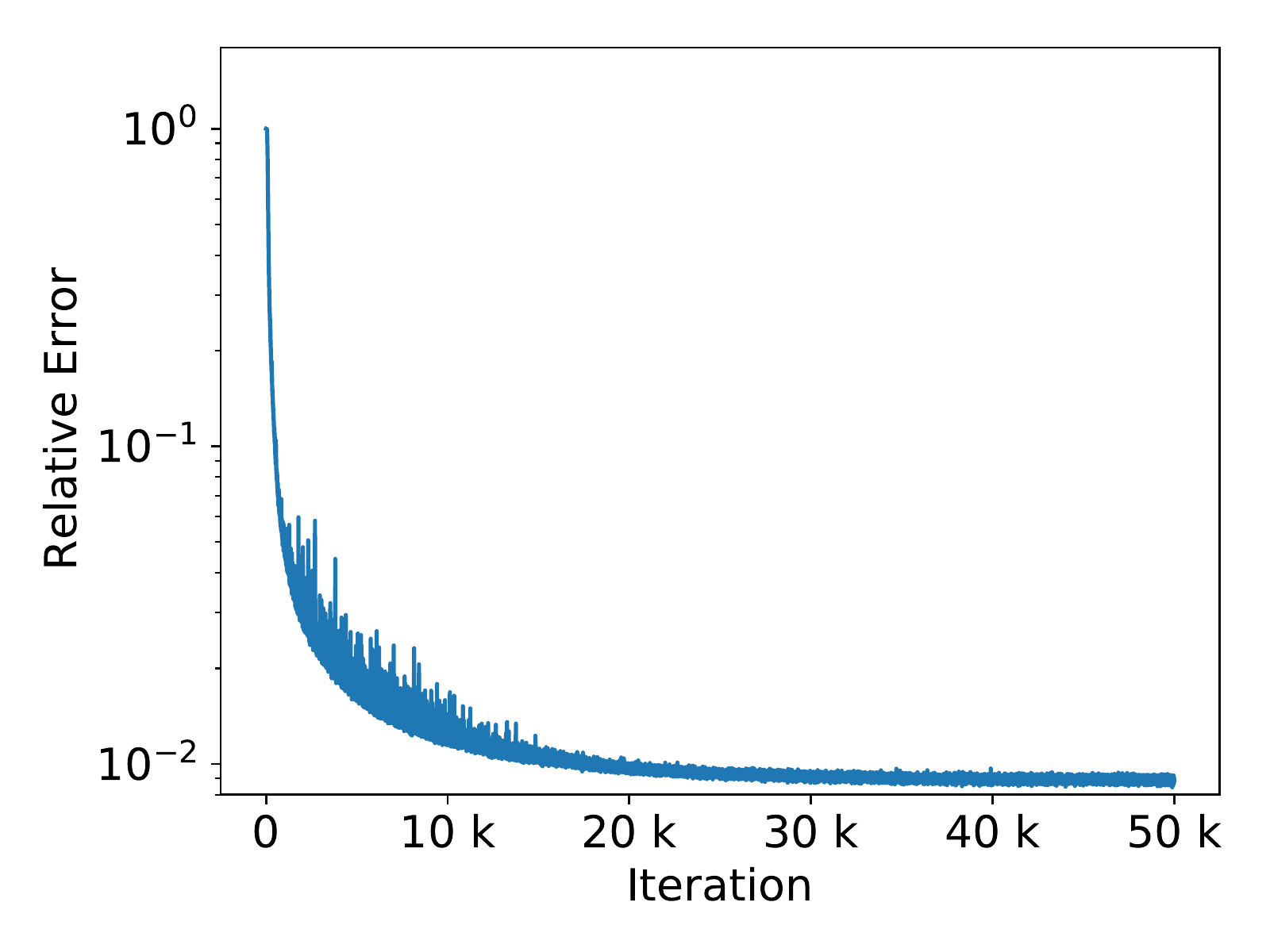}
        \caption{\NetName{} with random initialization.}
    \end{subfigure}
    \quad
    \begin{subfigure}[b]{0.48\textwidth}
        \centering
        \includegraphics[width=\linewidth]{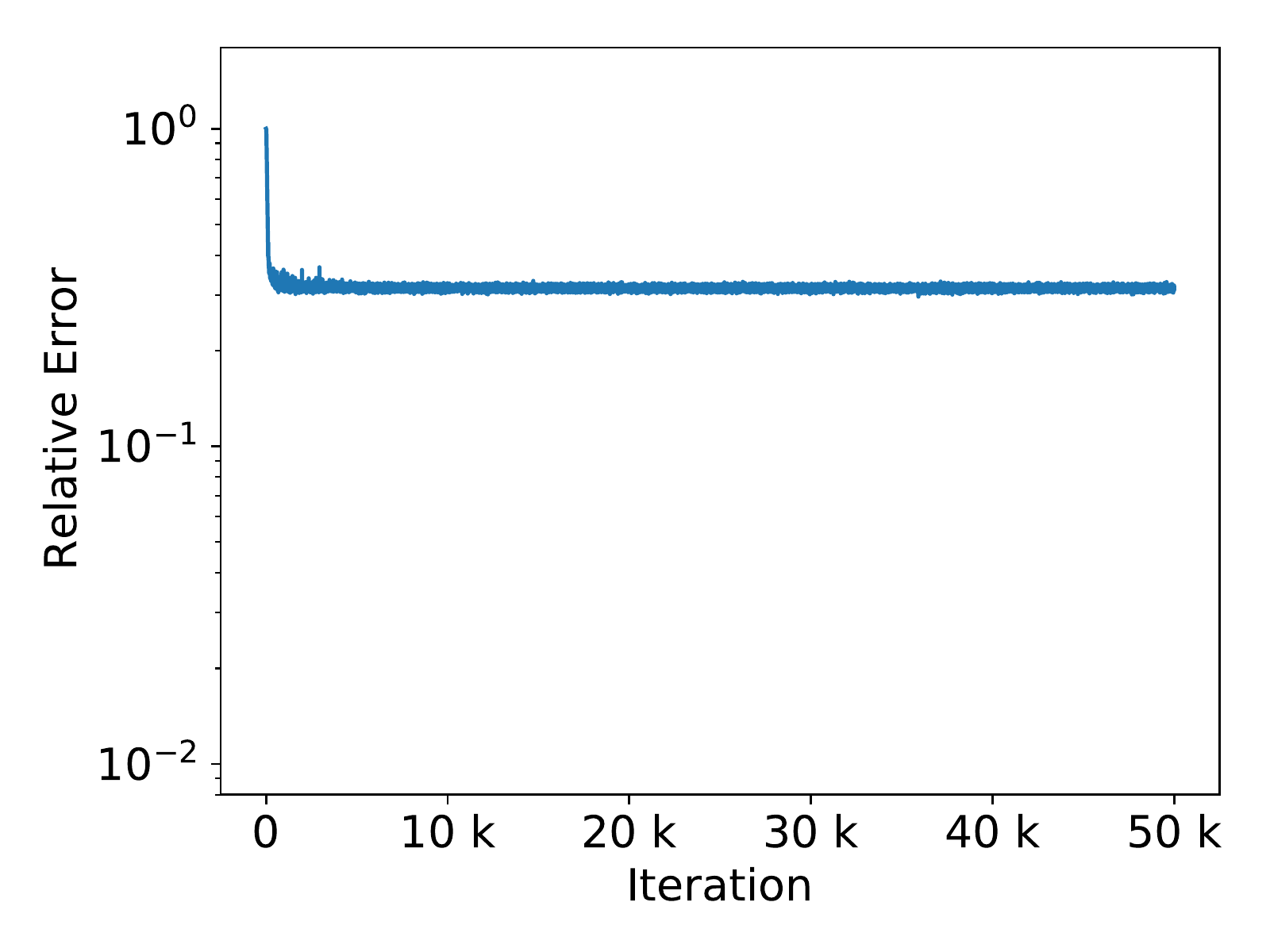}
        \caption{\INetName{} with random initialization.}
    \end{subfigure}
    \caption{Convergence behavior for various networks applied to
    DFTSmooth-Lfreq dataset. The number of layers after switch layer
    for all networks is $L_\xi = 1$.} \label{fig:conv-behaviors}
\end{figure}

\end{document}